\documentclass[12pt]{article}
\usepackage{amssymb}
\usepackage{amsfonts}
\usepackage{epsfig}
\usepackage{latexsym}
\usepackage{graphicx}
\usepackage{tikz}
\usetikzlibrary{calc,arrows}
\usepackage{scalefnt}
\usepackage{caption}
\usepackage{url}
\usepackage{amsmath}
\usepackage[margin=1in]{geometry}
\usepackage{verbatim}
\usepackage{parskip}
\usepackage{pgfplots}
\usepackage{float}
\usepackage{amsthm}
\usepackage{float}
\restylefloat{table}

\newcommand{\bbm}{\begin{boldmath}}
\newcommand{\ebm}{\end{boldmath}}

\newcommand{\lt}{<}
\newcommand{\gt}{>}
\newtheorem{theo}{Theorem}[section]
\newtheorem{cor}{Corollary}[section]
\newtheorem{lem}{Lemma}[section]
\newtheorem{prop}{Proposition}[section]

\newtheorem{defy}{Definition}[section]

\newtheorem{mex}{Example}[section]

\newtheorem{mex2}{Example}[subsection]
\newtheorem{conjecture}{Conjecture}[section]
\newtheorem{remk}{{\bf Remark:}}[section]
   {\begin{remk}\begin{normalshape}\mbox{}}%
   {\hfill  \end{normalshape}\end{remk}}%

\tikzstyle{level 1}=[sibling distance=40mm]
\tikzstyle{level 2}=[sibling distance=30mm]
\tikzstyle{level 3}=[sibling distance=20mm]
\tikzstyle{level 4}=[sibling distance=10mm]

\newcommand{\hmytab}{\hspace*{1.9em}}
\begin{document}
\begin{center}
{\Large {\bf Convergence and divergence testing theory and applications by Integration at a point}} \\
\vspace{0.5cm}
{\bf Chelton D. Evans and William K. Pattinson}
\end{center}
\begin{abstract}
Integration at a point is a new kind of integration derived from
 integration over an interval in infinitesimal and infinity domains 
 which are
 spaces larger than the reals.
%
 Consider a continuous monotonic divergent function that is continually increasing.
 Apply the fundamental theorem of calculus. The integral is a difference
 of the function integrated at the end points. If one of these point integrals is much-greater-than
 the other in magnitude delete it by non-reversible arithmetic.
 We call this type of integration ``convergence sums" because our primary application
 is a theory for the determination of convergence and divergence of
 sums and integrals.
 The theory is far-reaching.
 It reforms known convergence tests and arrangement theorems,
 and it connects integration and series 
 switching between the different forms.
 By separating the finite and infinite domains, the mathematics is more naturally considered,
 and is a problem reduction.
 In this endeavour we rediscover and reform 
 the ``boundary test" which we believe to be the boundary between convergence
 and divergence: the boundary is represented as an infinite class of generalized p-series functions.
 All this is derived from extending du Bois-Reymond's theory with gossamer numbers
 and function comparison algebra.
\end{abstract}

\section*{Introduction}

The papers are an interdisciplinary collaboration. The coauthor is a retired high school
and tertiary Physics/Maths teacher with over 33 years teaching experience.
 I produced the
mathematics,
 considered by myself for over 20 years, and collaborated with William, over
the past four years.

\ref{S01}. \textbf{Convergence sums at infinity with new convergence criteria} 

\ref{S0101}. Introduction \\
\ref{S0102}. What does a sum at infinity mean? \\
\ref{S0103}. Euler's Convergence criteria \\
\ref{S0108}. A reconsidered convergence criteria \\
\ref{S0110}. Reflection \\
\ref{S0111}. Monotonic sums and integrals \\
\ref{S05}. Comparing sums at infinity \\
\ref{S12}. Integral and sum interchange \\
\ref{S09}. Convergence integral testing \\
\ref{S07}. Convergence tests summary \\
\hmytab \ref{S0701} p-series test \\
\hmytab \ref{S0702} Power series tests (Section \ref{S14}) \\
\hmytab \ref{S0703} Integral test \\
\hmytab \ref{S0704} Comparison test \\
\hmytab \ref{S0705} nth term divergence test \\
\hmytab \ref{S0706} Absolute convergence test \\
\hmytab \ref{S0707} Limit Comparison Theorem \\
\hmytab \ref{S0708} Abel's test \\
\hmytab \ref{S0709} L'Hopital's convergence test \\
\hmytab \ref{S0710} Alternating convergence test (Theorem \ref{P206}) \\
\hmytab \ref{S0711} Cauchy condensation test \\
\hmytab \ref{S0712} Ratio test (Section \ref{S17}) \\
\hmytab \ref{S0713} Cauchy's convergence test \\
\hmytab \ref{S0714} Dirichlet's test \\
\hmytab \ref{S0716} Bertrand's test (Section \ref{S17}) \\
\hmytab \ref{S0717} Raabe's tests (Section \ref{S17}) \\
\hmytab \ref{S0718} Generalized p-series test (Section \ref{S1802}) \\
\hmytab \ref{S0719} Generalized ratio test (Section \ref{S17}) \\
\hmytab \ref{S0720} Boundary test (Section \ref{S18}) \\
\hmytab \ref{S0721} nth root test (Theorem \ref{P234}) \\
\ref{S13}. Miscellaneous \\
\hmytab \ref{S1301} Transference between sums and convergence sums \\
\hmytab \ref{S1302} Convergence rates 

\ref{S14}. \textbf{Power series convergence sums} 

\hmytab \ref{S1401} Introduction \\
\hmytab \ref{S1402} Finding the radius of convergence \\
\hmytab \ref{S1403} Briefly differentiation and continuity 

\ref{S15}. \textbf{Convergence sums and the derivative of a sequence at infinity} 

\hmytab \ref{S1501} Introduction \\
\hmytab \ref{S1503} Derivative at infinity \\
\hmytab \ref{S1502} Convergence tests 

\ref{S16}. \textbf{Rearrangements of convergence sums} 

\hmytab \ref{S1601} Introduction \\
\hmytab \ref{S1602} Periodic sums \\
\hmytab \ref{S1603} Tests for convergence sums 

\ref{S17}. \textbf{Ratio test and a generalization with convergence sums} 

\hmytab \ref{S1701} Introduction \\
\hmytab \ref{S1702} The ratio test and variations \\
\hmytab \ref{S1703} A Generalized test 

\ref{S18}. \textbf{The Boundary test for positive series}

\hmytab \ref{S1801}. Introduction \\
\hmytab \ref{S1802}. Generalized p-series \\
\hmytab \ref{S1803}. The existence of the boundary and tests \\
\hmytab \ref{S1804}. The boundary test Examples \\
\hmytab \ref{S1805}. Convergence tests \\
\hmytab \ref{S1806}. Representing convergent/divergent series 
\section{Convergence sums at infinity with new convergence criteria} \label{S01}
 Development of sum and integral convergence criteria, leading to
 a representation of the sum or integral as a point at infinity.
 Application of du Bois-Reymond's comparison of functions theory, when
 it was thought that there were none.
 Known convergence tests are alternatively stated and some are reformed.
 Several new convergence tests are developed,
 including an adaption of L'Hopital's rule.
 The most general, the boundary test is stated.  
 Thereby we give an overview of a new field we call `Convergence sums'. 
 A convergence sum is essentially a strictly monotonic sum or integral where one of the
 end points after integrating is deleted resulting in a sum or integral at a point.

\subsection{Introduction}\label{S0101}

 Before attempting to evaluate a sum or integral, we need to
 know if we \textit{can} evaluate it. 
 The aim of this paper is to introduce a new field of mathematics
 concerning sum convergence, with an `aerial' view of the field, 
 for the purpose of convincing the reader of its existence and 
 extensive utility. 
 In general,
 we are concerned with positive series.

 We are building the theory of this new mathematics
 from the foundations of our previous papers,
 and by extending du Bois-Reymond's Infinitesimal and Infinitary Calculus 
 (\cite{cebp21}, \cite{cebp10}).

 The introduction of the `gossamer' number system \cite[Part 1]{cebp21}
 reasons in, and defines infinitesimals and infinities,
 and is comparable with the hyper reals, but more user friendly.

 Paper \cite[Part 3]{cebp21} introduces an algebra
 that results from the comparison of functions,
 which is instrumental in the following development of the
 convergence tests. 

 The goal is to reproduce known tests and new tests with this theory.
 Advantage is taken of known tests,
 where a parallel test is constructed
 from the known test.
 However in some cases the new test is quite different,
 with an extended problem range and usage.
 The p-series, by application of non-reversible arithmetic
 at infinity, is such a test.
 
The tests listed above marked with * are referenced but are considered in greater detail
 in subsequent papers. Other ideas such as rearrangements (Section \ref{S16}),
 derivatives (Section \ref{S15}), and applications with infinite products
 are also discussed in subsequent papers. 
 
 In 
 \cite[Part 3]{cebp21} a method of comparing functions, 
 generally at zero or infinity, is developed. This is the core idea
 for developing the convergence criterion, and subsequent tests.

 However, even in this paper we will have to request faith,
 as not until the consideration of convergence or divergence
 at the boundary (Section \ref{S18}) 
 can we understand why this maths works.
 For this we need du Bois-Reymond's infinitary calculus
 and relations. 

 To this point it has generally been believed,
 even by advocates such as G. H. Hardy, that du Bois Reymond's infinitary
 calculus has little application.  We believe,
 with the discovery of sums at infinity,
 which we call `convergence sums',
 this view may be overturned. 

With over twenty tests, and the application of an
 infinitesimal and infinity number system, the infinireals,
 we believe this recognizes the infinitary calculus,
 produced by 
 both du Bois-Reymond and
 Hardy.

Hardy himself did not believe du Bois-Reymond's
 theory to be 
 of major mathematical significance.
 However, if our work changes this belief,
 both du Bois-Reymond and Hardy's path and intuition
 can be justified. 
 
 In \cite[Part 1]{cebp21} we develop the infinireals and `gossamer' number
 system $*G$, which is constructed from infinite integers. 

  When reasoning with infinitesimals
 and infinities, often simpler or
 more direct constructions are
 possible. We are motivated
 to seek this both for theoretical
 and practical calculations.

 We have argued in our number system that infinitesimals
 and infinities are numbers.

 In the gossamer number system,
  $\Phi$ are infinitesimals,
 $\Phi^{-1}$ are infinities
 and 
 $\mathbb{R}_{\infty}$ are `Infinireals' which are either
 infinitesimals or infinities.

 Further, the numbers in $*G$
 have an explicit number type
 such as integers, rational numbers, infinite integers, infinite transcendental numbers,
 etc.

 Infinitary calculus is
 a non-standard analysis,
 which
 we see as complementing and
  replacing
 standard analysis and other non-standard
 analyses, where applicable.
\subsection{What does a sum at infinity mean?}\label{S0102}
Zeno's paradoxes provide excellent reasons for us to accept infinity,
 as we need to consider a sum of infinite terms to obtain
 a finite result. The following argument from our
 readings on Zeno, demonstrates that 
 a partial sum is subject to Zeno's paradoxes.

\begin{quote}
 Consider an arrow in flight. 
 After travelling half the distance,
 half the distance remains.
 Repeating this, after each
 repetition, since half the distance always
 remains, the arrow never reaches its target.
\end{quote}

 The remaining distance $\frac{1}{2^{n}}|_{n=\infty}$
 is of course a positive infinitesimal,
 smaller than any positive real number.
 With these kinds of problems and the discovery of infinitesimal
 calculus, mathematical knowledge exponentially increases.

 The problem was the infinite sum:
 $\frac{1}{2} + \frac{1}{4} + \frac{1}{8} \ldots =1$.
 By expressing $y = \frac{1}{2} + \frac{1}{4} + \ldots$
 as an infinite sum, 
 $2y = 1 + \frac{1}{2} + \frac{1}{4} + \ldots = 1 + y$,
 the infinitesimal is side-stepped (the infinite series are subtracted)
 and $y$ is solved for $y=1$,
 and the arrow hits the target.
 Alternatively, express as
 a partial sum,
 $y_{n} = \frac{1}{2} + \frac{1}{4} + \ldots + \frac{1}{2^{n}}$, then
 $y_{n} = 1 - \frac{1}{2^{n}}$,
 $\frac{1}{2^{n}}|_{n=\infty} \in \Phi$, 
 $y_{n}|_{n=\infty}=1$. 

In fact, Zeno's paradoxes prove the existence of the infinitesimal,
 for such a number can always be constructed.
 With the arrow striking the target,
 an application of the transfer principle where the infinitesimal
 is set to zero($(*G,\Phi) \mapsto (\mathbb{R},0)$). We see 
 from the two number systems $*G$ and $\mathbb{R}$,
 that in $*G$ the arrow has no collision with the target, but always approaches
 infinitesimally close to the target.
 The realization of the sum is by the transfer principle, transferring
 to a collision in $\mathbb{R}$.
\[ \sum_{k=0}^{\infty} a_{k} = a_{0} + a_{1} + a_{2} + \ldots \]
Typically such a sum is described as
 converging when the sum of the remaining terms tends to zero.
 $\sum_{k=0}^{\infty} a_{k} = \sum_{k=0}^{n-1} a_{k} + r_{n}$
 where $r_{n} = \sum_{k=n}^{\infty} a_{k}$ and for convergence
 $r_{n} \to 0$ as $n \to \infty$. 
 As this partial sum includes finite and infinite
 terms,
 we will need to construct
 a different partial sum, separating
 finite and infinite terms, 
 as convergence/divergence is considered
 for the infinite terms only.

 Since a sum is a sequence of partial sums, even at infinity,
 then the convergence of the sum is reduced
 to sequence theory and the convergence
 of the sequence.
 Let $s_{n} = \sum_{k=1}^{n} a_{n}$, then does $( s_{n})|_{n=\infty}$ converge?

We find that all sum convergence or divergence
 is determined at infinity \cite[Corollary 6.1]{cebp21}, except
 if the sum diverges by summing
 a prior singularity. This is independent
 of the sum convergence criteria.

 We apply this reasoning to sums.

 Consider a sum of positive terms $\sum_{k=1}^{\infty} a_{k}$.
 Let $a_{k}$ be finite so the series
 contains no singularities for an arbitrary number of finite terms 
(for example
 none of the terms divide by zero).
 Then such a sum,
  if it diverges, can only diverge
 at infinity, as a sum with a finite
 number of terms is always convergent. 

 Since a convergent series 
 is the negation of a divergent series,  
 such that all series are classified
 as either convergent or divergent,
 then this is determined at
 the singularity $n=\infty$.
\bigskip
\begin{defy}\label{DEF006}
We say $\sum_{k=n_{0}}^{n_{1}} a_{k}$ is a
 `convergence sum' at infinity. The domain
 of the sum is not finite; $\, n_{0}, n_{1} \in \mathbb{N}_{\infty}$; 
\end{defy}
We define iterating over a sum and integral with number types.
\bigskip
\begin{defy}
  \[ \sum_{\mathbb{N}_{\lt}} a_{k} = a_{1} + a_{2} + a_{k}\ldots \text{ where } k \text{ is finite.} \]
  \[ \sum_{\mathbb{N}_{\infty}} a_{k} = \ldots, a_{n-1}+ a_{n} + a_{n+1} + \ldots \text{ for } n \in \mathbb{N}_{\infty} \]
\end{defy}
\bigskip
\begin{defy}
$\int_{\mathrm{min}(+\Phi^{-1})}^{\mathrm{max}(+\Phi^{-1})} a(x)\,dx =  \int_{+\Phi^{-1}}a(x)\,dx = \int a(x)\,dx|_{+\Phi^{-1}}$
\end{defy}

With the existence of infinite integers $\mathbb{J}_{\infty}$,
 integer sequences can be partitioned into finite and infinite
 parts,
$(a_{1}, a_{2}, \ldots ) = (a_{1}, a_{2}, \ldots, a_{k})|_{k \lt \infty} + (\ldots, a_{n}, a_{n+1}, \ldots)|_{n=\infty}$ \cite[Definition 6.4]{cebp21}. 
 Consequently
 sequences converge
 at infinity,
 $a_{m} - a_{n} |_{\forall m,n = \infty} \simeq 0$ \cite[Definition 2.14]{cebp21} ,
 and since a sum is a sequence, the sum
 can now be deconstructed into
 finite and infinite parts.
\bigskip
\begin{theo}\label{P028}
 $\sum_{j=1}^{\infty} a_{j}$
 $= \sum_{j=1}^{j \lt \infty}a_{j} + \sum_{\mathbb{N}_{\infty}}a_{j}$
 $= \sum_{\mathbb{N}_{\lt}}a_{j} + \sum_{\mathbb{N}_{\infty}}a_{j}$
\end{theo}
\begin{proof}
By the existence of infinite integers,
 with infinireals, $\mathbb{N}_{\lt} \lt \mathbb{N}_{\infty}$, 
 we can extend an integer sequence to infinity.

 $(1, 2, 3, \ldots)$
 $=(\mathbb{N}_{\lt}) + (\mathbb{N}_{\infty})$,
 then
 $\sum_{j=1}^{\infty} a_{j}$
 $=\sum_{\mathbb{N}_{\lt}}a_{j} + \sum_{\mathbb{N}_{\infty}}a_{j}$.
\end{proof}
\bigskip
\begin{defy}\label{DEF007}
We say 
 $\int_{x_{0}}^{x_{1}} a(x)\,dx|_{x=\infty}$ is a
 `convergence integral' at infinity. The domain
 of the integral is not finite; $\,x_{0}, x_{1} \in +\Phi^{-1}$;
\end{defy}
\bigskip
\begin{theo}\label{P052}
 $\int_{\alpha}^{\infty} a(x)\,dx = \int_{\alpha}^{x \lt \infty} a(x)\,dx + \int_{+\Phi^{-1}} a(x)\,dx$
\end{theo}
\begin{proof}
By the existence of infinite real numbers,
 $\mathbb{R}^{+}\! \lt \!+\Phi^{-1}$, partition the integral
 on the domain between the finite and infinite integrals.
\end{proof}
\bigskip
\begin{theo}\label{P001}
If given the sum $\sum_{k=0}^{\infty} a_{k}$, $a_{k}$ has no singularities
 for finite $k$($k \lt \infty$),
 then convergence or divergence of the sum is 
 determined at infinity, $\sum_{\mathbb{N}_{\infty}} a_{k}$.
\end{theo}
\begin{proof}
Since every finite sum is convergent, only
 an infinite sum is divergent. 
 As the negation of a divergent 
 series determines
 convergence, only at infinity can
 an infinite series be 
 determined to be convergent or divergent.
\end{proof}
\bigskip
\begin{theo}\label{P016}
If given the integral $\int_{x_{0}}^{\infty} a(x)\,dx$, $a(x)$ has
 no singularities for finite $x$($x \lt \infty$),
 then convergence or divergence of the integral is 
 determined at infinity, $\int a(x)\,dx|_{+\Phi^{-1}}$.
\end{theo}
\begin{proof}
Since every definite integral without singularities is convergent,
 only at infinity can the integral diverge.
 Hence the determination of convergence or divergence occurs at infinity.
\end{proof}

Theorem \ref{P001} of course does not say what is happening at infinity,
 nor does it say how to use this fact; 
 for 
 this we need a convergence criterion. 
 The development of a criterion is  
 necessary for
 defining a sum,
 because the sums concerned have an infinite
 number of terms.

This more general view of sum convergence makes more sense
 when
 we view other sum criteria.
 An exotic example is in string theory \cite{stringtheory} where
 $\sum_{k=1}^{n} k|_{n=\infty} = -\frac{1}{12}$.
  Here a sum is defined at
 infinity by summing shifted sequences,
 the finite sum is meaningless,
 as you need to consider an infinity of terms.

All this follows from \cite[Part 6 Sequences]{cebp21}
 where an alternate
 definition for convergence of
 a sequence at infinity is defined,
 and since a sum is a type of sequence,
 the application to the definitions
 of sums at infinity follows.

The other aspect of sums are divergent sums, sums 
 which do not converge.
 So any function which does not
 converge by definition diverges,
 and we say of such a sum that it
 is equal to infinity,
  $\sum a_{n}|_{n=\infty}=\infty$. 
\bigskip
\begin{defy}\label{DEF008}
If $\sum a_{n}$ diverges at infinity we say
 $\sum a_{n}|_{n=\infty}=\infty$ 
\end{defy}

 Through logic and language,
 diverging sums are divided into two cases,
 (i) where the sum does not continually grow in
 magnitude, $1 - 1 + 1 - \ldots$ being
 an example, and (ii)
 where the sum continually grows in magnitude,
 as in  $1 + 4 + 8 + 16 + \ldots$
 diverging to
 infinity.
 We could classify divergence as 
 either the divergent sum is monotonically
 increasing, or it is not.

 The characterisation of the sum
 will also depend on the inner term $a_{n}$ being
 summed. If we say a sum is monotonic we will 
 mean 
 the sum's sequence is monotonic.
 Of course we can have a sum
 increasing for both monotonically
 increasing and monotonically decreasing
 terms.

The sums at infinity in this paper
 are concerned with monotonic
 functions that do not plateau,
 which for diverging
 sums would be the class that continually
 grows.

 This does not deny the many possible applications
 of non-monotonic series, but instead we
 are concerned with monotonic series as an input
 with our tests.

Monotonic functions,
 because of their
 guaranteed behaviour are really useful
 and the subject of much of infinitary
 calculus. The scales of infinities are 
 examples. 

 In our culture, let us ask the question,
 are there different infinities? The idea that
 like different numbers,
 we can have different infinities
 is most likely unrecognised.  
 So $x^{2}|_{x=\infty}$ and $x^{3}|_{x=\infty}$
 are being seen as $\infty$.  While this generalisation is useful,
 the use of infinitary calculus or little-o/big-O notation
 is not as readily recognised; yet 
 it becomes advantageous to
 treat different infinities as different numbers.

Interpreting what a sum is generally boils
 down to interpreting the little dots,
 that is, saying
 what happens at infinity.

 Euler, Hardy, 
 Ramanujuan and many
 others have pursued 
 and found applications 
 in defining
 a sum at infinity.
 
 In this paper another
  criterion is developed and compared with a 
 non-standard
 analysis convergence criterion.

The new criterion separates finite and infinite
 arithmetic as other criteria have done.
 After this separation,
  the infinite part of the sum 
 is considered as a point at infinity.
 At 
 this point we use infinitary
 calculus as the mechanism 
 to do calculations.

 To help explain why this is interesting,
 consider a divergent sum
 $1 + 1 + \ldots$, 
 if we express  
 the sum by $\sum 1 |_{n=\infty} = \infty$ 
 thus diverges at infinity.
 That is,
 we are in an infinite loop where one is being 
 continually added. 
 We could further describe
 the sum by the divergent function,
 $\sum 1 = n|_{n=\infty}$
\bigskip
\begin{defy}\label{DEF002}
Given a function $f(n)$,  
 let $\sum f(n)|_{n=\infty} = g(n)$ be interpreted
 as a function at infinity.
\end{defy}
\bigskip
If we construct $g(n)$ counting from a reference
 point then
 $\sum_{k=a}^{b} f(n)$
 $= \sum^{b} f(n) - \sum^{a} f(n)$.
\bigskip
\begin{defy}\label{DEF003}
$\sum a_{n}|_{n=\infty}$, 
$\,\int a(n)\,dn \in *G$;
\end{defy}

Infinitary calculus can then be applied to the sum
 at infinity. 
 Looking at convergent series
 we can ask what happens at infinity.
 In the same way a stone thrown into
 a still pond generates a ripple,
  and 
 over time the ripples subside again leaving
 the still pond,
 so the 
 idea of a steady state for
 a sum
 is to look at the sum at infinity
 and enquire about the sum's behaviour.

If the pond does not settle down,
 but continually vibrates, this too can be considered a steady state.
 A steady state looks at the behaviour of the system after an infinity
 of time.

 From the p-series it is known
 that it is not enough that what
 is being added,
 tends to zero. 
 For example $p=1$ gives 
 $\frac{1}{n}|_{n=\infty}=0$, but
 $\sum_{k=1}^{n} \frac{1}{k}|_{n=\infty} = \infty$ diverges.
 This shows that summing an infinity of infinitesimals
 is not necessarily finite.

 Interestingly, a sum that is convergent will have the terms
 being added,
 and these will no longer have an effect on
 the end state. We can explain this
 by considering the sum in a higher
 dimension, with infinitesimals,
 which when projected (by approximation) back to $\mathbb{R}$,
 can disappear. 

 By the Criterion E2 
 described in the next section,
 $\sum f(n)|_{n=\infty}=0$ is a 
 necessary and sufficient condition for sum
 convergence (provided the sum
 did not diverge before reaching infinity).

 In \cite[Part 5]{cebp21} classes
 of functions $\{ \frac{1}{x^{2}+\pi}, \frac{1}{x^{2}-3x}, \ldots \}|_{x=\infty}$
 could be simplified
 to $\frac{1}{x^{2}}|_{x=\infty}$ by arguments of
 magnitude.
 The same simplifications with care
 can be applied to sums, reducing
 classes of sums to particular
 cases. Indeed this is described
 later by the p-series test at infinity.
\bigskip
\begin{mex}\label{MEX001}
For $\sum_{n=1}^{\infty} \frac{1}{n^{2}-3n}$, 
 convergence or divergence
 can be determined by
 considering the sum at infinity,
 $\sum \frac{1}{n^{2}-3n}|_{n=\infty}$
$=\sum \frac{1}{n^{2}}|_{n=\infty}$ 
$=0$ is convergent by comparison
 with known p-series, with
 $p \gt 1$ known to converge. $n^{2}-3n = n^{2}|_{n=\infty}$ as $n^{2} \succ 3n|_{n=\infty}$.
\end{mex}

The following scale can be applied to solving
 equations of the form $a+b=a$ (non-reversible arithmetic \cite[Part 5]{cebp21}) when $a \neq 0$,
 and will be useful when solving sums.
\[ (c \prec \mathrm{ln}(x) \prec x^{p}|_{p \gt 0}  \prec a^{x}|_{a \gt 1} \prec x! \prec x^{x} )|_{x=\infty}\, \cite[\text{Part 2}]{cebp21} \]
\subsection{Euler's Convergence criteria}\label{S0103}
Euler on the nature of series convergence \cite{eulerho}:
\begin{quote}
Series with a finite sum when infinitely continued,
 do not increase this sum
 even if continued to the double of its terms.
 The quantity which is increased
 behind an infinity of terms
 actually remains infinitely small.
 If this were not the case, the sum of the series would not be determined and, 
 consequently, would not be finite.  
\end{quote}

Laugwitz reasons from Euler's criterion
 \cite[p.14]{laugwitz}:
\begin{quote}
A series (of real numbers) has a finite sum if the values of the sum
 between infinitely large numbers is an infinitesimal.
\end{quote}
Reference \cite[p.212]{gold} 
 defined
 a convergence Criterion E1, 
 as a reformation of Euler's criterion
 in
 A. Robinson's non-standard analysis.
 Consider the tail of a sum, and a countable infinity section.

\textbf{Criterion E1.}
\textit{
 The series with general term $a_{k}$,
 where $a_{k} \geq 0$, is convergent(has a finite sum)
 if $\sum_{k=\omega}^{2\omega} a_{k}$ is an infinitesimal
 for any infinitely large $\omega$.
}

We have formed other criteria through
 minor variations, and then considered infinity as a point.
 The Criteria can be implemented 
 by using an extension of du Bois-Reymond's
 infinitary calculus.
 Whereby we can define the sum at infinity
 as a function.
\bigskip
\begin{defy}\label{DEF004}
For Criteria E2, E3, with 
 monotonic sequence $(a_{n})|_{n=\infty}$, $n \in \mathbb{N}_{\infty}$,
\[ \sum a_{n}|_{n=\infty} \in \mathbb{R}_{\infty} \text{ is an infinireal.} \]
\end{defy}

\begin{prop}\label{P049}
 $\sum a_{n}|_{n=\infty} \mapsto \{ 0, \infty \}$
\[ \text{If } \sum a_{n}|_{n=\infty} = \mathbb{R}_{\infty} \text{ then either } \sum a_{n}|_{n=\infty}=0 \text{ converges or }
 \sum a_{n}|_{n=\infty}=\infty \text{ diverges.} \]
\end{prop}
\begin{proof}
 When the infinitesimals and 
 infinities $\mathbb{R}_{\infty}$ are realized,
 $\Phi \mapsto 0$ and $\Phi^{-1} \mapsto \infty$,
 then these are the only two possible values of the sum at infinity.
 If the sum is undefined, by definition
 the sum is said to diverge and assigned $\infty$.
\end{proof}

\textbf{Criterion E2.}
\textit{
 The series with general term $a_{k}$,
 where $a_{k} \geq 0$, is
 convergent (has a finite sum) if
 and only if $\sum a_{n}|_{n=\infty} \in \Phi$,
 or
 else the series is divergent
 and $\sum a_{n}|_{n=\infty}\in \Phi^{-1}$ or $\infty$.
}

\textit{
The integral version of Criterion E2, 
 either  
 $\int^{n} a(n)\,dn|_{n=\infty} \in \Phi$ 
 converges. Else  
 the integral 
 $\int^{n} a(n)\,dn|_{n=\infty} \in \Phi^{-1}$ 
 or $\infty$ and 
 diverges.
}

\textit{For convergence, $(a)|_{n=\infty}$ must be a monotonic sequence.}

The requirement that the series has a sequence
 of monotonic terms for convergence will later be 
 overcome by converting the  
 sequences to an auxiliary monotonic
 sequence for testing.
 
In developing a criterion, the 
 convergence Criterion E2 is also 
 justified from Euler's considerations.
 The
 same quoted
 criterion is realised
 with Criterion E2. 
 Hence we refer to Criteria E1, E2, and E3 in the following section
 as Euler's convergence criteria.
 
In comparing convergence Criteria E1 and E2,
 the non-standard analysis Criterion E1 compares in a sense
 with an interval between two infinities, 
 whereas the Criterion E2 compares at a point, infinity.

However, Criteria E1 and E2 can be compared.
 Replacing the infinitesimals in Criterion E1 with
 zero($\Phi \mapsto 0$) then
 Criterion E1 gives a similar convergence with  Criterion E2.
 For convergence, both sums are zero.
\[ \text{Criterion E1: } \sum_{k=\omega}^{2\omega} a_{k} =0 \;\;\;\;  
 \text{Criterion E2: } \sum a_{n}|_{n=\infty}=0 \]

\begin{mex}\label{MEX002}
Showing that the harmonic 
 series diverges by Euler convergence Criterion E1.
 $\sum_{k=\Omega+1}^{2 \Omega} \frac{1}{k}$
 $\approx \sum_{k=1}^{2 \Omega} \frac{1}{k} - \sum_{k=1}^{\Omega} \frac{1}{k}$
 $\approx \mathrm{ln}\,2\Omega-\mathrm{ln}\,\Omega$
 $\approx \mathrm{ln}\,2 \neq 0$ hence $\sum \frac{1}{k}$ diverges.
\end{mex}
\bigskip
\begin{mex}\label{MEX003}
Showing that the harmonic 
 series diverges by Euler convergence Criterion E2.
 $\sum \frac{1}{n}|_{n=\infty}$
 $=\int \frac{1}{n}\,dn|_{n=\infty}$
 $=\mathrm{ln}\,n|_{n=\infty} = \infty$ diverges.
\end{mex}
\subsection{A reconsidered convergence criteria}\label{S0108}
An integral or a sum as a point may seem shocking, however there
 turns out to be a simple explanation. By the fundamental 
 theorem of calculus, an integral can be expressed as a difference
 in two integrals at a point. If one of these integrals has
 a much greater than magnitude, we can apply non-reversible
 arithmetic $a_{n} + b_{n}|_{n=\infty}=a_{n}$ when $a_{n} \succ b_{n}|_{n=\infty}$
 and only one integral point need be tested.
\[ \int_{a}^{b} f(x)\,dx = \int_{a}f(x)\,dx \text{ or } \int^{b}f(x)\,dx \]

The significance of the reduction to
 evaluation of the integral at a point 
 is to 
 reduce an integral or sum convergence test
 to one, instead of two function evaluations.

Hence we really are evaluating a sum or integral
 at a point.

 As it will be advantageous to convert between
 sums and integrals, we can always thread a continuous
 monotonic function through a monotonic sequence,
 and conversely
 from a monotonic function generate a monotonic sequence.
 At integer values the function and sequence are equal.
 \[ a(n) = a_{n}|_{n \in \mathbb{J}_{\infty}} \]

 Sums and integrals can sandwich each other.
 Either consider the criterion with a sum or
 integral. With the criterion having the same conditions
 for both sums and integrals allows for the integral test
 in both directions.

\textbf{Criteria E3} sum and integral convergence

 The following criteria E3 and E3' are linked,
 hence we refer to both collectively as ``the E3 criteria".

\textbf{E3.0}
Consider an arbitrary infinite interval $[n_{0},n_{1}]$
 which can be grown to meet the conditions.
 $n_{1}-n_{0}=\infty$ is a minimum requirement; $\, n_{0}, n_{1} \in \Phi^{-1}$;

\textbf{E3.1}
 If $\int a(n)\,dn|_{n=\infty}$ cannot
 form a monotonic function, or
 the other E3 conditions fail,
 then the integral diverges, 
 $\int a(n)\,dn|_{n=\infty}=\infty$. \\
\textbf{E3.2}
$\int a(n)\,dn|_{n=\infty} \in \mathbb{R}_{\infty}$; 
 $n \in \Phi^{-1}$ \\
\textbf{E3.3} For divergence $\int a(n)\,dn|_{n=\infty} \in \Phi^{-1}$ \\
\textbf{E3.4} For convergence $\int a(n)\,dn|_{n=\infty} \in \Phi$ \\
\textbf{E3.5} For divergence, $\int a(n)\,dn|_{n=\infty}$
 can be made arbitrarily large. \\
\textbf{E3.6} For convergence, $\int a(n)\,dn|_{n=\infty}$ can be
 made arbitrarily small. 

\textbf{E3'.0}
Consider an arbitrary infinite interval $[n_{0},n_{1}]$
 which can be grown to meet the conditions.
 $n_{1}-n_{0}=\infty$ is a minimum requirement; $\, n_{0}, n_{1} \in \mathbb{J}_{\infty}$; \\
\textbf{E3'.1}
 If $\sum a_{n}|_{n=\infty}$ cannot
 form a monotonic function, or
 the other E3' conditions fail,
 then the sum diverges, 
 $\sum a_{n}|_{n=\infty}=\infty$. \\
\textbf{E3'.2}
$\sum a_{n}|_{n=\infty} \in \mathbb{R}_{\infty}$; 
 $n \in \Phi^{-1}$ \\
\textbf{E3'.3} For divergence $\sum a_{n}|_{n=\infty} \in \Phi^{-1}$ \\
\textbf{E3'.4} For convergence $\sum a_{n}|_{n=\infty} \in \Phi$. \\
\textbf{E3'.5} For divergence, $\sum a_{n}|_{n=\infty}$
 can be made arbitrarily large. \\
\textbf{E3'.6} For convergence, $\sum a_{n}|_{n=\infty}$ can be
 made arbitrarily small.

E3.0-6 satisfy the E2 criterion by adding additional properties
 which justify integration at a point. 
 E3.1 overrides other E3 conditions.

Concerning E3.1: by consideration of 
 arrangements, many classes of non-monotonic
 functions can be rearranged into
 monotonic series for testing.
 (Section \ref{S16})

We do not require a strict monotonic function 
 for the criterion E3
 as a consequence of E3.5 and E3.6 conditions.
 A monotonic series may be
 tested by
 generating a contiguous subsequence
 which is strictly monotonic,
 removing
 equality  (Section \ref{S16}).
 
Concerning E3.2: the NSA Example \ref{MEX002}  
 integrates leaving neither an infinity or
 infinitesimal, then clearly
 such integrals exist.
 Our approach, however excludes this
 case,
 by defining 
 any integral between two infinities
 as an infinireal $\mathbb{R}_{\infty}$.
 This was done 
 to reduce complexity and increase usability.
 If a computation occurs which gives
 the above case ($\sum a_{n}|_{n=\infty} \not\in \mathbb{R}_{\infty}$),
 then an assumption or condition
 has failed.
  However, the theory
 can be extended to include these cases,
 but may be more complicated.
 This results in Example \ref{MEX002}
 alternatively being evaluated
 in Example \ref{MEX003}
 by condition E3.5.
\bigskip
\begin{prop}
If $\sum a_{n} \not \in \mathbb{R}_{\infty}$
 or $\int a(n)\,dn \not \in \mathbb{R}_{\infty}$ 
 then the sum or integral is said to diverge.
\end{prop}
\begin{proof}
 Divergence in the `undefined sense' and not
 a diverging magnitude.
 Conditions E3.3 or E3.4 or E3'.3 or E4'.3 not met.
\end{proof}

Conditions E3.5 and E3.6
 stop the integral from
 plateauing. In $*G$ we could easily
 have a sum of infinitesimals 
 monotonically increasing
 but bounded above.
 These same conditions allow that only
 one integral or sum needs to 
 be tested.

Integration at a point approach makes
 sense when one of the integral evaluations has
 a much greater than magnitude
 than the other. For example, consider
 a diverging integral,  
 $\int^{n} f(x) \,dx \succ \int^{a} f(x)\,dx|_{n=\infty}$
 then $\int^{n}_{a} f(x) \,dx$
 $= \int^{n} f(x)\,dx - \int^{a} f(x)\,dx$
 $= \int^{n} f(x)\,dx$
 $= \int f(n)\,dn|_{n=\infty}$.
 The same order of magnitude
 situation could occur for the infinitely small.

Integrating a single point 
 has in a sense decoupled 
 integration over
 an interval,
 however this is not unexpected.
 The fundamental theorem of calculus
 itself is an expression of the
 difference of two integrals
 at a point.
\[ \int_{a}^{b} f(x) \,dx = \int^{b} f(x)\,dx - \int^{a} f(x)\,dx \] 
\begin{theo}      
A sum representation of the fundamental theorem of calculus \cite[Theorem 5.1]{cebp10}.
\[ \sum_{a}^{b} f_{n} = \sum^{b} f_{n} - \sum^{a} f_{n} \]
\end{theo}

While an interval integration is
 a concrete and tangible calculation,
 the theory of
 calculus often uses integration at a point 
 implicitly. Other examples can be found,
 such as integrating using the power rule
 $\int x^{p}\,dx = \frac{1}{p+1} x^{p+1}$.

Integration at a point has theoretical advantages,
 the complexity of theory and calculation can
 be reduced. 
 This and subsequent papers in the series have found
 working at infinity, with du Bois-Reymond's
 infinitary calculus and functional space,
 provides ways to build function and sequence constructions.
\bigskip
\begin{prop}\label{P042}
Convergence: When $n_{1} \gt n_{0}$; $\, n_{0}$, $n_{1} \in \Phi^{-1}$;
 there exists $n_{1}:$
\[ \int_{n_{0}}^{n_{1}} a(n)\,dn = \int_{n_{0}} a(n)\,dn|_{n=\infty} \]
\end{prop}
\begin{proof}
By arbitrarily making
 the infinitesimal $\int^{n_{1}}a(n)\,dn$ smaller(E3.6),
 to when the condition $\int^{n_{1}}a(n)\,dn \prec \int^{n_{0}}a(n)\,dn$ is met,
 then $\int^{n_{1}}a(n)\,dn-\int^{n_{0}}a(n)\,dn = -\int^{n_{0}} a(n)\,dn$.
\end{proof}
\bigskip
\begin{prop}\label{P043}
Divergence: When 
$n_{1} \gt n_{0}$; $\, n_{0}, n_{1}\ \in \Phi^{-1}$;
 there exists $n_{0}:$
\[ \int_{n_{0}}^{n_{1}} a(n)\,dn = \int^{n_{1}} a(n)\,dn|_{n=\infty} \]
\end{prop}
\begin{proof}
Since there is no
 least diverging infinity,
 we can always decrease
 an infinity, and from E3.5 we can construct
 an arbitrary smaller infinity.

Then given $n_{1}$ we can decrease $n_{0}$
 till the condition $\int^{n_{1}}a(n)\,dn \succ \int^{n_{0}} a(n)\,dn|_{n=\infty}$ is met.
 Then $\int_{n_{0}}^{n_{1}} a(n)\,dn|_{n=\infty}$
 $=\int^{n_{1}}a(n)\,dn - \int^{n_{0}} a(n)\,dn|_{n=\infty}$
 $=\int^{n_{1}}a(n)\,dn|_{n=\infty}$.
\end{proof}
\bigskip
\begin{theo}\label{P044}
For integral convergence or divergence, we need
 only test one point at infinity.
 $\int a(n)\,dn|_{n=\infty} \in \Phi$ when convergent
 and $\int a(n)\,dn|_{n=\infty} \in \Phi^{-1}$ when
 divergent.
\end{theo}
\begin{proof}
By Criterion E3.1 we 
 need only test a monotonic sequence.
 If E3.1 is satisfied, then either the
 integral is converging or diverging, as
 the sum is an infinireal.
 Both these cases are handled by Propositions
 \ref{P042} and \ref{P043}.
\end{proof}

Both Propositions \ref{P042}, \ref{P043} and Theorem \ref{P044}
 can be demonstrated.
\bigskip
\begin{mex}\label{MEX042}
$\int_{n}^{n^{2}} \frac{1}{x^{4}} \,dx|_{n=\infty}$
$= -\frac{1}{3} \frac{1}{{n^{2}}^{3}}$ 
$ +\frac{1}{3} \frac{1}{n^{3}}|_{n=\infty}$ 
$= -\frac{1}{3} \frac{1}{n^{6}}$ 
$ +\frac{1}{3} \frac{1}{n^{3}}|_{n=\infty}$ 
$= \frac{1}{3} \frac{1}{n^{3}}|_{n=\infty}$
 $=0$ converges.
 $\int^{n^{2}} \frac{1}{x^{4}} \,dx \prec \int^{n} \frac{1}{x^{4}} \,dx|_{n=\infty}$

 By Criterion E3, $\int^{n} \frac{1}{x^{4}} \,dx|_{n=\infty}$
$= \frac{1}{3} \frac{1}{n^{3}}|_{n=\infty}$
 $=0$ converges.
\end{mex}
\bigskip
\begin{mex}\label{MEX043}
$\int_{\mathrm{ln}\,n}^{n} x^{2}\,dx|_{n=\infty}$
 $= \frac{x^{3}}{3}|_{\mathrm{ln}\,n}^{n}|_{n=\infty}$
 $= \frac{n^{3}}{3} - \frac{(\mathrm{ln}\,n)^{3}}{3}|_{n=\infty}$
 $= \frac{n^{3}}{3}|_{n=\infty}$
 $= \infty$
 diverges.
 $\int^{n} x^{2}\,dx \succ \int^{\mathrm{ln}\,n} x^{2}\,dx|_{n=\infty}$

 By Criterion E3, 
 $\int^{n} x^{2}\,dx|_{n=\infty}$
 $= \frac{x^{3}}{3}|_{n=\infty}$
 $= \frac{n^{3}}{3}|_{n=\infty}$
 $= \infty$
 diverges.
\end{mex}
 Similar propositions are constructed for sums.
\bigskip
\begin{prop}\label{P055}
Convergence: When $n_{1} \gt n_{0}$; $ n_{0}$, $n_{1} \in \mathbb{J}_{\infty}$;
 there exists
 $n_{1}:$
\[ \sum_{n_{0}}^{n_{1}} a_{n} = \sum_{n_{0}} a_{n}|_{n=\infty}\]
\end{prop}
\begin{proof}
By arbitrarily making
 the infinitesimal $\sum a_{n}|_{n=n_{1}}$ smaller(E3'.6),
 to when the condition $\sum a_{n}|_{n={n_{1}}} \prec \sum a_{n}|_{n={n_{0}}}$ is met,
 $\sum_{n_{0}}^{n_{1}} a_{n}$
 $=\sum a_{n}|_{n=n_{1}} -\sum a_{n}|_{n=n{o}} = -\sum a_{n}|_{n=n_{0}}$.
\end{proof}
\bigskip
\begin{prop}\label{P056}
Divergence: When 
$n_{1} \gt n_{0}$; $n_{0}, n_{1} \in \mathbb{J}_{\infty}$;
 there exists $n_{0}:$ 
\[ \sum_{n_{0}}^{n_{1}} a_{n} = \sum^{n_{1}} a_{n}|_{n=\infty} \]
\end{prop}
\begin{proof}
Since there is no
 least diverging infinity,
 we can always decrease
 an infinity, and from E3'.5 we can construct
 an arbitrary smaller infinity.

Then given $n_{1}$ we can decrease $n_{0}$
 till the condition $\sum^{n_{1}} a_{n} \succ \sum^{n_{0}} a_{n}|_{n=\infty}$ is met.
 Then $\sum_{n_{0}}^{n_{1}} a_{n}|_{n=\infty}$
 $=\sum^{n_{1}}a_{n} - \sum^{n_{0}} a_{n}|_{n=\infty}$
 $=\sum^{n_{1}}a_{n}|_{n=\infty}$.
\end{proof}
\bigskip
\begin{theo}\label{P057}
For sum convergence or divergence, we need
 only test one point at infinity.
 $\sum a_{n}|_{n=\infty} \in \Phi$ when convergent
 and $\sum a_{n}|_{n=\infty} \in \Phi^{-1}$ when
 divergent.
\end{theo}
\begin{proof}
By Criterion E3'.1 we 
 need only test a monotonic sequence.
 If E3'.1 is satisfied, then either the
 integral is converging or diverging, as
 the sum is an infinireal.
 Both these cases are handled by Propositions
 \ref{P055} and \ref{P056}.
\end{proof}

Consider the tail or remainder
 of an integral. $n_{0}, n_{1} \in \Phi^{-1}$;  
 $\,r(n_{0},n_{1}) = \int_{n_{0}}^{n_{1}} a(n)\,dn$.
 Let $a(n)$ be a continuous 
 function passing through
 the sequence points $(a_{n})|_{n=\infty}$.
 The tail explains the convergence criteria for E1 and E3, 
 and the tail explains monotonic divergence, as the tail
 is at infinity.

Then Criterion E1 is a criterion for the tail $r(n,2n)|_{n=\infty}$.
 We find Criterion E3 also satisfies Criterion E1,
 but also true for any part or whole
 of the tail. 

Criterion E3 is a more encompassing criterion than Criterion E2, as the
 whole tail or sum at infinity is captured. 
 However, only a part of the tail needs to be tested, condition E0.
 With
 the reduction of rearrangement theorems (Section \ref{S16}),
 Criterion E3's limitation
 of requiring monotonic input is addressed.
\subsection{Reflection}\label{S0110}
 In comparing convergence criteria E1 and E3,
 both convergence criteria have a role.
 However, in our opinion we do not see Robinson's non-standard analysis
 being used everyday by engineers. 
 Criterion E3 is better
 suited for this
 purpose and leads to a non-standard model which better supports
 calculation and theory, by solving in a more accessible way.
 That is, solving more simply,
 primarily by removing the ``mathematical logic"- that is
 the field of logic, from the application.
 This is the specialization
 that is necessary to use Robinson's non-standard analysis.

The indirect logic reasoning
 used in the current convergence and
 divergence tests is similarly frustrating.
 While there are many uses for such
 logic, with direct reasoning at infinity, 
 classes of these can be reduced.
 A particular example is the Limit Comparison Theorem (LCT).
 The act of picking the 
 correct series to apply the test,
 in full knowledge that
 the test is going to work
 (as often the asymptotic series
 is chosen), is unnecessary (see Example \ref{MEX035}). 
 Not from a logical viewpoint, but from
 an application viewpoint, 
 you have
 already solved the problem!
\bigskip
\begin{mex}\label{MEX035}
Determine using LCT that $\sum_{k=1}^{n} \frac{1}{n^{2}+1}$ converges.
 By the p-series test, we know $\sum \frac{1}{n^{2}}|_{n=\infty}=0$
 converges. Compare using Test \ref{S0707}, $\,\sum \frac{1}{n^{2}+1}|_{n=\infty}$ 
 and $\sum \frac{1}{n^{2}}|_{n=\infty}$.
 $\frac{1}{n^{2}+1} \frac{n^{2}}{1}|_{n=\infty}$  
 $= 1 - \frac{1}{n^{2}+1}|_{n=\infty}=1$ hence both sums converge.
\end{mex}

The tests under the
 new convergence criteria have been
 developed to address these problems,
 with more direct reasoning.
 This will include in a later paper
 a concept of a derivative of a
 sequence (Section \ref{S15}). The flow and mixing of these
 tests, and the introduction of
 a universal convergence test (see Test \ref{S0720}),
 we believe, will change the nature
 of convergence testing.

 In introducing a new way of calculating, a focus
 was chosen toward the development of convergent
 and divergent series, theory and tests.
 Without further ado, the following discussion relates to the E3 convergence
 criterion.
\bigskip
\begin{defy}
`convergence sums' use Criterion E3,
 `convergence integrals' use Criterion E3.
\end{defy}
\subsection{Monotonic sums and integrals}\label{S0111}
 In determining convergence or
 divergence,
 if the sum or integral is well behaved,
 that is no singularities on the finite
 interval,
 then we only need to state this result
 at the end, to say whether the sum or integral
 has converged or diverged.

 That convergence or divergence,
 with no finite singularities,
 is determined at infinity,
 is self evident and
 purely logical.  If a sum or integral does not diverge for
 any finite values, then the only possible place
 where the sum or integral can diverge is at infinity.
 For all monotonic series and integrals,
 the convergence tests work at infinity.

Singularities are often caused by division
 by zero, so when these singularities
 are not present, what
 remains is a continuous function (see \ref{MEX033}),
 which could then be tested at the singularity
 infinity. 
 For a series, the continuous function is constructed
 by threading a continuous function through
 the sequence points.
\bigskip
\begin{theo}
If the integral $\int_{\alpha}^{\infty}a(n)\,dn$ has no 
 singularities before infinity and
\[
\int a(n)\,dn|_{n=\infty} = \left\{ 
  \begin{array}{rl}
    0 & \; \text{then } \int_{\alpha}^{\infty}a(n)\,dn \text{ is convergent,} \\
    \infty & \; \text{ then } \int_{\alpha}^{\infty}a(n)\,dn \text{ is divergent.}
  \end{array} \right. 
\]
\end{theo}
\begin{proof}
 Since the finite part of the integral has no influence on convergence
 or divergence, the evaluation is achieved using Criterion E3, 
 which was applied in Theorem \ref{P044}.
\end{proof}
\bigskip
\begin{theo}
If the sum $\sum_{k=k_{0}}^{\infty}a_{k}$ has no 
 singularities before infinity and
\[
\sum a_{n}|_{n=\infty} = \left\{ 
  \begin{array}{rl}
    0 & \; \text{then } \sum_{k=k_{0}}^{\infty} a_{k} \text{ is convergent,} \\
    \infty & \; \text{ then } \sum_{k=k_{0}}^{\infty} a_{k} \text{ is divergent.}
  \end{array} \right. 
\]
\end{theo}
\begin{proof}
 Since the finite part of the sum has no influence on convergence
 or divergence, the evaluation is achieved using Criterion E3', 
 which was applied in Theorem \ref{P057}.
\end{proof}

 The given theorems are for using the infinity convergence test easily.

 When the condition of no singularities
 for finite values occurs, 
 the sum at infinity's purpose will be to decide and hence complete 
 the convergence test.
 In such a sum it is enough to refer to the
 sum at infinity to determine the sum's convergence/divergence.

For example, in testing $\sum_{k=1}^{\infty} \frac{1}{n^{2}}$
 we find $\sum \frac{1}{n^{2}}|_{n=\infty}=0$ converges.
 This is the end of the test, the implication that
 the original series $\sum_{k=1}^{\infty} \frac{1}{n^{2}}$
 converges is unnecessary. 

Formally repeating the implication 
 of convergence or divergence of
 the sum from the infinite partition
 is superfluous.
 Convergence or divergence at infinity
 logically follows from the sum.
 
 Having said that, by convention  
 we state $\sum a_{n}|_{n=\infty}=0$ converges.
 Technically this symbolic statement alone says the series converges;
 however we have chosen 
 to state convergence or divergence 
 at the end to aid communication, 
 even though it is saying the same thing twice,
 once symbolically and another with a word.

 Through rearrangements, we believe the convergence
 criteria applies
 to all series, but the application of the criteria
 generally works on monotonic series.
 The reasons for this are that the major
 mathematical tools for comparing functions
 and the use of infinitary calculus, work
 best for monotonic sequences.
  Monotonic sequences and series are a more primitive structure
  than non-monotonic sequences and series, so we can
 develop the theory 
 and use them as building blocks.

Testing on monotonic series 
 makes sense if we imagine at infinity,
 in a similar
 way to the steady state, where the sum or system
 settles down. 
 It is not till we have a settled state
 that we can determine or apply the convergence criterion.
 Hence the need for monotonic sequence $(a_{n})|_{n=\infty}$.

 While this reduces the classes of series that
 can be considered, in actual effect it
 applies a structure to the test. 
 A method will be developed to transform or 
 rearrange sums that are not monotonic to sums
 that are monotonic, and then
 test the transformed sum for convergence or divergence. 
 The restriction of monotonic sequences
 in determining convergence or divergence will be addressed by
 extending the class of sequences 
 in another paper. For example,
 the Alternating Convergence Theorem (Theorem \ref{P007}) is excluded
 by such a restriction.
 So the restriction is not a restriction at all, but
 a problem reduction.

Consider the steady state of the sum at infinity.
 Let $(a_{n})|_{n=\infty}$ describe the
 sequence at infinity. 
 Then we can categorize the
 relationship cases between sum convergence
 and the sequence behaviour.

1. Sequence $(a_{n})$ is monotonic \\ 
2. Sequence $(a_{n})$ is not monotonic and the sum at infinity diverges. \\
3. Sequence $(a_{n})$ is not monotonic and the sum at infinity converges.

1.1 $\sum a_{n}|_{n=\infty}=0$ converges then
 $(a_{n})|_{n=\infty}$ is monotonically decreasing. \\
1.2 $\sum a_{n}|_{n=\infty}=\infty$ diverges
 and 
 $(a_{n})|_{n=\infty}$ is monotonically increasing. \\
1.3 $\sum a_{n}|_{n=\infty}=\infty$
 diverges and $(a_{n})|_{n=\infty}$ is monotonically decreasing.

 For example, case 3 is 
 used 
 in the comparison test variation
 Theorem \ref{P011},
 and 
 the Alternating Convergence Theorem \ref{P007}.

 Most of the tests assume a
 monotonic sequence, case 1.
 Further, just because a sequence is not case 1,
 does
 not mean that the problem cannot be transformed
 into a case 1 category.
\bigskip
\begin{mex}\label{MEX036}
 1.1: $\sum (\frac{1}{2})^{n}|_{n=\infty}$,
 1.2: $\sum 2^{n}|_{n=\infty}$,  
 1.3: $\sum \frac{1}{n}|_{n=\infty}=\infty$. 
\end{mex}
\subsection{Comparing sums at infinity}\label{S05}
The comparison at infinity is to remove the finite part of the sum,
 and we seek monotonic behaviour to test for convergence/divergence.
 So a sequence $(a_{n})$ becomes a sequence at infinity, $(a_{n})|_{n=\infty}$.
 In a sense the finite part of the sum is the transient.

By choosing positive monotonic sequences,
 primitive relations between sequences and sums, functions and integrals,
 can be preserved, practically allowing an interchange between the inequalities.
\bigskip
\begin{theo}\label{P038}
 $f_{n}, g_{n} \in \mathbb{R}_{\infty}$;
 $f_{n}\geq 0, g_{n}\geq 0$. If $z \in \{ \lt, \leq, \gt, \geq \}$
 and the sequences $f_{n}$ and $g_{n}$ are monotonic then 
 $\sum f_{n} \; z \; \sum g_{n} \Leftrightarrow f_{n} \; z \; g_{n}|_{n=\infty}$ in $*G$.
\end{theo}
\begin{proof}
 Thread a continuous monotonic function through the sequence, which
 satisfies E3 criteria. Then apply Theorem \ref{P039}.
\end{proof}
\bigskip
\begin{theo}\label{P039}
 Let $f(n) \geq 0, g(n) \geq 0$ be monotonic, continuous functions at infinity.
 If  $z \in \{ \lt, \leq, \gt, \geq \}$ then  
 $\int f(n)\,dn \; z \; \int g(n)\,dn \Leftrightarrow f(n) \; z \; g(n)|_{n=\infty}$ \cite[Part 6]{cebp21}
\end{theo}

Often we solve for the relation $z$ to include equality, so
 when applying the transfer principle
 the relation is unchanged.
 When  ; $f, g \in \mathbb{R}_{\infty}$; then 
 $(*G,\lt) \not\mapsto (\mathbb{R} \text{ or } \overline{\mathbb{R}}, \lt)$.
 However,
 $(*G,\leq) \mapsto (\mathbb{R} \text{ or } \overline{\mathbb{R}}, \leq)$.
\bigskip
\begin{mex}
 $x^{2} \lt x^{3}|_{x=\infty}$ in $*G$ $\not\mapsto \; \infty \lt \infty$ in $\overline{\mathbb{R}}$.  
 $x^{2} \leq x^{3}|_{x=\infty}$ in $*G$ $\mapsto \; \infty \leq \infty$ in $\overline{\mathbb{R}}$ 
\end{mex}
\bigskip
\begin{mex}
Compare $\sum \frac{1}{n^{3}} \; z \; \sum \frac{1}{n^{2}}|_{n=\infty}$.
 Solve for $z$.
\begin{align*}
 \sum \frac{1}{n^{3}} \; z \; \sum \frac{1}{n^{2}}|_{n=\infty} \tag{remove the sum} \\
 \frac{1}{n^{3}} \; z \; \frac{1}{n^{2}}|_{n=\infty} \tag{cross multiply} \\
 n^{2} \; z \; n^{3} \\
 n^{2} \lt n^{3} \tag{$(+\Phi^{-1}, \lt) \not\mapsto (\infty,\lt)$} \\
  & \tag{$*G \mapsto *G: \;\; \lt \; \Rightarrow \; \leq$} \\ 
 n^{2} \leq n^{3} \tag{relation can be realized, $\infty \leq \infty$} \\
 \sum \frac{1}{n^{3}} \leq \sum \frac{1}{n^{2}}|_{n=\infty} \tag{Solve for $z = \; \leq$ and substitute back} \\
\end{align*}
\end{mex}
\begin{mex}
If we know $\sum \frac{1}{n^{2}}|_{n=\infty}=0$ converges, 
 then by comparison we can determine the convergence of $\sum \frac{1}{n^{3}}|_{n=\infty}$.
\begin{align*}
 0 \leq \sum \frac{1}{n^{3}} \leq \sum \frac{1}{n^{2}}|_{n=\infty} \tag{$\Phi \mapsto 0$} \\
 0 \leq \sum \frac{1}{n^{3}}|_{n=\infty} \leq 0  \\
 \sum \frac{1}{n^{3}}|_{n=\infty}=0 \tag{sum converges}
\end{align*}
\end{mex}

The idea is, that if the sum's state has reached a steady
 state at infinity (this could of course be a function in $n$),
 then we may compare different sums. Now, our sums at 
 a point could simply have the summation sign sigma removed, and
 by infinitary calculus the inner functions compared.

If the sums $\sum f_{n}|_{n=\infty}$ and $\sum g_{n}|_{n=\infty}$
 are monotonic,
 then at infinity their state
 is settled, hence
 $\sum f_{n} \; z \; \sum g_{n} \Rightarrow f_{n} \; z \; g_{n}|_{n=\infty}$

 Showing 
 $f_{n} \; z \; g_{n} \Rightarrow \sum f_{n} \; z \; \sum g_{n}|_{n=\infty}$ and then reversing the argument shows
 the implication in the other direction.
$f_{n} \; z \; g_{n}|_{n=\infty}$,
$f(n) \; z \; g(n)|_{n=\infty}$,
$f(n) \Delta n \; z \; g(n) \Delta n|_{n=\infty}$,
$\int f(n) \,dn \; z \; \int g(n)\,dn|_{n=\infty}$,
$\sum f_{n} \; z \; \sum g_{n}|_{n=\infty}$.

Converting the sum to an integral and then differentiating,
 as this is dividing by positive infinitesimal
 quantities equally, the relation $z$ will not change.
 $\sum f_{n} \; z \; \sum g_{n}|_{n=\infty}$,
 $\int f(n)\,dn \; z \; \int g(n)\,dn|_{n=\infty}$,
 $\frac{d}{dn} \int f(n)\,dn \; (D z) \; \frac{d}{dn} \int g(n)\,dn|_{n=\infty}$,
 $f(n) \; (D z) \; g(n)|_{n=\infty}$,
 $f(n) \; z \; g(n)|_{n=\infty}$.

\[ \sum f_{n} \; z \;\! \sum g_{n} \; \Leftrightarrow  \;\; f_{n} \; z \; g_{n}|_{n=\infty} \]
\[ \int f(n) \,dn \; z \;\! \int g(n) \,dn \; \Leftrightarrow  \;\; f(n) \; z \; g(n)|_{n=\infty} \]

We have chosen continuous curves $f(n)$ and $g(n)$
such that $f_{n} \;\; z \;\; g_{n} \Leftrightarrow f(n) \;\; z \;\; g(n)|_{n=\infty}$, where for integer values $f(n) = f_{n}$ and $g(n)=g_{n}$.
 These continuous curves maintain the relation for their interval
 between the points.

To summarize, comparing two sums
 with each other at infinity,
 the sum over an interval's convergence
 is converted to a sum at a point at infinity
 for both sums,
 and 
 the sums at infinity are compared. 
 By removing the sigma, this comparison now compares the sum's inner components,
 solving for the relation and substituting
 back into the original sum's comparison.

 The conjecture is that
 all sums with integer indices can be compared at infinity.
 This belief will lead to considerations of arrangements,
 and ways the sums can be compared in subsequent papers.
 An example being 
 the Alternating Convergence Theorem (Theorem \ref{P007}),
 with its alternate representation given
 because of the test's importance.
\bigskip
\begin{mex}\label{MEX004}
Determine convergence of 
 $\sum \frac{1}{n \sqrt{n^{3}+1}}|_{n=\infty}$.
For interest, compare with a known convergent
 sum $\sum \frac{1}{n^{2}}|_{n=\infty}=0$.
 $\sum \frac{1}{n \sqrt{n^{3}+1}} \;\; z \;\; \sum \frac{1}{n^{2}}|_{n=\infty}$,
 $\sum \frac{1}{n (n^{3}+1)^{\frac{1}{2}} } \;\; z \;\; \sum \frac{1}{n^{2}}|_{n=\infty}$,
 $\frac{1}{n (n^{3}+1)^{\frac{1}{2}} } \;\; z \;\; \frac{1}{n^{2}}|_{n=\infty}$,
 $n^{2} \;\; z \;\; n (n^{3}+1)^{\frac{1}{2}}|_{n=\infty}$,
 $n^{2} \prec n (n^{3}+1)^{\frac{1}{2}}|_{n=\infty}$.
 For our purposes we will not need as strong  a relation.
 The lesser relation, less than or equal to, will suffice.
 $n^{2} \leq n (n^{3}+1)^{\frac{1}{2}}|_{n=\infty}$.
 Then $z = \; \leq$, substituting
 this back into the original sum comparison,  
 $0 \leq \sum \frac{1}{n \sqrt{n^{3}+1}} \leq \sum \frac{1}{n^{2}}|_{n=\infty}$,
 $0 \leq \sum \frac{1}{n \sqrt{n^{3}+1}}|_{n=\infty} \leq 0$,
 $\sum \frac{1}{n \sqrt{n^{3}+1}}|_{n=\infty}=0$ converges.
\end{mex}

 We did restrict $(a_{n})$ to being monotonic
 in the determination of convergence
 or divergence of $\sum a_{n}|_{n=\infty}$. 
 If we can bound a function between two monotonic
 functions that are both converging or both diverging,
 a determination of convergence or divergence 
 can be made by applying the sandwich principle.
\bigskip
\begin{mex}\label{MEX005}
$\sum_{k=1}^{n} \frac{3 + \mathrm{sin}\,k}{k^{2}}|_{n=\infty}$.
 Determine convergence or divergence of $\sum \frac{3 + \mathrm{sin}\,n}{n^{2}}|_{n=\infty}$.
 $2 \leq 3+\mathrm{sin}\,n \leq 4|_{n=\infty}$,
 $\frac{2}{n^{2}} \leq \frac{3+\mathrm{sin}\,n}{n^{2}} \leq \frac{4}{n^{2}}|_{n=\infty}$,
 $\sum \frac{2}{n^{2}} \leq \sum \frac{3+\mathrm{sin}\,n}{n^{2}} \leq \sum \frac{4}{n^{2}}|_{n=\infty}$,
 $0 \leq \sum \frac{3+\mathrm{sin}\,n}{n^{2}} \leq 0|_{n=\infty}$,
 $\sum \frac{3+\mathrm{sin}\,n}{n^{2}}|_{n=\infty}=0$ converges.
\end{mex}
 
One of the most important comparisons is to compare against the
 p-series, as it greatly simplifies calculation. 
\bigskip
\begin{mex}\label{MEX006}
$\sum \frac{1}{n \sqrt{n^{3}+1}}|_{n=\infty}$
 $= \sum \frac{1}{n(n^{3}+1)^{\frac{1}{2}}}|_{n=\infty}$
 $= \sum \frac{1}{n n^{\frac{3}{2}}}|_{n=\infty}$
 $= \sum \frac{1}{n^{\frac{5}{2}}}|_{n=\infty}$
 $=0$ by comparison with the known p-series, hence
 the sum converges.
 $\sum \frac{1}{n^{p}}|_{n=\infty} = 0$ when $p \gt 1$.
\end{mex}
\bigskip
The p-series comparison is more
 heavily used with sums at infinity,
 as with infinitary calculus, arguments
 of magnitude can reduce the sum to
 a p-series test. The p-series at infinity
 can often replace other tests. 
 In this sense it is different from
 the standard p-series test.
\bigskip
\begin{mex}\label{MEX007}
Rather than using the comparison
 test, 
 $\sum \frac{1}{n+n^{\frac{3}{2}}}|_{n=\infty}$ 
 $=\sum \frac{1}{n^{\frac{3}{2}}}|_{n=\infty}$ 
 $=0$ converges.
 As $n + n^{\frac{3}{2}} = n^{\frac{3}{2}}|_{n=\infty}$
 because $n^{\frac{3}{2}} \succ n|_{n=\infty}$.
\end{mex}

Another comparison is when a sum varies
 within an interval not containing zero,
 as given by example 
\ref{MEX005} with the sum having the terms interval $[2,4]$. 

 As a consequence of Criterion E3,
 when a sum at infinity is either 
 zero or infinity,
  then multiplying the sum by a constant
 or a bounded variable without zero
 leaves the sum unchanged.
 $0 \cdot \alpha_{n} = 0$ and $\alpha_{n} \cdot \infty = \infty$ 
\bigskip
\begin{theo}\label{P059}
 If $+\Phi \lt \alpha_{n} \lt +\Phi^{-1}$ then
 $\sum a_{n} \alpha_{n} = \sum a_{n}|_{n=\infty}$.
\end{theo}
\begin{proof}
Since $\alpha_{n} \neq +\mathbb{R}_{\infty}$,
 $\exists ; \beta_{1}, \beta_{2} \in +\mathbb{R}; : \beta_{1} \leq \alpha_{n} \leq \beta_{2}$.  
 $0 \leq \beta_{1} \leq \alpha_{n} \leq \beta_{2}|_{n=\infty}$,
 $0 \leq \beta_{1} a_{n} \leq \alpha_{n} a_{n} \leq \beta_{2} a_{n}|_{n=\infty}$,
 $0 \leq \sum \beta_{1} a_{n} \leq \sum \alpha_{n} a_{n} \leq \sum \beta_{2} a_{n}|_{n=\infty}$,
 $0 \leq \beta_{1} \sum a_{n} \leq \sum \alpha_{n} a_{n} \leq \beta_{2} \sum a_{n}|_{n=\infty}$.
 If $\sum a_{n}|_{n=\infty}=c$, $c \in \{0,\infty\}$,
 then $\beta_{1}c =c$, $\beta_{2} c=c$.
 Substituting $c$ into the above inequality,
 $0 \leq \beta_{1} c \leq \sum \alpha_{n} a_{n} \leq \beta_{2} c|_{n=\infty}$,
 $0 \leq  c \leq \sum \alpha_{n} a_{n} \leq c|_{n=\infty}$,
 $\sum \alpha_{n} a_{n}|_{n=\infty} = c = \sum a_{n}|_{n=\infty}$.
\end{proof}
\bigskip
\begin{theo}\label{P060}
If we can deconstruct a product inside a sum to a real part and
 an infinitesimal, the infinitesimal may be ignored.
 $\alpha \in \mathbb{R}^{+}$; $\,\delta \in \Phi$;
 \[ \sum a_{n}(\alpha + \delta)|_{n=\infty} = \sum a_{n}|_{n=\infty} \]
\end{theo}
\begin{proof} 
$\sum a_{n}(\alpha + \delta)|_{n=\infty}$
 $= \sum a_{n}|_{n=\infty} + \sum \delta a_{n}|_{n=\infty}$.
 If $\sum \delta a_{n}|_{n=\infty}=\infty$ then $\sum a_{n}|_{n=\infty}=\infty$
 as $\sum \alpha a_{n} \geq \sum \delta a_{n}$.
 If $\sum \delta a_{n}|_{n=\infty}=0$
 then $\sum (\alpha a_{n} + \delta a_{n})|_{n=\infty}$ 
 $= \sum \alpha a_{n} + \sum \delta a_{n}|_{n=\infty}$
 $= \alpha \sum a_{n} + 0|_{n=\infty}$
 $= \sum a_{n}|_{n=\infty}$
\end{proof}
\bigskip
\begin{theo}\label{P061}
If $\Phi^{+} \lt \alpha_{n} \lt +\Phi^{-1}$ or
 $-\Phi^{-1} \lt \alpha_{n} \lt -\Phi$ then
 $\sum \alpha_{n} a_{n} = \sum a_{n}|_{n=\infty}$.
\end{theo}
\begin{proof}
Positive case see Theorem \ref{P059}.
 Negative case,
 $-\Phi^{-1} \lt \alpha_{n} \lt - \Phi$,
 multiply the inequality by $-1$,
 $\Phi^{-1} \gt (-\alpha_{n}) \gt \Phi$,
 $\Phi \lt (-\alpha_{n}) \lt \Phi^{-1}$
 which is the positive case.
\end{proof}
\bigskip
\begin{mex}\label{MEX008}
 $\sum \frac{3 + \mathrm{sin}\,n}{n^{2}}|_{n=\infty}$
 $= \sum \frac{1}{n^{2}}|_{n=\infty}=0$
 as $2 \leq 3 + \mathrm{sin}\,n \leq 4|_{n=\infty}$
 is positive bounded.
\end{mex}
\bigskip
\begin{mex}\label{MEX009}
The bounded variable is used
 in the proof of Dirichlet's test,
 see convergence Test \ref{S0714}.
\end{mex}
When about zero, we need to be careful,
 because infinitesimals can appear in the calculation
 which contradict the use of Theorem \ref{P061}.
\bigskip
\begin{mex}\label{MEX010}
An example of when not to use
 the positive bound.
If we reason that because
 the $\mathrm{sin}$ function
 is bounded, that we can treat this
 as a constant, and since constants
 in sums are ignored the $\mathrm{sin}$ function
 is ignored.

 For demonstration purposes only,
 considering $\mathrm{sin}$ as finite and simplifying
 as a constant (rather than the infinitesimal it is)
 $\sum \frac{1}{n} |\mathrm{sin}\frac{1}{n}|$
 $=\sum \frac{1}{n}|_{n=\infty}=\infty$ diverges.
 This is incorrect.
 Actually the reverse case happens
 and the sum converges.

 The problem is 
 the requirement of
 Theorem \ref{P061} was not met.
 The infinitesimal
 cannot be treated as a constant, as it interacts in
 the test.

 Considering a $\mathrm{sin}$ expansion,
 $|\mathrm{sin}\frac{1}{n}| = \frac{1}{n}|_{n=\infty}$,
 $\sum \frac{1}{n} |\mathrm{sin}\frac{1}{n}| = \sum \frac{1}{n^{2}}|_{n=\infty}=0$ converges.
\end{mex}

 We now venture into the darker side of infinity,
 with non-uniqueness, 
 where out of necessity we need to explain
 a counter-example.
 Or rather, by non-uniqueness,
 the counter-example is nullified.

 Infinity as a space
 branches into other possibilities.
 Our convention to use left-to-right $=$ operator
 as a directed assignment \cite[Definition 2.9]{cebp21}
 allows for exploration at infinity as, rather than
 one line of logic, several may need to be followed.
 In fact the following problem requires non-uniqueness
 at infinity to be understood.

Returning to $\sum a_{n} = \int a(n)\,dn|_{n=\infty}$,
 we have a counter-example where,
 if we directly integrate a convergent integral at infinity,
 Criteria E3 fails to give $0$,
 and hence does not establish convergence.  
\bigskip
\begin{mex}\label{MEX011}
 $\sum \frac{1}{8{n}^{2}+12n+4}|_{n=\infty}=0$
 can be shown
 to converge by the comparison test at infinity,
 or reduce the sum to a known p-series. However
 converting the sum to an integral, without
 applying infinitary simplification,
 then integrating the sum at infinity fails
 to give $0$, contradicting Criteria E3. 

$\sum \frac{1}{8{n}^{2}+12n+4}|_{n=\infty}$
 $=\int \frac{1}{8\,{n}^{2}+12\,n+4} dn|_{n=\infty}$
 $=\frac{\mathrm{ln}\left( 2\,n+1\right) }{4}-\frac{\mathrm{ln}\left( n+1\right) }{4}|_{n=\infty}$
 $= \frac{\mathrm{ln}\,2}{4} \neq 0$ fails Criteria E3.
 $\frac{\mathrm{ln}\,2}{4} \not \in \mathbb{R}_{\infty}$

However, had infinitary arguments been applied,
 we would have got the correct result. 

$\int \frac{1}{8\,{n}^{2}+12\,n+4} dn|_{n=\infty}$
 $=\int \frac{1}{8\,{n}^{2}} dn|_{n=\infty}$
 $= -\frac{1}{8n}|_{n=\infty}$
 $=0$ satisfies Criterion E3'.

The above, by non-uniqueness at infinity, does not contradict.
 Firstly, E3 found the convergence by another path, rather
 than directly integrating. By not applying non-reversible
 arithmetic, this contradicted the E3 criteria, hence
 it is not a valid counter-example.  

 Other criterion did find a solution.
 Again, this is not a contradiction, as different theories
 have different rules. Even when the two calculations
 had the same starting point 
 $\int \frac{1}{8\,{n}^{2}+12\,n+4} dn|_{n=\infty}$ but arrived
 at different conclusions does not contradict because they
 are governed by different sets of rules.
\end{mex}
\bigskip
\begin{remk}
 Another criterion could extend E3 and not require
 the convergence sum to be an infinireal.
 Then, comparing against the boundary the above would have converged.
 However, we have chosen to purse the simpler Criteria E3.
\end{remk}
\[ \text{Conjecture: By application of infinitary 
 arguments before integrating,} \]
\[ \text{removing lower order terms in the
 integral at infinity,} \]
\[ \text{we realize Criterion E3 in testing for convergence.} \]

With the boundary test, see Test \ref{S0720}, which
 is the subject of the next paper,
 the conjecture is justified as lower order terms
 become additive identities and are simplified
 as part of the test.

What this says about infinity
 is interesting.
 Firstly, we are forced to consider infinitary
 arguments to explain what is going on.
 Secondly, the nature of expressions
 at infinity encodes information.
 In this case, persisting with the lower order
 magnitude terms contradicted infinitary 
 magnitude arguments. Having $\int \frac{1}{8n^{2}+12n+4}\,dn|_{n=\infty}$
 has perpetuated the lower order terms $12n+4|_{n=\infty}$ 
 at infinity, and subsequently affected the integration.

 We cannot assume a rule for finite arithmetic will
 carry over and be the same for infinite arithmetic.
 Infinity is realising itself to hold different number
 systems,
 and be more complicated, and yet explain so much more.

 $\int \frac{1}{8n^{2}+12n+4}\,dn \neq \int \frac{1}{8n^{2}}\,dn|_{n=\infty}$
 implicitly assumes $8n^{2} \not \succ 12n+4|_{n=\infty}$, where we associated 
 the much-greater-than relation with non-reversible arithmetic.

If we consider simplification in general, that
 is, where we apply arguments of magnitude,
 then the very process often effects other
 evaluations, particularly comparison.
 In the previous case, 
 $n^2+ \frac{1}{n} \gt n^{2}|_{n=\infty}$,
 but after simplification
 equality is realized($n^{2}=n^{2}$ as $\frac{1}{n}|_{n=\infty}=0$).
 Since the simplification is non-Archimedean arithmetic,
 which exists everywhere in analysis, when limits
 are being taken, the realization of 
 infinitesimals and infinities matters.

This is not a blanket statement of
 denying the realization operation, but 
 that a given simplification may have
 consequences, such as the previous example showed.
 We remind ourselves that truncation is 
 a subset of realization, so the simplification is common.
 Our goal is to 
 look at such arithmetic, and
 manage it.

In this spirit, papers follow with further 
 results and conjectures,
 at present from necessity and empirical observation.
 For example, the classes of sums are greatly
 increased by considering periodic sums at infinity;
 with Criterion E3 as a guide,
 we conjecture a rearrangement for integral sums at infinity
 (Section \ref{S16}).

 Criterion E3 is being used as a necessary test.
 We can relax Criterion E3, but 
 then other results may become less certain, or
 the theory becomes more complex, and may serve 
 other purposes. 
 (Other criteria are possible.)
 While we may conjecture, it is important
 to be able to test that conjecture.
\bigskip
\begin{prop}\label{P053}
 If $f=\infty$ and $\mathrm{ln}\,n \prec f|_{n=\infty}$ then $\sum \frac{1}{e^{f}}|_{n=\infty}=0$
 and $\int \frac{1}{e^{f}}\,dn|_{n=\infty}=0$ converges.
\end{prop}
\begin{proof}
Let $p \gt 1$, comparing against the convergent p-series,
 $\sum \frac{1}{n^{p}}|_{n=\infty}=0$.
 Solving for relation $z$, 
 $\sum \frac{1}{e^{f}} \; z \; \sum \frac{1}{n^{p}}|_{n=\infty}$,
 $\frac{1}{e^{f}} \; z \; \frac{1}{n^{p}}|_{n=\infty}$,
 $n^{p} \; z \; e^{f}|_{n=\infty}$,
 $p \,\mathrm{ln}\,n \; (\mathrm{ln}\,z) \; f|_{n=\infty}$,
 $p \,\mathrm{ln}\,n \prec f|_{n=\infty}$,
 $\mathrm{ln}\,z = \; \prec$,
 $z = e^{\prec} = \; \prec \; = \; \leq$. 
 (Recall from \cite[Part 2]{cebp21} the left-to-right generalization
 of the equals operator, $\prec \; \Rightarrow \; \leq$ for positive
 values.)
 Substituting $z$ back in the sum comparison,
 $0 \leq \sum \frac{1}{e^{f}} \leq 0$,
 $\sum \frac{1}{e^{f}}|_{n=\infty}=0$ converges.
\end{proof}
\bigskip
\begin{mex}\label{MEX044}
 Determine convergence of $\sum \frac{1}{e^{x}}|_{x=\infty}$.
 By Proposition \ref{P053}, $f=x$, 
 $\mathrm{ln}\,x \; z \; x|_{x=\infty}$,
 $\mathrm{ln}\,x \prec x|_{x=\infty}$,
 then  $\sum \frac{1}{e^{x}}|_{x=\infty}=0$ converges.
 
 However, it is easier to compare with a known converging p-series directly.
 $x^{p} \prec e^{x}|_{x=\infty}$,
 $\frac{1}{e^{x}} \prec \frac{1}{x^{p}}|_{x=\infty}$, 
 $\sum \frac{1}{e^{x}} \lt \sum \frac{1}{x^{p}}|_{x=\infty}$, 
 and for $p \gt 1$ the convergence result follows.
\end{mex}
\subsection{Integral and sum interchange}\label{S12}
The integral test is often given in one direction:
 if a sum can be bounded below and above by 
 a monotonic integral and if the integral converges,
 the sum converges, and if the integral diverges
 the sum diverges.

With the convergence sum Criteria E3 and E3',
 positive monotonic sums can be made arbitrarily large or small.
 Then the sum can bound an integral
 and determine the integral's convergence or divergence,
 or vice versa.
 Hence the integral test can be in two directions,
 a consequence of the E3 and E3' criteria.

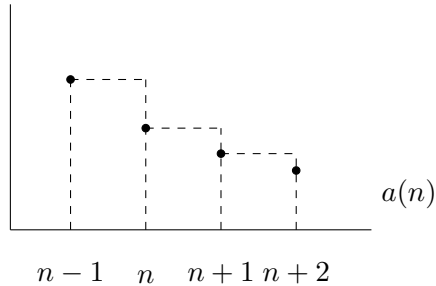
\begin{figure}[H]
\centering
\begin{tikzpicture}[domain=0.5:5.0]
  \tikzstyle{every node}=[font=\small]
\begin{scope}[xshift=00mm, yshift=00mm]
  \draw[color=orange] plot[id=exp] function{5.0/x**(0.20)-3.0}; 
  \draw[fill=black] (1.0,2.0) circle [radius=0.05cm] ;
  \draw[fill=black] (2.0,1.35275) circle [radius=0.05cm] ;
  \draw[fill=black] (3.0,1.0137) circle [radius=0.05cm] ;
  \draw[fill=black] (4.0,0.7892) circle [radius=0.05cm] ;
\end{scope}
\begin{scope}[xshift=00mm, yshift=00mm]
  \draw [black] (0.2,0.0) -- (5.0,0.0);
  \draw [black] (0.2,0.0) -- (0.2,3.0);
  \draw[dashed,black] (1.0,0.0) -- (1.0,2.0);
  \draw[dashed,black] (2.0,0.0) -- (2.0,2.0);
  \draw[dashed,black] (3.0,0.0) -- (3.0,1.35275);
  \draw[dashed,black] (4.0,0.0) -- (4.0,1.0137);
  \draw[dashed,black] (1.0,2.0) -- (2.0,2.0);
  \draw[dashed,black] (2.0,1.35275) -- (3.0,1.35275);
  \draw[dashed,black] (3.0,1.0137) -- (4.0,1.0137);
  \node [label={[shift={(1.0,-1.0)}]$n-1$}] {};
  \node [label={[shift={(2.0,-1.0)}]$n$}] {};
  \node [label={[shift={(3.0,-1.0)}]$n+1$}] {};
  \node [label={[shift={(4.0,-1.0)}]$n+2$}] {};
  \node [label={[shift={(5.5,0.0)}]$a(n)$}] {};
\end{scope}
\end{tikzpicture}
\caption{Strictly monotonic decreasing function} \label{fig:FIG01}
\end{figure}

\begin{figure}[H]
\centering
\begin{tikzpicture}[domain=0.0:5.0]
  \draw[color=orange] plot[id=exp] function{0.08*x*x}; 
  \draw[fill=black] (1.0,0.08) circle [radius=0.05cm] ;
  \draw[fill=black] (2.0,0.32) circle [radius=0.05cm] ;
  \draw[fill=black] (3.0,0.72) circle [radius=0.05cm] ;
  \draw[fill=black] (4.0,1.28) circle [radius=0.05cm] ;
  \node [label={[shift={(1.0cm,0.0cm)}]$a_{n-1}$}] {};
  \node [label={[shift={(2.0cm,0.2cm)}]$a_{n}$}] {};
  \node [label={[shift={(3.0cm,0.70cm)}]$a_{n+1}$}] {};
  \node [label={[shift={(4.0cm,1.40cm)}]$a_{n+2}$}] {};
  \node [label={[shift={(5.5cm,1.75cm)}]$a(n)$}] {};
\end{tikzpicture}
\caption{Strictly monotonic increasing function} \label{fig:FIG02}
\end{figure}
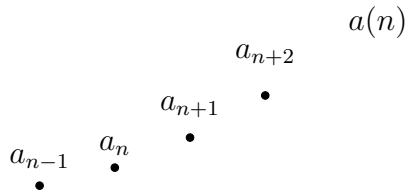
\bigskip
\begin{prop}\label{P062} 
 For a monotonic function $a(x)$ in $*G$.  
\[ \ldots \leq \int^{n-1} a(x)\,dx \leq \int^{n} a(x)\,dx \leq \int^{n+1} a(x)\,dx \leq \ldots|_{n=\infty} \]

For a strictly monotonic function, replace the inequality with a strict inequality.
\end{prop}
\begin{proof} 
 Apply the fundamental theorem of calculus in $*G$ \cite{cebp10}.

Divergence,
 $ \ldots \leq \int^{n-1}_{n_{0}} a(x)\,dx \leq \int^{n}_{n_{0}} a(x)\,dx \leq \int^{n+1}_{n_{0}} a(x)\,dx \leq \ldots|_{n=\infty}$,
 $ \ldots \leq \int^{n-1}a(x)\,dx - \int^{n_{0}} a(x)\,dx \leq \int^{n} a(x)\,dx - \int^{n_{0}} a(x)\,dx \leq \int^{n+1} a(x)\,dx - \int^{n_{0}} a(x)\,dx \leq \ldots|_{n=\infty}$,
 choose $n_{0}$ as an additive identity, 
 $ \ldots \leq \int^{n-1} a(x)\,dx \leq \int^{n} a(x)\,dx \leq \int^{n+1} a(x)\,dx \leq \ldots|_{n=\infty}$.

Convergence,  
 $\ldots \geq \int^{n_{1}}_{n-1} a(x)\,dx \geq \int^{n_{1}}_{n} a(x)\,dx \geq \int^{n_{1}}_{n+1} a(x)\,dx \geq \ldots|_{n=\infty}$, 
 $\ldots \geq \int^{n_{1}} a(x)\,dx - \int^{n-1} a(x)\,dx \geq \int^{n_{1}} a(x)\,dx - \int^{n} a(x)\,dx \geq \int^{n_{1}} a(x)\,dx - \int^{n+1} a(x)\,dx \geq \ldots|_{n=\infty}$, 
 choose $n_{1}$ for an additive identity,
 $\ldots \geq - \int^{n-1} a(x)\,dx \geq - \int^{n} a(x)\,dx \geq - \int^{n+1} a(x)\,dx \geq \ldots|_{n=\infty}$. 
\end{proof}
\bigskip
\begin{prop}\label{P066}
 For a monotonic sequence $a_{n}$ in $*G$.
\[ \ldots \leq \sum a_{n-1} \leq \sum a_{n} \leq \sum a_{n+1} \ldots|_{n=\infty} \]

For a strictly monotonic function, replace the inequality with a strict inequality.
\end{prop}
\begin{proof}
Apply for sums, the mirror to the fundamental theorem of calculus in $*G$ \cite{cebp10}.

Divergence:
 $\ldots \leq \sum_{k=n_{0}}^{n-1} a_{k} \leq \sum_{k=n_{0}}^{n} a_{k} \leq \sum_{k=n_{0}}^{n+1} a_{k} \ldots|_{n=\infty}$, 
 $\ldots \leq \sum a_{n-1} - \sum a_{n_{0}} \leq \sum a_{n} - \sum a_{n_{0}} \leq \sum a_{n+1} - \sum a_{n_{0}} \ldots|_{n=\infty}$, 
 choose $n_{0}$ for an additive identity,
 $\ldots \leq \sum a_{n-1} \leq \sum a_{n} \leq \sum a_{n+1} \ldots|_{n=\infty}$, 

Convergence:
 $\ldots \sum_{k=n-1}^{n_{1}} a_{k} \geq \sum_{k=n}^{n_{1}} a_{k} \geq \sum_{k=n+1}^{n_{1}} a_{k} \geq \ldots$,
 $\ldots \geq \sum a_{n_{1}} - \sum a_{n-1} \geq \sum a_{n_{1}} - \sum a_{n} \geq \sum a_{n_{1}} - \sum a_{n+1} \ldots|_{n=\infty}$, choose $n_{1}$ as an additive identity,
 $\ldots \geq -\sum a_{n-1} \geq -\sum a_{n} \geq -\sum a_{n+1} \ldots|_{n=\infty}$ 
\end{proof}
\bigskip
\begin{prop}\label{P063}
 For monotonic function $a(n)$ and sequence $a_{n}$: $a(n)=a_{n}$ in $*G$.
\[ \sum a_{n} \leq \int^{n}a(x)\,dx \leq \sum a_{n+1} \leq \int^{n+1} a(x)\,dx|_{n=\infty} \]

For a strictly monotonic function,
 replace the inequality with a strict inequality.
\end{prop}
\begin{proof}
 From the geometric construction (see figures \ref{fig:FIG01} and \ref{fig:FIG02}),
 applying the fundamental theorem
 of calculus and sums,
 with a choice of the
 second integrand to be an additive identity. 

Divergence case:
 $\sum_{k=n_{0}}^{n} a_{n} \leq \int_{n_{0}}^{n}a(x)\,dx \leq \sum_{k=n_{0}}^{n+1} a_{k} \leq \int_{n_{0}}^{n+1} a(x)\,dx|_{n=\infty}$,
 $\sum a_{n} - \sum a_{n_{0}} \leq \int^{n}a(x)\,dx - \int^{n_{0}}a(x)\,dx \leq \sum a_{n+1} - \sum a_{n_{0}} \leq \int^{n+1} a(x)\,dx - \int^{n_{0}} a(x)\,dx|_{n=\infty}$, 
 $\sum a_{n} \leq \int^{n}a(x)\,dx \leq \sum a_{n+1} \leq \int^{n+1} a(x)\,dx|_{n=\infty}$, 

Convergence case:
 $\sum_{k=n}^{n_{1}} a_{k} \geq \int_{n}^{n_{1}} a(x)\,dx \geq \sum_{k=n+1}^{n_{1}} a_{k} \geq  \int_{n+1}^{n_{1}} a(x)\,dx|_{n=\infty}$,    
 $\sum a_{n_{1}} -\sum a_{n} \geq \int^{n_{1}} a(x)\,dx - \int^{n} a(x)\,dx \geq \sum a_{n_{1}} - \sum a_{n+1} \geq  \int^{n_{1}} a(x)\,dx - \int^{n+1} a(x)\,dx|_{n=\infty}$,    
 $-\sum a_{n} \geq -\int^{n} a(x)\,dx \geq -\sum a_{n+1} \geq -\int^{n+1} a(x)\,dx|_{n=\infty}$.    
\end{proof}

Further assumptions about the steady state
 are that successive terms reach the
 same state at infinity. 
 This may seem contradictory,
 but it is a property of non-uniqueness
 at infinity. We can have $n \lt n+1 \lt n+2 \lt \ldots|_{n=\infty}$
 and $n = n+1 = n+2 = n+3|_{n=\infty}$ after approximation where the
 overriding magnitude dominates.
 We would say that they are just two different views of the
 same event.

 Another example,
 $(n+1)^{2} \gt n^{2}|_{n=\infty}$,
 but we can find equality in the leading coefficient,
 $n^{2} + 2n + 1 == n^{2}|_{n=\infty}$ as $n^{2} \succ 2n+1$.

 Such a view of magnitude leads to the following steady state interpretation at infinity.
\bigskip
\begin{prop}\label{P064}
\[ \sum a_{n} = \sum a_{n+1}|_{n=\infty} \]
\[ \int^{n} a(x) \,dx = \int^{n+1} a(x)\,dx|_{n=\infty} \]
\end{prop}
\begin{proof}
 Since the sequence of terms is monotonic,
 any two consecutive terms are more equal to
 each other than terms further way
 as the sequence increases.
 In this context we define equality 
 even when the derivative is diverging.

 Case $\sum a_{n} \in \Phi$, $\Phi \mapsto 0$ and consecutive sums
 obtain equality.
 Case $\sum a_{n} \in \Phi^{-1}$, $\Phi^{-1} \mapsto \infty$ and consecutive sums
 obtain equality.
\end{proof}
Since consecutive sums sandwich between the integral,
 and consecutive integrals sandwich between
 the sum, 
 if one converges 
 or diverges so does the other.
\bigskip
\begin{theo}\label{P065} 
The integral test in both directions, interchanging sums and integrals
 at infinity. $*G \mapsto \overline{\mathbb{R}}_{\infty}$
\[ \sum a_{n} = \int^{n} a(x)\,dx|_{n=\infty} \]
\end{theo}
\begin{proof}
 Since both sequences $(\sum a_{n})|_{n=\infty}$ and $(\int^{n} a(n)\,dn)|_{n=\infty}$
 are monotonically increasing, Proposition \ref{P063},
 these inequalities show that both are bounded above or both are unbounded. 
 Therefore,
 both sequences converge or both diverge.

 $s \in \{ 0, \infty \}$;
 $s' \in \{ 0, \infty\}$;
 By Proposition \ref{P064}, let
 $s = \sum a_{n}|_{n=\infty} = \sum a_{n+1}|_{n=\infty}$.
 By Proposition \ref{P063},
 $s \leq \int^{n} a(x)\,dx|_{n=\infty} \leq s$
 then $\int^{n} a(x)\,dx|_{n=\infty} = s$.
 Similarly 
 By Proposition \ref{P064}, 
 let $s' = \int a(x)\,dx|_{x=\infty}$.
 By Proposition \ref{P063}, $s' \leq \sum a_{n}|_{n=\infty} \leq s'$,
 $\sum a_{n}|_{n=\infty}=s'$
\end{proof}
\subsection{Convergence integral testing}\label{S09}
Identify the singularity points in the domain. If any of these
 diverge, the integral diverges. If they all converge,
 the integral converges.
\bigskip
\begin{mex} 
 \cite[Problem 1554, p.145]{demidovich}
Test $\int_{-\infty}^{\infty} \frac{dx}{1+x^{2}}$ for convergence 
 or divergence.
 Consider the singularity points $x=\pm \infty$.
 Use non-reversible arithmetic,
 as $x^{2} \succ 1|_{x=\infty}$
 then $x^{2}+1 = x^{2}|_{x=\infty}$.

 $\int^{\infty} \frac{dx}{1+x^{2}}$
 $=\int^{x} \frac{dx}{x^{2}}|_{x=\infty}$
 $= -\frac{1}{x}|_{x=\infty}$
 $=0$ converges.  Similarly,
 $\int_{-\infty}\frac{dx}{1+x^{2}}$
 $=\int_{-\infty}\frac{dx}{x^{2}}=0$
 converges.
\end{mex}
\bigskip
\begin{mex}
 \cite[Example 3, p.144]{demidovich}
Test $\int_{0}^{\infty} e^{-x^{2}}dx$ for convergence.

 Test the point of discontinuity $x=\infty$,

 A solution by Proposition \ref{P053}.
 $\mathrm{ln}\,x \; z \; x^{2}|_{x=\infty}$,
 $\mathrm{ln}\,x \prec x^{2}|_{x=\infty}$
 then
 $\int^{x} e^{-x^{2}}\,dx|_{x=\infty}$
 $=0$ converges and the initial integral converges.

 By a comparison against a known convergent integral,
 $p \gt 1$,
 $\int e^{-x^{2}}\,dx \; z \; \int \frac{1}{x^{p}}dx|_{x=\infty}$,
 $e^{-x^{2}} \; z \; \frac{1}{x^{p}}|_{x=\infty}$,
 $-x^{2} \; (\mathrm{ln}\,z) \; -p \,\mathrm{ln}\,x|_{x=\infty}$,
 $-x^{2} \; (\mathrm{ln}\,z) \; 0|_{x=\infty}$,
 $-x^{2} \succ 0|_{x=\infty}$,
 $e^{-x^{2}} \prec 1|_{x=\infty}$,
 $z = \; \prec = \leq$, and the integral converges.
\end{mex}
\bigskip
\begin{mex}
 \cite[Example 5, p.145]{demidovich}
Test for convergence of the elliptic integral
 $\int_{0}^{1} \frac{dx}{\sqrt{1-x^{4}}}$.

 The point of discontinuity of the integrand is $x=1$.
 Expand and use non-reversible arithmetic,
 $(1-x)^{4}  = 1 -4x + 6x^{2}-4x^{3}+x^{4}|_{x=0}$
 $=1-4x|_{x=0}$,
 then
 $(1-(x-1)^{4})|_{x=0}$
 $= -4x|_{x=0}$,
 
 $\int^{1^{-}} \frac{dx}{(1-x^{4})^{\frac{1}{2}}}$ 
 $=\int^{0^{-}} \frac{dx}{(1-(1+x)^{4})^{\frac{1}{2}}}$
 $=\int^{0^{-}} \frac{dx}{(-4x)^{\frac{1}{2}}}$
 $=-\int^{x} \frac{1}{2} \frac{dx}{x^{\frac{1}{2}}}|_{x=0}$
 $= x^{\frac{1}{2}}|_{x=0}=0$ converges.
\end{mex}
\subsection{Convergence tests}\label{S07}
The following is an exploration of sums
 convergence tests
 where the tests are rewritten
 with respect to the 
 \textbf{E3} convergence criteria.
 These tests are derived.
 Known results are 
 derived and re-written
 in terms of the new theory.
 This also helps to demonstrate
 the theory from a theoretical point of view.

It is assumed unless otherwise stated
 that the series being tested is monotonic.
 That is, given a series at infinity $\sum a_{n}|_{n=\infty}$,
 we can construct an associated sequence of terms
 from the series,
 $(a_{n})|_{n=\infty}$. We require
 this sequence to be monotonic.

 A notable exception is the 
 Alternating convergence test,
 which has a requirement that
 the general term is $(-1)^{n}a_{n}$,
 then $(a_{n})|_{n=\infty}$ is monotonic. 
 
Mathematics often implicitly
 works with infinitesimals and infinities
 but does not declare this, for example
 in the calculation of limits.
 When a reference is made, it is 
 usually to the extended reals $\overline{\mathbb{R}}$,
 however these do not explicitly declare
 or state infinitesimals or infinities,
 but $\pm \infty$ the number.

We do not necessarily mind the implicit use,
 however to be
 more descriptive, we generally
 reason in $*G$,
 and project back and state the proposition or theorem
 in $\mathbb{R}/\overline{\mathbb{R}}$. 
 As a default, we have done this for 
 the convergence tests in this paper.

 However, 
 there is a difference
 between transferring a result for $*G$ to  
 $\mathbb{R}/\overline{\mathbb{R}}$, 
 the $\lt$ relation.

Consider $*G: \;n^{2} \lt n^{3}|_{n=\infty}$ projected to
 $\overline{\mathbb{R}}$ then $\infty \lt \infty$
 is a contradiction.
 A similar contradiction occurs when
 we project an infinitesimal
 relation $\frac{1}{n^{3}} \lt \frac{1}{n^{2}}|_{n=\infty}$
 then $0 \lt 0$ contradicts. 

 Pragmatically, if we use $(*G, \lt) \mapsto (\mathbb{R}/\overline{\mathbb{R}},\leq)$
 then such contradictions can be minimized.
 (See \cite[Theorem 4.4]{cebp21})
 
 We can geometrically understand
 this as $*G$ with infinitesimals and infinities,
 being a much more dense space.
 We may have
 an infinity of curves infinitely close
 to each other in $*G$ project back (by
 infinitesimal truncation)
 to a single curve in $\mathbb{R}$
 \cite[Example 3.20]{cebp21}.

 However, the projection of infinitesimals to $0$ and positive infinities to $\infty$,
 in a contradictory way allows us to sandwich diverging infinities,
 or infinitesimals which are growing apart infinitesimally sandwiched to $0$.
 And hence we say
 $\sum a_{n}|_{n=\infty}=0 \text{ or } \infty$.
\subsubsection{p-series test}\label{S0701}
\begin{theo}\label{P054}
\[ \sum \frac{1}{n^{p}}|_{n=\infty}=0
 \text{ converges when } p \gt 1 \text{ and} \]  
\[ \sum \frac{1}{n^{p}}|_{n=\infty}=\infty
 \text{ diverges when } p \leq 1 \]
\[ \int \frac{1}{x^{p}} \,dx|_{x=\infty}=0
 \text{ converges when } p \gt 1 \text{ and} \] 
\[ \int \frac{1}{x^{p}} \, dx|_{x=\infty}=\infty
 \text{ diverges when } p \leq 1 \]
\end{theo}
\begin{proof}
The p-series sum test follows from applying the
 integral test and considering the
 convergence of the analogous integral .

Since the sum has no singularity
 except at infinity, hence this is the only
 place that the sum can diverge.
 $\sum \frac{1}{n^{p}}|_{n=\infty} = \int \frac{1}{x^{p}} \,dx|_{x=\infty} = \frac{1}{-p+1} \frac{1}{x^{p-1}}|_{x=\infty}$
 when $p \neq 1$. 
$p \lt 1$ then 
 $\frac{1}{-p+1} \frac{1}{x^{p-1}}|_{x=\infty}= \infty$ diverges, 
 $p \gt 1$ then $\frac{1}{-p+1} \frac{1}{x^{p-1}}|_{x=\infty}= 0$ converges.
 When $p=1$, $\sum \frac{1}{n} = \int \frac{1}{x}\,dx|_{x=\infty}$
 $=\mathrm{ln}\,x|_{x=\infty} = \infty$ which diverges.
\end{proof}
\bigskip
\begin{mex2}\label{MEX026}  
The p-test often uses infinitary calculus.
 $\sum \frac{1}{n(n+1)}|_{n=\infty}$
 $= \sum \frac{1}{n^{2}+n}|_{n=\infty}$
 $=\sum \frac{1}{n^{2}}|_{n=\infty}=0$ converges. 
 Perhaps a trickier but equally valid way, $n = n+1|_{n=\infty}$ 
 then $n(n+1)=n^{2}|_{n=\infty}$ and the simplification follows.
\end{mex2}
\bigskip
\begin{mex2}\label{MEX027}  
 $\sum \frac{5n+2}{n^{3}+1}|_{n=\infty}$
 $= \sum \frac{5n}{n^{3}}|_{n=\infty}$
 $= \sum \frac{5}{n^{2}}|_{n=\infty}=0$ converges.
 Alternatively see Example \ref{MEX028}.
\end{mex2}

 We can convert the integral into a series, for example power series
 or a Riemann sum and test the integral as a series.
 However, we often convert from a series to an integral
 because it is easier work with continuous functions.
 For example, to an integral apply the chain rule.
\subsubsection{Power series tests}\label{S0702} 
While this is not generally called a test,
 possibly because $x=0$ is almost always a solution,
 hence there generally is a convergence at $x=0$.
 As the method uses other tests at the interval end points,
 and by its nature it has been included.
 Power series convergence sums (Section \ref{S14}) 
 is more detailed.
\bigskip
\begin{theo}\label{P029}
Transform
 the sum  
 $\sum a_{n}x^{n} = \sum (b_{n}x)^{n}|_{n=\infty}$.
 For convergence 
 $\sum (b_{n}x)^{n}|_{n=\infty}=0$, 
 solving for $|b_{n}x| \lt 1$, 
 the radius of convergence $r = \frac{1}{|b_{n}|}|_{n=\infty}$
 If $r$ exists $x=(-r,r)$ converges. For the interval of convergence
 the end points need to be tested. (Section \ref{S14}) 
\end{theo}
\bigskip
\begin{mex2}\label{MEX032}
Determine the radius of convergence 
 and the interval of convergence of the following power
 series.
 $\sum \frac{n x^{n}}{2^{n+1}}$,
 the only place divergence is taking place is at the point at infinity.
 $\sum \frac{n x^n}{2^{n+1}} |_{n=\infty}$ 
 $= \sum \frac{1}{2} n (\frac{x}{2})^{n}|_{n=\infty}$ 
 $= \sum (n^{\frac{1}{n}} \frac{x}{2})^{n} |_{n=\infty}$ 
 $= \sum (\frac{x}{2})^{n}|_{n=\infty} =0$ when
 $| \frac{x}{2} | \lt 1$, 
 $| x | \lt 2$,  
 $r=2$.
 Investigating the end points,
 case $x=2$, $\sum \frac{n 2^{n}}{2^{n+1}}|_{n=\infty}$
 $=\sum n|_{n=\infty}=\infty$ diverges.
 Case $x=-2$, $\sum \frac{n (-2)^{n}}{2^{n+1}}|_{n=\infty}$
 $=\sum n (-1)^{n}|_{n=\infty}=\infty$ diverges.
 Interval of convergence $x=(-2,2)$.
\end{mex2}
\bigskip
\begin{mex}\label{MEX014}  
Find the radius of convergence and convergence interval for 
 $\sum_{n=1}^{\infty} \frac{ x^n}{ n^2 3^n }$.

Finding the radius of convergence at infinity, 
$\sum \frac{1}{n^2 3^n } x^n|_{n=\infty}$ 
 $= \sum \frac{1}{3^n } x^n|_{n=\infty}$ [as $3^{n} \succ\!\succ n^{2}|_{n=\infty}$] 
 $= \sum (\frac{x}{3})^n|_{n=\infty}$, 
 $| \frac{x}{3} | \lt 1$,  
 $|x| \lt 3$,  
 $r=3$. Alternatively see Example \ref{MEX031}.

Testing the end points of the interval,
 case $x=-3$,  
 $\sum \frac{ (-3)^{n}}{n^{2} 3^{n}}|_{n=\infty}$ 
 $=\sum \frac{ (-1)^{n}3^{n}}{n^{2} 3^{n}}|_{n=\infty}$ 
 $=\sum \frac{(-1)^n}{n^2}|_{n=\infty}=0$ converges by the ACT (Theorem \ref{P007}).  
 Case $x=3$, 
 $\sum \frac{ (3)^{n}}{n^{2} 3^{n}}|_{n=\infty}$ 
 $=\sum \frac{1}{n^2}|_{n=\infty}=0$ converges. 
 Interval of convergence $x=[-3,3]$
\end{mex}

\subsubsection{Integral test}\label{S0703}  
See Theorem \ref{P065}. $\sum a_{n} = \int^{n} a(x)\,dx|_{n=\infty}$
\bigskip
\begin{mex2}\label{MEX045}
$\sum \frac{\mathrm{ln}\,n}{n^{2}}|_{n=\infty}$
 $= \int \frac{\mathrm{ln}\,n}{n^{2}} \,dn|_{n=\infty}$.
 Using the symbolic maths package Maxima as a calculator to solve
 the integral, 
 $\mathrm{integrate(log(n)/n}$\verb|^|$\mathrm{2,n);}$
 $\int \frac{\mathrm{ln}\,n}{n^{2}} \,dn|_{n=\infty}$
$= - \frac{\mathrm{ln}\,n}{n} - \frac{1}{n}|_{n=\infty}=0$ converges.
 Hence 
 $\sum \frac{\mathrm{ln}\,n}{n^{2}}|_{n=\infty}=0$
 converges.
\end{mex2}

The integral test allows the application of the
 chain rule.  Consider when the subscript is
 itself a diverging function, and its derivative
 is a product.  The sum's constant multiplier
 is irrelevant.
 $f_{n} \to \infty$
 Determine convergence/divergence. 
 $\sum a_{f_{n}} \frac{d f_{n}}{dn} |_{n=\infty}$
 $=\int a(f(n)) \frac{d f(n)}{dn} dn |_{n=\infty}$
 $=\int a(f(n)) \,d f(n)|_{n=\infty}$
 $=\int a(n) \,dn|_{n=\infty}$
 $=\sum a_{n}|_{n=\infty}$

\bigskip
\begin{mex2}\label{MEX046}
 \cite[10.14.16, p.16]{apostol}
 $\sum n e^{-n^{2}}|_{n=\infty}$ 
 $=\int -2n e^{-n^{2}}\,dn|_{n=\infty}$ 
 $=\int \frac{d(-n^{2})}{dn} e^{-n^{2}}\,dn|_{n=\infty}$ 
 $=\int  e^{-n^{2}}d (-n^{2})|_{n=\infty}$ 
 $= e^{-n^{2}}|_{n=\infty}$
 $=0$ converges.
\end{mex2}
\bigskip
\begin{mex2}\label{MEX047}
 \cite[3.2.31.a, p.96]{kaczor}
 Show $\sum a_{n}|_{n=\infty}=0$ converges if the given series converges. 
$\sum 3^{n} a_{3^{n}}|_{n=\infty}=0$.
$\sum 3^{n} a_{3^{n}}|_{n=\infty}$
$= \int 3^{n} a(3^{n})dn|_{n=\infty}$
$= \int 3^{n} a(3^{n})\frac{dn}{d 3^{n}} d(3^{n})|_{n=\infty}$
$= \int 3^{n} a(3^{n})\frac{1}{\mathrm{ln}\,3 \cdot 3^{n}} d(3^{n})|_{n=\infty}$
$= \int a(3^{n}) d(3^{n})|_{n=\infty}$
$= \int a(n) dn|_{n=\infty}$
$= \sum a(n)|_{n=\infty}$
 $=0$ converges.
\end{mex2}

 The integral test can be combined with other tests, which makes it 
 really useful.
\bigskip
\begin{mex2}\label{MEX048}
 \cite[2.3.4]{kaczor}
Given $a_{n}|_{n=\infty}=\infty$,  
 show 
$\frac{1}{n} \sum_{k=1}^{n} a_{k}|_{n=\infty}=\infty$. 

 A ratio in integers
 is converted by the integral test to 
 a continuous variable,
 and L'Hopital's rule is applied.

 Since $a_{n}|_{n=\infty}=\infty$, $\int a_{n}\,dn|_{n=\infty} = \infty$,
 $\frac{1}{n} \sum a_{n}|_{n=\infty}$
 $=\frac{1}{n} \int a_{n}\,dn|_{n=\infty}$,
 since $\frac{\infty}{\infty}$ form then differentiate,
 $\frac{1}{n} \int a_{n}\,dn|_{n=\infty}$
 $=\frac{1}{\frac{d}{dn} n} \frac{d}{dn} \int a_{n}\,dn|_{n=\infty}$
 $=a_{n}|_{n=\infty}$
 $=\infty$.
\end{mex2}
\subsubsection{Comparison test}\label{S0704}  
\begin{theo}\label{P010}
$0 \leq a_{n} \leq b_{n}|_{n=\infty}$
\[ \text{If } \sum a_{n}|_{n=\infty}=\infty 
 \text{ diverges then } \sum b_{n}|_{n=\infty}=\infty \text{ diverges} \] 
\[ \text{If } \sum b_{n}|_{n=\infty}=0 \text{ converges then } \sum a_{n}|_{n=\infty}=0 \text{ converges}. \]
\end{theo}
\begin{proof}
If $0 \leq a_{n} \leq b_{n}|_{n=\infty}$
then
  $0 \leq \sum a_{n} \leq \sum b_{n}|_{n=\infty}$ \\
Case $\sum a_{n}|_{n=\infty}=\infty$ diverges, $0 \leq \infty \leq \sum b_{n}|_{n=\infty}$ then $\sum b_{n}|_{n=\infty}=\infty$ \\ 
Case $\sum b_{n}|_{n=\infty}=0$ converges $0 \leq \sum a_{n} \leq 0$ then $\sum a_{n}|_{n=\infty}=0$
\end{proof}

The aim of the comparison test is to find a sum 
 where 
 convergence or divergence is known, and compare
 against that sum, component wise.
\bigskip
\begin{mex2}\label{MEX049}
$\sum \frac{\mathrm{ln}\,n}{n^{2}}|_{n=\infty}$
 An inequality approach, an indirect
 and not as accessible for non-maths people
 because of the factorization. 

$\frac{ \mathrm{ln}\,n }{n^{2}}$  
 $=\frac{\mathrm{ln}\,n}{n^{\frac{1}{2}}} \frac{1}{n^{\frac{3}{2}}}|_{n=\infty}$.
 As $\mathrm{ln}\,n \prec n^{\frac{1}{2}}|_{n=\infty}$,
 $\frac{\mathrm{ln}\,n}{n^{\frac{1}{2}}}|_{n=\infty} = 0$,
 $\frac{\mathrm{ln}\,n}{n^{\frac{1}{2}}} \frac{1}{n^{\frac{3}{2}}} \leq \frac{1}{n^{\frac{3}{2}}}|_{n=\infty}$, 
 $0 \leq \frac{\mathrm{ln}\,n}{n^{2}} \leq \frac{1}{n^{\frac{3}{2}}}|_{n=\infty}$, 
  Summing the inequality at infinity, 
 $0 \leq \sum \frac{\mathrm{ln}\,n}{n^{2}} \leq \sum \frac{1}{n^{\frac{3}{2}}}|_{n=\infty}$, 
 $0 \leq \sum \frac{\mathrm{ln}\,n}{n^{2}}|_{n=\infty} \leq 0$, 
  $\sum \frac{\mathrm{ln}\,n}{n^{2}}|_{n=\infty}=0$
 converges.
\end{mex2}
\bigskip
\begin{mex2}\label{MEX050}
The same problem above could be solved with a different comparison,
 where we test against the known convergent p-series sums.
 Let $p \gt 1$,
 $\sum \frac{ \mathrm{ln}\,n }{n^{2}} \;\; z \;\; \sum \frac{1}{n^{p}}|_{n=\infty}$, 
 $\frac{ \mathrm{ln}\,n }{n^{2}} \;\; z \;\; \frac{1}{n^{p}}|_{n=\infty}$, 
 $\mathrm{ln}\,n \;\; z \;\; n^{2-p}|_{n=\infty}$. 
 When $2-p \gt 0$ then 
 $\mathrm{ln}\,n \prec n^{2-p}|_{n=\infty}$. Let $p=\frac{3}{2}$
 satisfies both conditions and the sum converges. 
 Not necessary, but just to demonstrate the theory is working,
 $\prec \; = \; \leq$, substituting back into the sum,
 $\sum \frac{ \mathrm{ln}\,n }{n^{2}} \leq \sum \frac{1}{n^{\frac{3}{2}}}|_{n=\infty}$, 
 $\sum \frac{ \mathrm{ln}\,n }{n^{2}} \leq 0 |_{n=\infty}$. 
\end{mex2}

In a variation of the comparison test,
 we can sandwich a series which may not be monotonic
 between two monotonic series with the same convergence.
 In the 
 sandwiched comparison test where either
 side of the test are monotonic sequences $(b_{n})|_{n=\infty}$ and
 $(c_{n})|_{n=\infty}$. Thereby extending the
 test to $\sum a_{n}|_{n=\infty}$ which may not be monotonic.
\bigskip
\begin{theo}\label{P011}
 If we can sandwich a non-monotonic sequence between two monotonic sequences,
 which either both converge or both diverge,
 we can determine the non-monotonic sequence's convergence.
\[ \text{Given } 0 \leq a_{n} \leq b_{n} \leq c_{n}|_{n=\infty}
 \text{ where } (a_{n})|_{n=\infty}
 \text{ and } (c_{n})|_{n=\infty}
 \text{ are monotonic sequences. } \]
\[ \text{If } \sum a_{n}|_{n=\infty} = \sum c_{n}|_{n=\infty} \text{ then }
  \sum b_{n} = \sum c_{n}|_{n=\infty} \]
\end{theo}
\begin{proof}
$0 \leq a_{n} \leq b_{n} \leq c_{n}|_{n=\infty}$
 then 
$0 \leq \sum a_{n} \leq \sum b_{n} \leq \sum c_{n}|_{n=\infty}$.
 Divergent case, $\infty \leq  \sum b_{n} \leq \infty$
 then $\sum b_{n}|_{n=\infty}=\infty$ diverges.
 Convergent case,
 $0 \leq  \sum b_{n} \leq 0$
 then $\sum b_{n}|_{n=\infty}=0$ converges.
\end{proof}
\subsubsection{nth term divergence test}\label{S0705}  
\begin{theo}\label{P013}
\[ \text{If } a_{n}|_{n=\infty} \neq 0 \Rightarrow \sum a_{n}|_{n=\infty}=\infty \text{ diverges.} \]
\end{theo}
\begin{proof}
By the negation of Theorem \ref{P014}, as when not equal to $0$, the negation of convergence is divergence,
 the sum diverges.
\end{proof}
\begin{theo}\label{P014}
If $\sum a_{n}|_{n=\infty}=0$ converges then $a_{n}|_{n=\infty}=0$
\end{theo}
\begin{proof}
 A sum of terms greater than or equal to
 zero is positive if their exists a term greater
 than zero. Since the sum is zero, their exists
 no such term, consequently $a_{n}|_{n=\infty}=0$.
\end{proof}
\bigskip
\begin{mex2}\label{MEX051}
$1-1+1-1+ \ldots$, $a_{n}|_{n=\infty}=(-1)^{n} \neq 0$ hence
 the series diverges.
\end{mex2}
\bigskip
\begin{mex2}\label{MEX033}
 An example of the second case, 
 determine convergence/divergence
 of $\sum_{n=1}^{\infty} \frac{n^3+n}{5 n^{3} + n^{2} +27}$.
 Since no division by zero, consider
 $\sum \frac{n^3+n}{5 n^{3} + n^{2} +27}|_{n=\infty}$.
 $a_{n}|_{n=\infty} = \frac{n^3+n}{5 n^{3} + n^{2} +27} |_{n=\infty}$
$= \frac{n^{3}}{5 n^{3}}|_{n=\infty}$
$= \frac{1}{5}$
$\neq 0$ 
 therefore divergent by nth term test.
 With the sum at infinity, simplifying 
 makes it clear
 why the sum diverges, without need to even refer
 to the nth term divergence test.
 $\sum \frac{n^3+n}{5 n^{3} + n^{2} +27}|_{n=\infty}$
 $= \sum \frac{n^{3}}{5 n^{3}} |_{n=\infty}$
 $= \sum \frac{1}{5}|_{n=\infty} = \infty$
\end{mex2}
\subsubsection{Absolute convergence test}\label{S0706}  
\begin{theo}\label{P015}
If $\sum | a_{n} | |_{n=\infty}=0$ converges then
  $\sum a_{n}|_{n=\infty}=0$ is convergent.
\end{theo}
\begin{proof}
Let $a_{k} \neq 0$. 
 $-|a_{k}| \leq a_{k} \leq |a_{k}|$,
 using $a_{k} = \mathrm{sgn}(a_{k})|a_{k}|$,
 where $\mathrm{sgn}(x) = \pm 1$ when $x \neq 0$.
 $-|a_{k}| \leq |a_{k}|\mathrm{sgn}(a_{k}) \leq |a_{k}|$,
 dividing by $|a_{k}|$,
 $-1 \leq \mathrm{sgn}(a_{k}) \leq 1$ which is always
 true,
 hence
 $-|a_{k}| \leq a_{k} \leq |a_{k}|$ is true.
 $n_{0}$, $n_{1} \in \mathbb{J}_{\infty}$;
 Summing the inequalities between two infinities $n_{0}$ and $n_{1}$,
 $-\sum_{k=n_{0}}^{n_{1}}|a_{k}| \leq \sum_{k=n_{0}}^{n_{1}} a_{k} \leq \sum_{k=n_{0}}^{n_{1}} |a_{k}|$. 
 
 The condition $\sum |a_{n}||_{n=\infty}$ is
 itself governed by Criterion E3.
 With a suitable choice of $n_{1}$(see Proposition \ref{P042}),
 $\sum_{k=n_{0}}^{n_{1}} |a_{n}|$
 $=\sum_{k=n_{0}} |a_{n}|$ as the sum converges, replace
 $n_{0}$ by $n$,
 $\sum_{n} a_{n}|_{n=\infty}=0$,
 $-\sum a_{n}|_{n=\infty}=0$,
 $\sum a_{n}|_{n=\infty}=0$.

 Consider the inequality,
 $-\sum_{k=n_{0}}^{n_{1}}|a_{k}| \leq \sum_{k=n_{0}}^{n_{1}} a_{k} \leq \sum_{k=n_{0}}^{n_{1}} |a_{k}|$, 
 $-\sum_{k=n_{0}}|a_{k}| \leq \sum_{k=n_{0}} a_{k} \leq \sum_{k=n_{0}} |a_{k}|$, 
 $-\sum_{n}|a_{k}| \leq \sum_{n} a_{k} \leq \sum_{n} |a_{k}||_{n=\infty}$, 
 $\sum|a_{n}| \leq -\sum a_{n} \leq -\sum |a_{n}||_{n=\infty}$, 
 $0 \leq -\sum a_{n}|_{n=\infty} \leq 0$, 
 $\sum a_{n}|_{n=\infty}=0$ converges.
\end{proof}

Theorem \ref{P015} is a sum rearrangement theorem
 at infinity, summing at infinity only.
 However it was used in (Section \ref{S16}) 
 to prove the more general sum
 theorem below.

If $\sum a_{\nu}$ is
 absolutely convergent and $\sum a_{\nu}'$ is
 an arbitrary rearrangement then $\sum a_{\nu} = \sum a_{\nu}'$ \cite[Theorem 4, p.79]{knopp}.
\subsubsection{Limit Comparison Theorem (LCT)}\label{S0707}  
\begin{theo}\label{P003}
  If  $\frac{ a_{n}}{b_{n}}|_{n=\infty}=c$
 and $c \neq 0$ and $c$ is a constant then
  $\sum a_{n} = \sum b_{n}|_{n=\infty}$. 
 Either both series converge or both diverge.
\end{theo}
\begin{proof}
$\frac{ a_{n}}{b_{n}}|_{n=\infty} = c$,  
$a_{n} = c b_{n}|_{n=\infty}$.   
 Apply summation to both sides,  
$\sum a_{n} = c \sum b_{n} |_{n=\infty}$.  
 Ignoring the constant as the sum either converges or diverges ($c \cdot \infty=\infty$ or $c \cdot 0=0$), 
$\sum a_{n} = \sum b_{n} |_{n=\infty}$.  Two possibilities.  
 Case 
 $\sum a_{n} = 0 \Leftrightarrow \sum b_{n} = 0 |_{n=\infty}$.  
 Case 
 $\sum a_{n} = \infty \Leftrightarrow \sum b_{n} = \infty|_{n=\infty}$ 
\end{proof}
\bigskip
\begin{mex2}\label{MEX028}  
$\sum_{k=0}^{n} \frac{5k+2}{k^{3}+1}|_{n=\infty}$,
 the limit comparison test assumes 
 the answer. If you have already worked out
 that the above sum tends to $\sum \frac{1}{n^{2}}$,
 then forming a limit is redundant.
 However you can verify the result by calculating
 the limit.

 Let $a_{n} = \frac{5n+2}{n^{3}+1}$,
 $b_{n} = \frac{1}{n^{2}}$.
$\sum \frac{1}{n^{2}}|_{n=\infty}=0$ converges,
 as this 
 is a p-series with $p=2 \gt 1$.
 
 $\frac{a_{n}}{b_{n}}|_{n=\infty}$
 $= \frac{5n+2}{n^{3}+1} n^{2}|_{n=\infty}$
 $=\frac{5n^{3}}{n^{3}}|_{n=\infty} = 5$.
 Since $\sum \frac{1}{n^{2}}|_{n=\infty}=0$ converges
 then $\sum a_{n}|_{n=\infty}=0$ converges.
 Alternatively see Example \ref{MEX027}. 
\end{mex2}

See Example \ref{MEX035}.
\subsubsection{Abel's test}\label{S0708}  

\begin{theo}\label{P002}  
Suppose $\sum b_{n}|_{n=\infty}=0$ converges
 and $( a_{n} )$ is a monotonic
 convergent sequence
 then $\sum a_{n} b_{n}|_{n=\infty}=0$ converges.
\end{theo}
\begin{proof}
Since $(a_{n})$ is a monotonic convergent sequence,
 then let $a_{n}|_{n=\infty}=a$, $a \prec \infty$.
 $a_{n}b_{n} = a b_{n}|_{n=\infty}$,
 then $\sum a_{n} b_{n}|_{n=\infty}$
 $= \sum a b_{n}|_{n=\infty}$
 $= a \sum b_{n}|_{n=\infty}=0$ converges.
\end{proof}
\begin{proof}
Since $(a_{n})|_{n=\infty}$ is convergent, the sequence
 is bounded above, say by $M$.
 $0 \leq a_{n}b_{n} \leq M b_{n}|_{n=\infty}$,
 $0 \leq \sum a_{n}b_{n} \leq \sum M b_{n}|_{n=\infty}$,
 $0 \leq \sum a_{n}b_{n} \leq M \sum b_{n}|_{n=\infty}$,
 $0 \leq \sum a_{n}b_{n} \leq 0|_{n=\infty}$,
 $\sum a_{n} b_{n}|_{n=\infty}=0$ converges.
\end{proof}
\subsubsection{L'Hopital's convergence test}\label{S0709}  
 See L'Hopital's convergence test (Section \ref{S180501}).
 Similarly for integrals.
\bigskip
\begin{conjecture}\label{P244}
 $f = \infty$,
 $g=\infty$
 and $\frac{f}{g}$ in indeterminate form.\\
 When $\frac{f}{g} \neq \frac{\mathrm{ln}_{w+1}}{ \prod_{k=0}^{w}\mathrm{ln_{k}}}|_{n=\infty}$ then
 \[ \sum \frac{f}{g} = \sum \frac{f'}{g'}|_{n=\infty} \] 
\end{conjecture}
\bigskip
\begin{mex2}\label{MEX052}
$\sum \frac{\mathrm{ln}\,n}{n^{2}}|_{n=\infty}$
 Since
 $\frac{\mathrm{ln}\,n}{n^{2}}|_{n=\infty}$ is in 
 $\infty/\infty$ form, differentiate by L'Hopital's rule, 
 $\sum \frac{\mathrm{ln}\,n}{n^{2}} |_{n=\infty}$
 $=\sum \frac{\frac{d}{dn} \mathrm{ln}\,n}{ \frac{d}{dn} n^{2}} |_{n=\infty}$
 $= \sum \frac{1}{n} \frac{1}{2n}|_{n=\infty}$
 $= \sum \frac{1}{n^{2}}|_{n=\infty}=0$ converges.
\end{mex2}
\bigskip
\begin{mex2}\label{MEX053}
Test $\sum_{n=1}^{\infty} \frac{n^2}{2^{n}}$ for convergence or divergence.
$\sum \frac{n^{2}}{2^{n}}|_{n=\infty}$
 $= \sum \frac{n^{2}}{e^{n\,\mathrm{ln}\,2}}|_{n=\infty}$
 $= \sum \frac{2n}{\mathrm{ln}\,2 e^{n\,\mathrm{ln}\,2}}|_{n=\infty}$
 $= \sum \frac{2}{(\mathrm{ln}\,2)^{2} e^{n\,\mathrm{ln}\,2}}|_{n=\infty}$
 $= \sum \frac{1}{2^{n}}|_{n=\infty}$
 $=0$ converges.
\end{mex2}
\bigskip
\begin{mex2}\label{MEX054}
 \cite[3.2.17.a, p.74]{kaczor}
 Chain rule with L'Hopital's rule.
$\sum \frac{1}{2^{n^{\frac{1}{2}}}}|_{n=\infty}$
$= \int \frac{dx}{2^{x^{\frac{1}{2}}}}|_{x=\infty}$,
 let $x = u^{2}$, $\frac{dx}{du} = 2u$, 
$\int \frac{dx}{2^{x^{\frac{1}{2}}}}|_{x=\infty}$
$= \int \frac{dx}{du} \frac{du}{2^{u}}|_{u=\infty}$
$= \int 2u \frac{du}{2^{u}}|_{u=\infty}$.
  [ $\frac{d}{du} 2^{u}$
 $= \frac{d}{du}e^{ u \,\mathrm{ln}\,2}$
 $= e^{ u \,\mathrm{ln}\,2} \cdot \mathrm{ln}\,2$
 $= 2^{u} \cdot \mathrm{ln}\,2$]
 As $\frac{u}{2^{u}}|_{u=\infty}=\frac{\infty}{\infty}$, apply L'Hopital's rule,
 $\int 2u \frac{du}{2^{u}}|_{u=\infty}$
$= \int 2 \frac{du}{2^{u} \cdot \mathrm{ln}\,2}|_{u=\infty}$
$= \int \frac{du}{2^{u}}|_{u=\infty}$
$=0$ converges.
\end{mex2}
\subsubsection{Alternating Convergence Test}\label{S0710}  
 Also called the Alternating Convergence Theorem (ACT),
 expressing the test at infinity. (See Theorem \ref{P206}) 
\bigskip
\begin{theo}\label{P007}
If $(a_{n})|_{n=\infty}$ is a monotonic
 decreasing sequence and 
 $a_{n}|_{n=\infty}=0$ then $\sum (-1)^{n} a_{n}|_{n=\infty}=0$ is convergent.
\end{theo}
\bigskip
\begin{mex2}\label{MEX029}
 \cite[3.4.25, p.96]{kaczor}
Determine convergence of $\sum (-1)^{n} \frac{ n! e^{n} }{ n^{n+p} }|_{n=\infty}$

 Recognizing the $n!$, rearrange Stirling's formula,
 $n! = (2 \pi n)^{\frac{1}{2}} (\frac{n}{e})^{n}|_{n=\infty}$,
 $\frac{n!}{n^{n}} =(2 \pi n)^{\frac{1}{2}} \frac{1}{e^{n}}|_{n=\infty}$.

 Substitute the rearranged expression into the sum.
 $\sum (-1)^{n} \frac{ n! e^{n}}{ n^{n+p} }|_{n=\infty}$
 $=\sum (-1)^{n} \frac{ n!}{n^{n}}  \frac{e^{n}}{ n^{p} }|_{n=\infty}$
 $= \sum (-1)^{n} (2 \pi)^{\frac{1}{2}} n^{\frac{1}{2}} \frac{1}{e^{n}} \frac{e^{n}}{n^{p}} |_{n=\infty}$
 $= \sum (-1)^{n} (2 \pi)^{\frac{1}{2}} \frac{1}{n^{p-1/2}} |_{n=\infty}$
 $= \sum (-1)^{n} \frac{1}{n^{p-1/2}} |_{n=\infty}$
 
 When $p=\frac{1}{2}$, 
 $\sum (-1)^{n}|_{n=\infty}=\infty$
 diverges.
 By ACT(Theorem \ref{P007}) when $p \gt \frac{1}{2}$ then $\frac{1}{n^{p-1/2}}|_{n=\infty}=0$ and
 the sum converges.
 When $p \lt \frac{1}{2}$, by the nth term
 test the sum diverges.
\end{mex2}
\subsubsection{Cauchy condensation test}\label{S0711}  
\begin{theo}\label{P006}
$\sum 2^{n}a_{2^{n}}|_{n=\infty}$ converges or diverges with $\sum a_{n}|_{n=\infty}$. \\
If $\sum 2^{n} a_{2^{n}}|_{n=\infty}=0$ then 
 $\sum a_{n}|_{n=\infty}=0$.
If $\sum 2^{n} a_{2^{n}}|_{n=\infty}=\infty$ then
 $\sum a_{n}|_{n=\infty}=\infty$.
\end{theo}
\begin{proof}
Convert to the continuous domain, apply
 the chain rule, and convert back
 to the discrete domain.

 $\sum 2^{n} a_{2^{n}} dn|_{n=\infty}$
 $= \int 2^{n} a(2^{n}) dn|_{n=\infty}$
 $= \int 2^{n} a(2^{n}) \frac{dn}{d2^{n}} d(2^{n}) |_{n=\infty}$
 $= \int 2^{n} a(2^{n}) \frac{dn}{d \, e^{n \mathrm{ln} \, 2}} d(2^{n}) |_{n=\infty}$
 $= \int 2^{n} a(2^{n}) \frac{1}{2^{n} \mathrm{ln} \, 2 } d(2^{n}) |_{n=\infty}$
 $= \int a(2^{n}) \frac{1}{ \mathrm{ln} \, 2 } d(2^{n}) |_{n=\infty}$
 $= \int a(2^{n}) d(2^{n}) |_{n=\infty}$
 $= \int a(n) dn |_{n=\infty}$
 $= \sum a_{n}|_{n=\infty}$.
 Constants ignored as either converge $0$ or diverge $\infty$.
\end{proof}
\bigskip
\begin{mex2}\label{MEX040}
Determine convergence/divergence of $\sum \frac{1}{n \,\mathrm{ln}\,n}|_{n=\infty}$.

Let $a_{n}=\frac{1}{n \,\mathrm{ln}\,n}$. 
 $a_{2^{n}} = \frac{1}{2^{n} \, \mathrm{ln}2^{n}}$
 $=\frac{1}{2^{n} \cdot n \cdot \ln\,2}$.
 Then 
 $\sum 2^{n} a_{2^{n}} dn|_{n=\infty}$
 $= \sum 2^{n} \frac{1}{2^{n} \cdot n \cdot \ln\,2}|_{n=\infty}$
 $=  \frac{1}{\mathrm{ln}\,2} \sum \frac{1}{n}|_{n=\infty}$
 $= \infty$ diverges.
\end{mex2}

\subsubsection{Ratio test}\label{S0712}  
\begin{theo}\label{P008}
 Theorem \ref{P211} 
 $a_{n} \in *G$;
\[ \text{If } (\overline{\mathbb{R}},\lt): \;\; \frac{a_{n+1}}{a_{n}} \lt 1 \text{ then } \sum a_{n}|_{n=\infty}=0 \text{ converges.} \]
\[ \text{If } (\overline{\mathbb{R}},\gt): \;\; \frac{a_{n+1}}{a_{n}} \gt 1 \text{ then } \sum a_{n}|_{n=\infty}=\infty \text{ diverges.} \]
\end{theo}
The ratio test can also be expressed as an
 inequality at infinity.
\bigskip
\begin{theo}\label{P009}
  Theorem \ref{P219}
 $a_{n} \in *G$;
\[ \text{If } (\overline{\mathbb{R}},\lt): \;\; a_{n+1} \lt a_{n}  \text{ then } \sum a_{n}|_{n=\infty}=0 \text{ converges.} \]
\[ \text{If } (\overline{\mathbb{R}},\gt): \;\; a_{n+1} \gt a_{n}  \text{ then } \sum a_{n}|_{n=\infty}=\infty \text{ diverges.} \]
\end{theo}

As the application of the ratio test
 at infinity is 
 in one-to-one correspondence with
 the same limit calculation,
 let us consider
 the 
 modified ratio test
 Theorem \ref{P009}
 examples.
\bigskip
\begin{mex2}\label{MEX030}
 Determine convergence of $\sum \frac{n!}{(2n)!}|_{n=\infty}$.
 Let $a_{n} = \frac{n!}{(2n)!}$. 
\begin{align*}
 a_{n+1} \;\; z \;\; a_{n}|_{n=\infty} \\
 \frac{(n+1)!}{(2(n+1))!} \;\; z \;\; \frac{n!}{(2n)!}|_{n=\infty} \\
 (n+1)! (2n)! \;\; z \;\; (2n+2)! n!|_{n=\infty} \\  
 (n+1) (2n)! \;\; z \;\; (2n+2)! |_{n=\infty} \\  
 (n+1) \;\; z \;\; (2n+1)(2n+2) |_{n=\infty} \\  
 n \lt 4n^{2}|_{n=\infty} \\
 1 \lt 4n|_{n=\infty}\tag{By Theorem \ref{P009} convergent} 
\end{align*}
\end{mex2}
\bigskip 
\begin{mex2}\label{MEX037}
A strict inequality in $*G$ is not a strict inequality in 
 $\mathbb{R}$ when infinitely close.
 Consider the known
 divergent sum $\sum \frac{1}{n}|_{n=\infty}$,
 with a strict inequality interpretation
 the test fails,
 $a_{n+1} \; z \; a_{n}|_{n=\infty}$,
 $\frac{1}{n+1} \; z \; \frac{1}{n}|_{n=\infty}$,
 $\frac{1}{n+1} \lt \frac{1}{n}|_{n=\infty}$, and the sum converges, which is incorrect.
 
However realizing the comparison, 
 $\frac{1}{n+1} \; z \; \frac{1}{n}|_{n=\infty}$,
$0 \; z \; 0$, $z = \; ==$ equality
 and the test is indeterminate. 
\end{mex2}
\bigskip
\begin{defy}\label{DEF005}
Define the radius of convergence $r$, 
 $\frac{1}{r} = |\frac{a_{n}}{a_{n-1}}||_{n=\infty}$
\end{defy}
\bigskip
\begin{mex2}\label{MEX038}
 By the ratio test with a point
 at infinity notation,
 find the radius of convergence and convergence interval for 
 $\sum_{n=1}^{\infty} \frac{ x^n}{ n^2 3^n }$.
 $a_{n} = \frac{1}{n^{2} 3^{n}}$,    
 $|\frac{a_{n}}{a_{n-1}}||_{n=\infty}$
$=| \frac{ (n-1)^2 3^{n-1}}{n^2 3^n} | |_{n=\infty}$
 $= \frac{1}{3} | \frac{ (n-1)^{2} }{n^{2}} \cdot \frac{ 3^{n-1}}{3^{n-1}} ||_{n=\infty}$
 $= \frac{1}{3}$  
 $= \frac{1}{r}$,  
$r = 3$. Converges when $x=(-3,3)$.

 Test the interval's end points.
 When $x=3$, $\sum \frac{ 3^{n}}{ n^{2} 3^{n}}|_{n=\infty}$
 $= \sum \frac{1}{n^{2}}=0$ converges.
 When $x=-3$,
 $\sum \frac{ (-1)^{n} 3^{n}}{ n^{2} 3^{n}}|_{n=\infty}$
 $= \sum \frac{(-1)^{n}}{n^{2}}=0$ converges.
 Interval of convergence: $x = [-3,3]$.

 Alternatively see 
 Example \ref{MEX014}.
\end{mex2}
\subsubsection{Cauchy's convergence test}\label{S0713}  
 The standard test.
 Applying Cauchy's
 convergent sequence test,
 with the partial sum as a general sequence term.
 If $\exists N: \forall n, m \gt N, |s_{m}-s_{n}| \lt \epsilon$ then $(s_{n})$ is a Cauchy sequence.

 This test is reformed at infinity:
 $s_{m}-s_{n}\in \Phi$ and ; $m,n \in \Phi^{-1}$; with the 
 condition $m-n \in \Phi^{-1}$  \cite[Part 6]{cebp21}.

 By considering the convergence sums as a sequence of
 points, if the sequence converges then the sum converges.
 As a partial sum which starts counting at infinity,
 $( \ldots, s_{n}, s_{n+1}, s_{n+2}, \ldots)|_{n=\infty}$.

 For the convergence test at infinity,
 we consider an infinite interval at infinity,
 and if the sum is an infinitesimal, the
 sum converges. 
\bigskip
\begin{theo}
 Consider the convergence sum, 
 let $s_{n}= \sum_{n} a_{k}|_{n=\infty}$;
 $m, n \in \Phi^{-1}$; $m-n=\infty$.
 If $s_{n}-s_{m} \in \Phi$
 then the sum $s_{n}$ converges, else the sum diverges.
\end{theo}
\begin{proof}
If the sum converges, then both sums satisfy the E3 criteria,
 then $s_{n}-s_{m} \in \Phi$ is a difference in infinitesimals,
 which is also an infinitesimal. $\Phi \mapsto 0$ and 
 the Cauchy sequence at infinity is satisfied.
\end{proof}

 The test forms the basis of Criterion E1 convergence.
 The example in \cite[pp.212--213]{gold} solves the same problem
 with a linear scale of infinities $(\omega, 2\omega, 3 \omega, \ldots)|_{\omega = \infty}$.
 If $s_{2 \omega} -s_{\omega} \in \Phi$ then $(s_{\omega})|_{\omega=\infty}$ converges else the sequence diverges.

 In constructing a Cauchy test at infinity,
 we can use infinities $2^{n}$ and $2^{n-1}$.
 Example \ref{MEX034} uses a power of $2$ scale of infinities $(2^{n}, 2^{n+1}, 2^{n+2}, \ldots)|_{n=\infty}$.
\bigskip
\begin{prop}\label{P047}
If  
 $s_{2^{n}} - s_{2^{n-1}}|_{n=\infty}=0$
 then $(s_{n})|_{n=\infty}$ converges,
 else the sum diverges.
\end{prop}
\begin{proof}
 Consider the E3 criterion.
 E.0: $2^{n}-2^{n-1}=\infty$ satisfied. If the sum is convergent,
 both the sums at a point are infinitesimals, their difference
 an infinitesimal $\Phi \mapsto 0$.  
\end{proof} 
\bigskip
\begin{mex2}\label{MEX034}
 Determine convergence/divergence of
 $s_{n} = \sum_{k=1}^{n} \frac{1}{k}$.
$s_{2^{n}} - s_{2^{n-1}}|_{n=\infty}$
 $=\int_{1}^{2^{n}} \frac{1}{x} dx - \int_{1}^{2^{n-1}} \frac{1}{x} dx$
 $= \mathrm{ln}\,x|_{1}^{2^{n}} - \mathrm{ln}\,x|_{1}^{2^{n-1}}|_{n=\infty}$
 $=n \, \mathrm{ln}\,2 - (n-1) \mathrm{ln}\,2|_{n=\infty}$
 $=\mathrm{ln}\,2 \neq 0$ diverges.
\end{mex2}

 However, Criteria E3 also do this 
 by integration at a point
 (see Theorems \ref{P044} and \ref{P057}).
 This could be considered as
 taking the Cauchy sequence a step further with magnitude
 arguments. 
\bigskip
\begin{theo}
 By convergence of a sequence at infinity
 \cite[Part 6]{cebp21},
 a convergence sum or integral need only have a point tested.
\end{theo}
\begin{proof}
 A partial sum $s_{n}$, where $|n-m|=\infty$,
 $s_{n}|_{n=\infty}-s_{m}|_{m=\infty}$
 $= \sum a_{n}|_{n=\infty} - \sum a_{m}|_{m=\infty}$
 $= \int a(n)\,dn|_{n=\infty}-\int a(m)\,dm|_{m=\infty}$
 $= \int a(n)\,dn|_{n=\infty}$ (divergent case)
 or $-\int a(m)\,dm|_{m=\infty}$ (convergent case),
 through the choice of the second variable and the E3 criteria,
 or diverges.

 Theorem \ref{P044} essentially does this in
 a simpler way, taking a slice of the tail with an infinite width
 in the domain and integrating. 
\end{proof}
\subsubsection{Dirichlet's test}\label{S0714}  
\begin{theo}\label{P004}
If $\sum b_{n}|_{n=\infty}=0$ converges
 and $(a_{n})|_{n=\infty}$ is positive
 and monotonically decreasing then
 $\sum a_{n} b_{n}|_{n=\infty}=0$ converges.
\end{theo}
\begin{proof}
Since $a_{n}$
 is positive and decreasing, $a_{n}$ is
 bounded above by a positive constant $\beta$.
 Let $a_{n} \leq \beta$,
 $0 \leq a_{n} \leq \beta$,
 $0 \leq a_{n} b_{n} \leq a_{n} \beta |_{n=\infty}$,
 $0 \leq \sum a_{n} b_{n} \leq \sum a_{n} \beta |_{n=\infty}$,
 but $\sum a_{n} \beta = \sum a_{n}|_{n=\infty} =0$ converges.
 By the sandwich principle,
 $0 \leq \sum a_{n} b_{n} \leq 0|_{n=\infty}$
 and $\sum a_{n} b_{n}|_{n=\infty}=0$ converges.
\end{proof}
\subsubsection{Bertrand's test*}\label{S0716}  
Bertrand's test \cite{bertrand} is included in the generalized ratio test.
\[
\frac{a_{n}}{a_{n+1}} = 1 + \frac{1}{n} + \frac{\rho_{n}}{n \, \mathrm{ln}\,n}, \;\; 
\rho_{n}|_{n=\infty} = \left\{ 
  \begin{array}{rl}
    \gt 1 & \; \text{then } \sum a_{n} \text{ is convergent,} \\
    \lt 1 & \; \text{ then } \sum a_{n} \text{ is divergent.}
  \end{array} \right. 
\]
\subsubsection{Raabe's tests*}\label{S0717}  
When $\frac{ a_{n+1}}{a_{n}}|_{n=\infty}=1$ the
 ratio test fails, then try Raabe's test. 
\bigskip
\begin{theo}\label{P023}
If $n (\frac{ a_{n}}{a_{n+1}}-1)|_{n=\infty} \gt 1$ 
 then $\sum a_{n}|_{n=\infty}=0$ converges.  
 If $n (\frac{ a_{n+1}}{a_{n}}-1)|_{n=\infty} \lt 1$
 then $\sum a_{n}|_{n=\infty}=\infty$ diverges.
\end{theo}
\bigskip
\begin{theo}\label{P024}
If $n a_{n} - (n+1) a_{n+1}|_{n=\infty} \gt 0$ then $\sum a_{n}=0$ converges.  
If $n a_{n} - (n+1) a_{n+1}|_{n=\infty} \lt 0$ then $\sum a_{n} = \infty$ diverges.
\end{theo}
\bigskip
\begin{mex2}\label{MEX031} 
 \cite[10.16.4]{apostol}
Determine convergence of $\sum \frac{3^{n}n!}{n^{n}}|_{n=\infty}$.

 Let $a_{n} = \frac{3^{n}n!}{n^{n}}$.
 The application of the Ratio test fails.
 $\frac{a_{n+1}}{a_{n}}|_{n=\infty}$
 $= \frac{3^{n+1}(n+1)!}{(n+1)^{n+1}} \frac{n^{n}}{3^{n}n!}|_{n=\infty}$
 $= 3(n+1) \frac{n^{n}}{(n+1)^{n}}|_{n=\infty}$
 $=1$
\begin{align*}
n a_{n} - (n+1)a_{n+1}|_{n=\infty} \; z \; 0 & \tag{Try Raabe's Theorem \ref{P024} } \\
n \frac{ 3^{n} n! }{n^{n}} - \frac{ (n+1) 3^{n+1} (n+1)! }{ (n+1)^{n+1} }|_{n=\infty} \; z \; 0 & \\
\frac{ n! }{n^{n-1}} - \frac{ 3 n! }{ (n+1)^{n-1} }|_{n=\infty} \; z \; 0 & \\
(\frac{n+1}{n})^{n-1}|_{n=\infty} -  3 \; z \; 0 & \\
 -2 \lt 0 & \tag{Converges by Theorem \ref{P024}}
\end{align*}
\end{mex2}
\begin{mex2}\label{MEX039}
 \cite[3.2.16]{kaczor}
\begin{theo}\label{P050}
If 
 $\lim\limits_{n \to \infty} n \; \mathrm{ln} \frac{a_{n}}{a_{n+1}} = g$,  
 show 
 $g \gt 1 \Rightarrow$ convergence and
 $g \lt 1 \Rightarrow$ divergence.
\end{theo}
\begin{proof}
Consider the case 
 $n \; \mathrm{ln} \frac{a_{n}}{a_{n+1}} \gt 1|_{n=\infty}$, 
 $\mathrm{ln} \frac{a_{n}}{a_{n+1}} \gt \frac{1}{n}|_{n=\infty}$, 
 $\frac{a_{n}}{a_{n+1}} \gt e^{\frac{1}{n}} |_{n=\infty}$,  
 $a_{n} \gt a_{n+1} e^{\frac{1}{n}} |_{n=\infty}$  
 Substitute 
$e = (\frac{n+1}{n})^{n}|_{n=\infty}$
 into the inequality,  
 $a_{n} \gt a_{n+1} ((\frac{n+1}{n})^{n})^{\frac{1}{n}}|_{n=\infty}$, 
 $a_{n} \gt a_{n+1} \frac{n+1}{n}|_{n=\infty}$, 
 $n a_{n} - (n+1) a_{n+1} \gt 0 |_{n=\infty}$. 
This is Raabe's convergence test Theorem \ref{P024},
 and hence
$n \; \mathrm{ln} \frac{a_{n}}{a_{n+1}} \gt 1|_{n=\infty} \Rightarrow \sum a_{n}$ is convergent.

For the divergent case, after a similar substitution, 
 $a_{n} \lt a_{n+1} ((\frac{n+1}{n})^{n})^{\frac{1}{n}}|_{n=\infty}$,
 $a_{n} \lt a_{n+1} \frac{n+1}{n}|_{n=\infty}$, 
 $n a_{n} - (n+1) a_{n+1} \lt 0$, is Theorem \ref{P024}, divergent
 case.
\end{proof}
\end{mex2}
\subsubsection{Generalized p-series test}\label{S0718}  
 See Theorem \ref{P230},
 Known results.
\bigskip
\begin{defy}\label{DEF001}
Let 
 $\sum \frac{1}{\prod_{k=0}^{w-1} \mathrm{ln}_{k} \cdot \mathrm{ln}_{w}^{p} }$
 and the corresponding integral 
 be called the generalized p-series. 
\end{defy}
\bigskip
\begin{theo}\label{P036}
$\sum \frac{1}{\prod_{k=0}^{w-1} \mathrm{ln}_{k} \cdot \mathrm{ln}_{w}^{p} }|_{n=\infty}$ 
 and 
 $\int \frac{1}{\prod_{k=0}^{w-1} \mathrm{ln}_{k} \cdot \mathrm{ln}_{w}^{p} }\,dn|_{n=\infty}$
diverge
 when $p \leq 1$
 and converges 
 when $p \gt 1$.
\end{theo}
\subsubsection{Generalized ratio test}\label{S0719}  
\begin{theo}\label{P037}
Includes the ratio test, Raabe's test, Bertrand's test.
 See Section \ref{S17} and Theorem \ref{P225}.
\[
\frac{a_{n}}{a_{n+1}} - (1 + \frac{1}{n} + \frac{1}{n\,\mathrm{ln}\,n} + \ldots + \frac{1}{n \, \mathrm{ln}\,n \, \ldots \mathrm{ln}_{k}\,n })|_{n=\infty} = \left\{ 
  \begin{array}{rl}
    \gt 0 & \; \text{then } \sum a_{n}|_{n=\infty} \text{ is convergent,} \\
    \leq 0 & \; \text{ then } \sum a_{n}|_{n=\infty} \text{ is divergent.}
  \end{array} \right. 
\]
\end{theo}
\subsubsection{Boundary test}\label{S0720}   
\begin{theo}\label{P041}
See Theorem \ref{P232}. Let $w$ be a fixed integer. Solve for relation $z$. 
\[ 
\sum a_{n} \; z \; \sum \frac{1}{\prod_{k=0}^{w} \mathrm{ln}\,n}|_{n=\infty},\;\; z = \left\{
  \begin{array}{rl}
    \lt & \; \text{then } \sum a_{n}|_{n=\infty} \text{ is convergent,} \\ 
    \geq & \; \text{then } \sum a_{n}|_{n=\infty} \text{ is divergent.} 
  \end{array} \right.
\]
\end{theo}
\subsubsection{nth root test*}\label{S0721}   

\begin{theo}\label{P026}
If $|a_{n}|^{\frac{1}{n}}|_{n=\infty} \lt 1$ then
 $\sum a_{n}|_{n=\infty}=0$ converges. (See Theorem \ref{P234})
\end{theo}
\subsection{Miscellaneous}\label{S13}
\subsubsection{Transference between sums and convergence sums}\label{S1301}
 When determining convergence or divergence
 with convergences sums
 we actually do a transfer \cite[Part 4]{cebp21} from an interval to
 a point in their construction.
\[ \sum_{n=n_{0}}^{\infty} a_{n} \mapsto \sum a_{n}|_{n=\infty} \]

 After determining convergence or
 divergence at infinity, we may need to
 translate the ``convergence sum" back into a sum.

 This is of course just reversing
 the direction which we previously used
 to solve for the sum's convergence or
 divergence.
\bigskip
\begin{mex}
 $\sum \frac{1}{n^{2}}|_{n=\infty}=0 \mapsto \sum_{k=1}^{\infty} \frac{1}{n^{2}}$ converges. 
\end{mex}

 The E3 criteria uses an infinite
 section of the tail (integration between two infinities) at infinity to
 determine convergence or divergence.
 However, this is enough to determine the
 whole infinite tail's convergence or divergence.
 
 In determining convergence or divergence,
 we take a sum and consider the sum or integral
 at infinity.
 The reverse is possible, where we take
 a sum or integral at infinity and
 construct a sum from a sequence.

 Provided the sequence is monotonic
 and does not contain singularities in $*G$, 
 the same convergence or divergence properties 
 are retained.
 
 This is an example of a transfer from a point
 to an interval,
 as we extend from one space into another.
\bigskip
\begin{theo}\label{P090} Transference from ``convergence sums" to sums.
\[ \sum a_{n}|_{n=\infty} \mapsto \sum_{k=k_{0}}^{\infty} a_{k} \]
\[ \int a(n)\,dn|_{n=\infty} \mapsto \int_{x_{0}}^{\infty} a(n)\,dn \]
\end{theo}
\begin{proof}
 By Theorems \ref{P091}, \ref{P092}, \ref{P093}.
\end{proof}
\bigskip
\begin{theo}\label{P091}
 If $\sum a_{n}|_{n=\infty}=0$ converges and 
 $(a_{k})|_{k=k_{0}}^{\infty}$ exists and is not
 an infinity
 then $\sum_{k=k_{0}}^{\infty} a_{k}$ converges.
\end{theo}
\begin{proof}
 Since $(a_{k})$ does not diverge, a finite sum
 of its terms do not diverge.
 Since the sum of the tail is an infinitesimal,
 then we can construct the stated sum,
 by Theorem \ref{P028}. 
\end{proof}
\bigskip
\begin{theo}\label{P092}
 If $\int a(x)\,dx|_{n=\infty}=0$ converges
 and function $a(x)$ exists and is not
 an infinity
 then
 $\int_{x_{0}}^{\infty} a(x)\,dn$ converges.
\end{theo}
\begin{proof}
 Since the continuous function
 $a(x)$ does not diverge, its finite integral 
 does not diverge.
 Since the integral of the tail is an infinitesimal,
 then we can construct the stated integral,
 by Theorem \ref{P052}. 
\end{proof}
\bigskip
\begin{theo}\label{P093}
If $\sum a_{n}|_{n=\infty}=\infty$
 or $\int a(n)\,dn|_{n=\infty}=\infty$
 diverge we can construct a respective 
 sum $\sum_{k=k_{0}}^{\infty} a_{k}$
 or integral $\int_{x_{0}}^{\infty} a(x)\,dx$ that diverges.
\end{theo}
\begin{proof}
 A point that does not exist leads to a diverging
 sum or integral.
 If the sequence or interval exist,
 then a sum or integral with a finite part and an infinite
 part can be constructed. Since the tail is an infinity,
 and the finite part of the sum added to the tail is
 still an infinity, hence as expected, the sum or integral will
 diverge.
\end{proof}
\subsubsection{Convergence rates} \label{S1302}
 We can show the theory of convergence sums includes error analysis.
\bigskip
\begin{defy}
 Rate of convergence of positive series is the ratio of the
 partial sums.
\end{defy}
\bigskip
\begin{theo}
 Rate of convergence of the positive convergent
 sum $\sum a_{n}|_{n=\infty}$ is $\frac{a_{n+1}}{a_{n}}|_{n=\infty}$.
\end{theo}
\begin{proof}
 Let $s_{n} = \sum a_{n}|_{n=\infty}$ be a convergent sum.
 $\frac{s_{n+1}}{s_{n}}$
 $= \frac{ \sum a_{n+1}}{ \sum a_{n}}|_{n=\infty}$
 $= \frac{ \int a(n+1)\,dn}{ \int a(n)\,dn}|_{n=\infty}$
 is of the form $\frac{0}{0}$ as 
 $\sum a_{n}|_{n=\infty}=0$ is convergent.
 Use L'Hopital's \cite[Part 5]{cebp21} rule to differentiate.
 $\frac{ \int a(n+1)\,dn}{ \int a(n)\,dn}|_{n=\infty}$
 $= \frac{a(n+1)}{a(n)}|_{n=\infty}$
 $= \frac{ a_{n+1}}{a_{n}}|_{n=\infty}$.
\end{proof}
\bigskip
\begin{mex}
 Determine the rate of convergence of 
 $\sum \frac{ (4n)!(1103+26390n)}{ (n!)^{4} 396^{4n} }|_{n=\infty}$
 \cite{ramanujanpi}.

Let $a_{n} = \frac{ (4n)!(1103+26390n)}{ (n!)^{4} 396^{4n} }$,
 $\frac{ a_{n+1}}{a_{n}}|_{n=\infty}$
 $=\frac{ (4(n+1))!(1103+26390(n+1))}{ ((n+1)!)^{4} 396^{4(n+1)} } \frac{ (n!)^{4} 396^{4n} }{ (4n)! (1103 + 26390n) }|_{n=\infty}$
 $=\frac{ (4(n+1))!}{ ((n+1)!)^{4} 396^{4} } \frac{ (n!)^{4} }{ (4n)! }|_{n=\infty}$
 $= \frac{ (4n+4)(4n+3)(4n+2)(4n+1) (4n)! }{ (n+1)^{4} (n!)^{4} 396^{4} } \frac{ (n!)^{4} }{ (4n)! }|_{n=\infty}$
 $= \frac{4^{4}}{396^{4}}$
 $= \frac{1}{99^{4}}$
 $= 1.041020\times{10}^{-8}$ Hence eight decimal digits per
 iteration. 
\end{mex}
\section{Power series convergence sums} \label{S14}
 Calculating the radius and interval of
 convergence with power series at infinity.
 By using non-reversible arithmetic,
 either by factoring, comparison or
 application of the logarithmic magnitude relation, convergence or divergence may be determined.
 We interpret uniform convergence with a convergence sum.
\subsection{Introduction}\label{S1401}
 While convergence testing for power series is 
 straight forward,
 we mirror the tests with convergence sums.
 The theory is general as it calculates in a different
 way the radius of convergence,
 intervals and theorems. 

 In power series convergence sums
 we find application of non-reversible multiplication,
 Theorem \ref{P032}.

 `Convergence sums' theory extensively uses
 power series at infinity.
 By threading a continuous
 curve through a monotonic sequence and interchanging between
 the continuous form and the sequence.
 This one idea leads to the integral test in both
 directions (Theorem \ref{P065}).  
 We also use power 
 series at infinity to describe a derivative of a sequence (Section \ref{S15}). 

 The power series representation at
 infinity is interesting because historically
 the power series has played a role in applications,
 continuity,
 uniform convergence, limit interchanges, partial differentiation,
 solution validity,
 and many other matters. 

 For example, we can represent a trigonometric function
 at infinity. As power series are analytic, the property
 is applicable over the infinite domain too. 

 With power series, a generalization of the geometric series
 is extensively used for function representation
 and approximation. Fourier series, partial differential
 equations and other applied topics also 
 appear in number theory of partitions with generating functions.

 It happens that simplifying a sum at infinity, by reasoning of 
 magnitude of $\sum a_{n} x^{n}|_{n=\infty}$,
 is a different experience,
 and is another way of determining convergence.
 The reasoning is often algebraic,
 arguing with magnitudes and factoring. 
\subsection{Finding the radius of convergence} \label{S1402}
 A power series is a geometric series;
 we know that $1+x +x^{2} + \ldots$ is convergent
 when $|x| \lt 1$. The convergence can also be derived at infinity
 by comparing against a convergent p-series.
 Intuitively a fraction less than one multiplied 
 by another fraction less than one infinitely many times,
 is infinitely small.
\bigskip
\begin{theo}\label{P027}
If $|x| \lt 1$ then
 $\sum x^{n}|_{n=\infty}=0$ converges, and has radius of convergence $r=1$.  
\end{theo}
\begin{proof}
 Comparing against the convergent p-series.
$\sum x^{n} \; z \; \sum \frac{1}{n^{\alpha}}|_{n=\infty}$
 converges when $\alpha \gt 1$.
 $x=0$ is a solution.
 Solving for $x$,
 $x^{n} \; z \; \frac{1}{n^{\alpha}}|_{n=\infty}$,
 $x^{n} n^{\alpha} \; z \; 1|_{n=\infty}$, 
 $\mathrm{ln}(x^{n} n^{\alpha}) \; (\mathrm{ln}\,z) \; \mathrm{ln}\,1|_{n=\infty}$, 
 $n \, \mathrm{ln}\,x + \alpha  \, \mathrm{ln}\,n \; (\mathrm{ln}\,z) \; 0|_{n=\infty}$, 
 $n \, \mathrm{ln}\,x \; (\mathrm{ln}\,z) \; 0|_{n=\infty}$, 
 $\mathrm{ln}\,x^{n} \; (\mathrm{ln}\,z) \; 0|_{n=\infty}$, 
 $x^{n} \; z \; e^{0}|_{n=\infty}$, 
 $x^{n} \; z \; 1|_{n=\infty}$, 
 $z = \;\lt$ 
 then $|x| \lt 1$ .
 (Solving for $z = \; \leq$ leads
 to $x=1$ which in the sum diverges hence
 this case is excluded).
\end{proof}
\bigskip
\begin{prop}\label{P033}
$\sum x^{n}|_{n=\infty}$ diverges when $|x| \gt 1$.
\end{prop}
\begin{proof}
 For convergence, $\sum a_{n}|_{n=\infty}$ requires
 $a_{n}|_{n=\infty}=0$. When $|x| \gt 1$ then $x^{n}|_{n=\infty} \neq 0$.
\end{proof}
\bigskip
\begin{defy}\label{DEF009}
 The \textit{radius of convergence} is absolute convergence of
 $\sum a_{n} x^{n}|_{n=\infty}$,
 solving Theorem \ref{P029}, $|a_{n}^{\frac{1}{n}}r||_{n=\infty} \lt 1$
 about the origin. $x$ is absolutely convergent about the
 origin within $(-r,r)$.
\end{defy}
\bigskip
\begin{defy}
 The \textit{interval of convergence} includes the radius
 of convergence, and the end points which need to be 
 tested separately.  
\end{defy}
 
A power series convergence test,
 Theorem \ref{P029} transforms
 the series at infinity to evaluate
 the radius of convergence, a distance
 about which the sum converges.
 Unimportant terms in the sums product,
 which are not required
 to determine convergence or divergence,
 become transients.
 Applying 
 non-reversible arithmetic, these variables and
 constants can be removed.

 Since a power series about a point can be translated
 to the origin, the calculation
 of the radius of convergence and
 the interval of convergence may be applied
 to infinite series of the form $\sum_{k=1}^{\infty} a_{k}(x-c)^{k}$.

 While solving absolute convergence finds the
  general interval,
 the end points
 of the interval need to be tested separately
 for the 
 interval of convergence \cite[Properties of functions represented by power series, p.431]{apostol}.
\bigskip
\begin{mex}
Determine the radius of convergence of
 $\sum n (\frac{x}{2})^{n}|_{n=\infty}=0$.
 
 Transform the power series by
 bringing the $n$ term into the product and simplifying.
\begin{align*}
 \sum n (\frac{x}{2})^{n}|_{n=\infty} \\
 = \sum (n^{\frac{1}{n}} \frac{x}{2})^{n}|_{n=\infty} & \tag{$n^{\frac{1}{n}}|_{n=\infty}=1$} \\
 = \sum (\frac{x}{2})^{n}|_{n=\infty}=0 \tag{for convergence} \\
 |\frac{x}{2}| \lt 1, \;\; |x| \lt 2 & \tag{Theorem \ref{P033}} \\
 \text{radius of convergence } r=2
\end{align*}
\end{mex}
\bigskip
\begin{theo}\label{P200}
Transform
 the sum  
 $\sum a_{n}x^{n} = \sum (b_{n}x)^{n}|_{n=\infty}$.
 For convergence 
 $\sum (b_{n}x)^{n}|_{n=\infty}=0$, 
 solving for $|b_{n}x| \lt 1$, 
 the radius of convergence $r = \frac{1}{|b_{n}|}|_{n=\infty}$
 If $r$ exists $x=(-r,r)$ converges. For the interval of convergence
 the end points need to be tested.
\end{theo}
\begin{proof}
 Let $b_{n} = a_{n}^{\frac{1}{n}}$,
$\sum a_{n} x^{n}|_{n=\infty}$
$= \sum (a_{n}^{\frac{1}{n}} x)^{n}|_{n=\infty} = 0$
 by Theorem \ref{P027} when $|a_{n}^{\frac{1}{n}}x||_{n=\infty} \lt 1$

The end points $b_{n}x|_{n=\infty}=1$ 
 evaluated separately using the Criteria E3 (Section \ref{S0108}). 
\end{proof}

 A primary technique in simplifying products is
 to apply the inverse log and exponential functions,
 then non-reversible arithmetic.
\[ ab = e^{\mathrm{ln}(ab)} = e^{\mathrm{ln}\,a+\mathrm{ln}\,b} = e^{\mathrm{ln}\,a} \text{ when } \mathrm{ln}\,a \succ \mathrm{ln}\,b \]
\begin{mex}
 $\sum n (\frac{x}{2})^{n}|_{n=\infty}$
 $= \sum e^{ \mathrm{ln}\,n + n \,\mathrm{ln}\,\frac{x}{2} }|_{n=\infty}$
 $= \sum e^{n \,\mathrm{ln}\,\frac{x}{2} }|_{n=\infty}$
 $= \sum (\frac{x}{2})^{n}|_{n=\infty}$, $|\frac{x}{2}|\lt 1$, $|x| \lt 2$,
 radius of convergence $r=2$.
\end{mex}

By application of `logarithmic magnitude' 
 we can directly simplify the product.
 Given positive functions and relation $f \; z \; g$,
 when $\mathrm{ln}\,f \prec \mathrm{ln}\,g$ then
 by definition we say $f \prec\!\prec g$. We then apply a non-reversible product
 theorem, if $a \prec\!\prec b$ then $ab = b$ \cite[Part 5]{cebp21}.
 For power series, this allows
 the simplification of product terms. 
\bigskip
\begin{theo}\label{P032}
For positive $a$ and $b$, if $a \succ\!\succ b|_{n=\infty}$ then
 $\sum ab = \sum a|_{n=\infty}$\
\end{theo}
\begin{proof}
$(\sum ab$
 $ = \sum e^{\mathrm{ln}(ab)}$
 $= \sum e^{\mathrm{ln}\,a+\mathrm{ln}\,b}$
 $= \sum e^{\mathrm{ln}\,a}$
 $= \sum a)|_{n=\infty}$
  since $a \succ\!\succ b$ means $\mathrm{ln}\,a \succ \mathrm{ln}\,b$.
\end{proof}
\bigskip
\begin{mex}\label{MEX041}
 As $\left( \frac{x}{2} \right)^{n} \succ\!\succ n|_{n=\infty}$,
 $\sum n (\frac{x}{2})^{n}|_{n=\infty}$
 $= \sum (\frac{x}{2})^{n}|_{n=\infty}$.

To establish the logarithmic magnitude relationship,
 solve the comparison.   
 $n \; z \; (\frac{x}{2})^{n}|_{n=\infty}$,
 $\mathrm{ln}\,n \; (\mathrm{ln}\,z) \; n\, \mathrm{ln}\frac{x}{2}|_{n=\infty}$,
 $(\mathrm{ln}\,z) = \; \prec$
 then by definition $z = \; \prec\!\prec$.
\end{mex}

Smaller infinities in the product/division may be simplified, 
 thereby making the sum easier to solve for convergence.
 It is not always easy to identify
 the dominant term.
 From \cite[Part 5]{cebp21}, products can
 be converted to sums by taking 
 the logarithm, and solving the
 relation.
\bigskip
\begin{mex}\label{MEX013} 
Show $\sum c n^{p} x^{n} = \sum x^{n}|_{n=\infty}$.
  $\sum c n^{p} x^{n}|_{n=\infty}$
 $= \sum e^{ \mathrm{ln}(c n^{p} x^{n})}|_{n=\infty}$
 $= \sum e^{ \mathrm{ln}\,c + p\, \mathrm{ln}\,n + n\, \mathrm{ln}\,x }|_{n=\infty}$
 $= \sum e^{ n \,\mathrm{ln}\,x }|_{n=\infty}$
 $= \sum x^{n}|_{n=\infty}$,
 because $n\, \mathrm{ln}\,x \succ p\, \mathrm{ln}\,n +\mathrm{ln}\,c|_{n=\infty}$. 
\end{mex}
\bigskip
\begin{prop}\label{P031}
When $p$ and $c$
  are constant then $\sum c a_{n} n^{p} x^{n}$
 $= \sum a_{n} x^{n}|_{n=\infty}$ 
\end{prop}
\begin{proof}
 By similar argument 
 to Example \ref{MEX013}.
 $\sum c a_{n} n^{p} x^{n}|_{n=\infty}$
 $=\sum e^{\mathrm{ln}(c a_{n} n^{p} x^{n})}|_{n=\infty}$.
 $\mathrm{ln}(c a_{n} n^{p} x^{n})|_{n=\infty}$
 $= \mathrm{ln}\,c + \mathrm{ln}\,a_{n} +p\,\mathrm{ln}\,n + n \,\mathrm{ln}\,x|_{n=\infty}$
 $= \mathrm{ln}\,a_{n} + n \,\mathrm{ln}\,x|_{n=\infty}$
 $= \mathrm{ln}a_{n}x^{n}|_{n=\infty}$,
 as $n \,\mathrm{ln}\,x \succ \mathrm{ln}\,c|_{n=\infty}$ and
 $n \,\mathrm{ln}\,x \succ p\,\mathrm{ln}\,n|_{n=\infty}$.
 Reversing the exponential and logarithmic operations,
 $\sum e^{\mathrm{ln}(c a_{n} n^{p} x^{n})}|_{n=\infty}$
 $= \sum e^{ \mathrm{ln}a_{n}x^{n}}|_{n=\infty}$
 $= \sum a_{n}x^{n}|_{n=\infty}$
\end{proof}
\bigskip
\begin{mex}\label{MEX200}  
Find the radius and interval of convergence for 
 $\sum_{n=1}^{\infty} \frac{ (x-5)^{n}}{ (n+2) 3^{n}}$.
\begin{align*}
 \sum \frac{ (x-5)^{n}}{ (n+2) 3^{n}}|_{n=\infty} \tag{Need only consider the point at infinity} \\
 = \sum \frac{ (x-5)^{n}}{ 3^{n}}|_{n=\infty} \tag{as $3^{n} \succ\!\succ n+2|_{n=\infty}$} \\
 = \sum (\frac{x-5}{ 3})^{n}|_{n=\infty}=0 \tag{for convergence} \\
 | \frac{x-5}{3}| \lt 1 \\
 |x-5| \lt 3 \tag{radius $r=3$} \\ 
 2 \lt x \lt 8
\end{align*}
Test the intervals end points. 
 Case
 $x=2$,
 $\sum (\frac{ 2-5}{3})^{n}|_{n=\infty}$
 $= \sum (-1)^{n}|_{n=\infty} = \infty$ diverges.
 Case 
 $x=8$,
 $\sum (\frac{ 8-5}{3})^{n}|_{n=\infty}$
 $= \sum 1|_{n=\infty}=\infty$ diverges.
 Radius of convergence is $3$, interval of convergence
 is $x=(2,8)$.
\end{mex}

Basic arithmetic is used in solving these problems,
 $a^{bc} = (a^{b})^{c} = (a^c)^{b}$.
 Where the raised powers are 
 interchanged.
 The best way is to evaluate
 the triple from the base upwards.
 E.g. 
 $2^{15}=2^{3 \cdot 5} = (2^3)^5 = (2^5)^3 = 32768$.
 If we write without correct bracketing,
 the order can be ambiguous,
 evaluating from the top down, 
 $2^{3^{5}} = 2^{(3^5)} = 2^{243}$,
 $2^{5^{3}} = 2^{(5^{3})} = 2^{125}$.
\bigskip
\begin{mex}\label{MEX015}  
 \cite[11.7.10]{apostol}. Determine the radius of convergence for   
 $\sum_{n=1}^{\infty} 3^{n^{\frac{1}{2}}} \frac{z^{n}}{n}$. \; 
 $\sum 3^{n^{\frac{1}{2}}} \frac{z^{n}}{n}|_{n=\infty}$ 
 $=\sum (3^{\frac{1}{2}})^{n} \frac{z^{n}}{n}|_{n=\infty}$ 
 $=\sum (3^{\frac{1}{2}}z)^{n} \frac{1}{n}|_{n=\infty}$ 
 $=\sum (3^{\frac{1}{2}}z)^{n}|_{n=\infty}$ 
 when $|3^{\frac{1}{2}}z| \lt 1|_{n=\infty}$,
 $|z| \lt \frac{1}{3^{\frac{1}{2}}}$,
 $r= \frac{1}{3^{\frac{1}{2}}}$
\end{mex}
\bigskip
\begin{mex}\label{MEX016}  
\cite[11.7.12 ]{apostol}. 
Determine the radius of convergence for   
 $\sum_{n=1}^{\infty} (1+\frac{1}{n})^{n^{2}} z^{n}$.
 $\sum (1+\frac{1}{n})^{n^{2}} z^{n}|_{n=\infty}$
 $=\sum ((\frac{n+1}{n})^{n} z)^{n}|_{n=\infty}$
 $=\sum (e z)^{n}|_{n=\infty}=0$
 when $|ez| \lt 1$, $r=\frac{1}{e}$
\end{mex}

By Stirling's formula we know $(n!)^{\frac{1}{n}}|_{n=\infty}=e^{-1}n$.
 This can be used in determining radius of convergence
 with factorial expressions.
\bigskip
\begin{mex}\label{MEX017}  
\cite[Example 5, p.795]{kreyszig}.
Determine the radius of convergence.    
 $\sum \frac{ (2n)! }{(n!)^{2}} y^{n}|_{n=\infty}$
 $=\sum (\frac{ ((2n)!)^{\frac{1}{n}} }{(n!)^{\frac{2}{n}}} y )^{n}|_{n=\infty}$
 $=\sum ( \frac{ ( ((2n)!)^{\frac{1}{2n}} )^{2} }{((n!)^{\frac{1}{n}})^{2}} y )^{n}|_{n=\infty}$
 $=\sum ( \frac{ (2n)^{2} }{ n^{2} } y )^{n}|_{n=\infty}$
 $=\sum ( 4 y )^{n}|_{n=\infty} =0$
 when $|4y| \lt 1$, $|y| \lt \frac{1}{4}$, $r=\frac{1}{4}$
\end{mex}

\bigskip
\begin{mex}\label{MEX020} 
 \cite[3.3.7.c, p.98]{kaczor2}, given $\sum a_{n} x^{n}|_{n=\infty}$
 has radius of convergence $R$, $R \prec \infty$.
 Solve the radius of convergence $r$ for
 $\sum \frac{n^{n}}{n!} a_{n} x^{n}|_{n=\infty}$.

 Solving for $R$. 
 $\sum a_{n} x^{n}|_{n=\infty}$
 $= \sum (a_{n}^{\frac{1}{n}} x)^{n}|_{n=\infty}=0$
 when
 $|a_{n}^{\frac{1}{n}}x| \lt 1|_{n=\infty}$, 
 $|x| \lt \frac{1}{ |a_{n}^{\frac{1}{n}}| }|_{n=\infty}$,
 $R = \frac{1}{ |a_{n}^{\frac{1}{n}}| }|_{n=\infty}$.
 $\sum \frac{n^{n}}{n!} a_{n} x^{n}|_{n=\infty}$
 $= \sum (\frac{n}{(n!)^{\frac{1}{n}}} a_{n}^{\frac{1}{n}} x)^{n}|_{n=\infty}$
 $= \sum (e a_{n}^{\frac{1}{n}} x)^{n}|_{n=\infty}=0$ 
 when $|e a_{n}^{\frac{1}{n}} x| \lt 1|_{n=\infty}$,  
 $e |x| \lt R$,  
 $|x| \lt \frac{R}{e}$ radius of convergence $r=\frac{R}{e}$ 
\end{mex}
\bigskip
\begin{mex}\label{MEX021} \cite[3.3.7.d, p.98]{kaczor2} 
 given $\sum a_{n} x^{n}|_{n=\infty}$
 has radius of convergence $R$, 
 as above $R = \frac{1}{ |a_{n}^{\frac{1}{n}}| }|_{n=\infty}$, 
$R \prec \infty$.
 Solve the radius of convergence $r$
 for $\sum a_{n}^{2} x^{n}|_{n=\infty}$.

 $\sum a_{n}^{2} x^{n}|_{n=\infty}$
 $=\sum (a_{n}^{\frac{2}{n}} x)^{n}|_{n=\infty}$,  
 $|a_{n}^{\frac{2}{n}} x| \lt 1|_{n=\infty}$,  
 $|x| \lt |a_{n}^{-\frac{2}{n}} | |_{n=\infty}$,  
 $|x| \lt R^{2}$
 then radius of convergence $r=R^{2}$.
\end{mex}

 Considering power series with the Alternating Convergence Theorem (ACT),
 we can determine convergence with functions that can be represented
 with these power series, for example log and trigonometric functions.
\bigskip
\begin{mex}\label{MEX022}  
Show $\mathrm{ln}(1+x) = \sum_{k=0}^{n} (-1)^{k} \frac{ x^{k+1}}{k+1}|_{n=\infty}$ 
 converges when radius of convergence $r=1$.
 $\sum (-1)^{n} \frac{ x^{n+1}}{n+1}|_{n=\infty}$ 
 converges by the ACT (See Theorem \ref{P206}) if $\frac{ x^{n+1}}{n+1}|_{n=\infty}=0$.
 Solve  
 $\frac{ x^{n+1}}{n+1}|_{n=\infty}=0$.
 When $|x| \lt 1$,
 $\frac{ x^{n+1}}{n+1}|_{n=\infty}$
 $=x^{n+1}|_{n=\infty}$
 $=0$ as
 $x^{n+1} \succ\!\succ n+1|_{n=\infty}$,
 radius of convergence $r=1$.
 More simply without $\succ\!\succ$, 
 $\frac{ x^{n+1}}{n+1}|_{n=\infty} = x^{n+1} \cdot \frac{1}{n+1} = 0 \cdot 0 =0$.
\end{mex}
\bigskip
\begin{mex}\label{MEX023}  
 Determining the radius of convergence of 
 $\mathrm{atan}\,x$ follows the same 
 reasoning as Example \ref{MEX022},
 in determining convergence 
 consider $\mathrm{atan}\,x$ at infinity,
 $\mathrm{atan}\,x = \sum (-1)^{n} \frac{x^{2n+1}}{2n+1}|_{n=\infty}$.

 For negative $x$, factoring out the negative sign leaves the
 positive case,
 hence need only consider 
 $x=(0,\infty)$.

 When $x \neq 1$ we observe that
 $x^{2n+1}$ `log dominates' $2n+1$.
  $x^{2n+1} \; z \; 2n+1|_{n=\infty}$,
 $(2n+1) \, \mathrm{ln}\,x \; (\mathrm{ln}\,z) \; \mathrm{ln}(2n+1)|_{n=\infty}$,
 $(2n+1) \, \mathrm{ln}\,x \succ \mathrm{ln}(2n+1)|_{n=\infty}$,
 then $x^{2n+1} \succ\!\succ (2n+1)$.
 [ $\mathrm{ln}\,z = \; \succ$, $z = e^{\succ} = \;\succ\!\succ$ ]
 Then 
 $\sum (-1)^{n} \frac{x^{2n+1}}{2n+1}|_{n=\infty}$ 
 $= \sum (-1)^{n} x^{2n+1}|_{n=\infty}$.

 Case $x \gt 1$, 
 $\sum (-1)^{n} x^{2n+1}|_{n=\infty}=\infty$ diverges.
 Case $x = (0,1)$,
 $\sum (-1)^{n} x^{2n+1}|_{n=\infty}=0$ converges.
 Hence the radius of convergence $r=1$.

 For the interval of convergence, test the end points.
 Case $x=1$ converges by the ACT.
 $\sum \frac{(-1)^{n} 1^{2n+1}}{2n+1}|_{n=\infty}$
 $=\sum \frac{(-1)^{n}}{2n+1}|_{n=\infty}=0$
 converges by ACT
 as $\frac{1}{2n+1}|_{n=\infty}=0$.
 Case $x=-1$, 
 Case $x=1$ converges by the ACT. $\sum \frac{(-1)^{n+1}}{2n+1}|_{n=\infty}=0$
 Interval of convergence is $[-1,1]$.
\end{mex}
\bigskip
\begin{mex}\label{MEX024}  
 Determine radius of convergence for 
 $\mathrm{sin}\,x = \sum_{k=0}^{n} (-1)^{k} \frac{x^{2k+1}}{(2k+1)!}|_{n=\infty}$.
 Determine 
  $\sum (-1)^{n} \frac{x^{2n+1}}{(2n+1)!}|_{n=\infty}$.
 Assume $x$ is positive as sign can be factored out.
  Solve $\frac{x^{2n+1}}{(2n+1)!}|_{n=\infty}=0$,
  $x^{2n+1} \;\; z \;\; (2n+1)!|_{n=\infty}$,
 $(2n+1) \, \mathrm{ln}\,x \;\; (\mathrm{ln}\,z) \;\; \sum_{k=1}^{2n+1} \mathrm{ln}\,k|_{n=\infty}$,
 $(2n+1) \, \mathrm{ln}\,x \;\; (\mathrm{ln}\,z) \;\; \int^{2n+1} \mathrm{ln}\,n\,dn|_{n=\infty}$,
 since $\int \mathrm{ln}\,n\,dn = n \,\mathrm{ln}\,n|_{n=\infty}$, 
 $(2n+1) \, \mathrm{ln}\,x \;\; (\mathrm{ln}\,z) \;\; (2n+1)\, \mathrm{ln}(2n+1)|_{n=\infty}$.
 $(2n+1) \, \mathrm{ln}\,x \prec (2n+1)\, \mathrm{ln}(2n+1)|_{n=\infty}$,
 $\mathrm{ln}\,z = \; \prec$, $z = e^{\prec} = \; \prec$,
 $\frac{x^{2n+1}}{(2n+1)!}|_{n=\infty}=0$,
 by ACT the series is convergent for all $x$.
 Similarly the same result for $\mathrm{cos}\,x$.
\end{mex}
\subsection{Briefly differentiation and continuity} \label{S1403}
In considering properties of power series,
 we again find parallel theorems with the standard theorems.  
\bigskip
\begin{theo}\label{P040}
\cite[Theorem 11.9, pp.432--433]{apostol} $f(x) = \sum_{n=0}^{\infty} a_{n}(x-a)^{n}$,
 the differentiated
 series $\sum_{n=1}^{\infty} n a_{n}(x-a)^{n-1}$ also has radius of
 convergence $r$. 
\end{theo}

The termwise differentiation and integration
 theorems given in \cite[Theorems 3 and 4, pp.643--644]{kreyszig},
 that the power series differentiated and integrated
 have the same radius of convergence, 
 follows from a finite number power of $n$ being
 simplified at infinity, demonstrated
 by $\sum c n^{p} x^{n} = \sum x^{n}|_{n=\infty}$.

 Our assumption is that if a convergence sum
 is an infinireal, it can be integrated
 and differentiated, by treating each term separately.
\bigskip
\begin{theo}\label{P035}
 Termwise differentiation and integration of the power series 
 have the same radius of convergence.
\[ \sum a_{n} x^{n}|_{n=\infty} = \frac{\partial}{\partial x} \sum a_{n} x^{n}|_{n=\infty} = \int \sum a_{n} x^{n}\, \partial x |_{n=\infty} \] 
\end{theo}
\begin{proof}
 $\frac{\partial}{\partial x} \sum a_{n} x^{n}|_{n=\infty}$
 $= \sum a_{n} n x^{n-1}|_{n=\infty}$
 $= \sum a_{n} n x^{n}|_{n=\infty}$
 $= \sum e^{ \mathrm{ln}(a_{n} n x^{n})}|_{n=\infty}$
 $= \sum e^{ \mathrm{ln}\,a_{n} + \mathrm{ln}\, n + n \, \mathrm{ln}\, x}|_{n=\infty}$
 $= \sum e^{ \mathrm{ln}\,a_{n} + n \, \mathrm{ln}\, x}|_{n=\infty}$
 $= \sum a_{n} x^{n}|_{n=\infty}$.
 Similarly
 $\int \sum a_{n} x^{n}\,\partial x |_{n=\infty}$
 $= \sum a_{n} \frac{1}{n+1} x^{n+1}|_{n=\infty}$
 $= \sum a_{n} \frac{1}{n+1} x^{n}|_{n=\infty}$
 $= \sum e^{\mathrm{ln}(a_{n} \frac{1}{n+1} x^{n})}|_{n=\infty}$
 $= \sum e^{\mathrm{ln}\,a_{n} - \mathrm{ln}(n+1) + n \, \mathrm{ln}\,x}|_{n=\infty}$
 $= \sum e^{\mathrm{ln}\,a_{n} + n \, \mathrm{ln}\,x}|_{n=\infty}$
 $= \sum a_{n} x^{n}|_{n=\infty}$
\end{proof}
\bigskip
\begin{mex}\label{MEX025}  
 Find the radius of convergence of the sum, 
 $\sum { n \choose 2 } x^{n}|_{n=\infty}$
 $= \sum  \frac{n!}{(n-2)!2!} x^{n}|_{n=\infty}$
 $= \sum  n(n-1) x^{n}|_{n=\infty}$
 \cite[Example 1, p.799]{kreyszig}.

 $\sum  n(n-1) x^{n}|_{n=\infty}$.
 $= \sum  x^{n}|_{n=\infty}$, as $x^{n} \succ\!\succ n(n-1)$,
 or by bringing the $n$ terms into the power,
 $\sum  n(n-1) x^{n}|_{n=\infty}$
 $=\sum (n^{\frac{1}{n}}(n-1)^{\frac{1}{n}}x)^{n}|_{n=\infty}$
 $=\sum x^{n}|_{n=\infty}$, radius of convergence
 $r=1$.

 By application of Theorem \ref{P035}, partially
 integrating,
 $\sum  n(n-1) x^{n}|_{n=\infty}$
 $= \int \sum n(n-1) x^{n} \partial x|_{n=\infty}$
 $= \sum n(n-1)\frac{1}{n+1} x^{n+1}|_{n=\infty}$
 $= \int \sum (n-1) x^{n+1} \partial x|_{n=\infty}$
 $= \sum (n-1) x^{n+2} \frac{1}{n+2} |_{n=\infty}$
 $= \sum x^{n+2}|_{n=\infty}$,
 radius of convergence $r=1$.
\end{mex}

 For general testing,
 the ratio test is simpler to implement.
\bigskip
\begin{mex}
 Determine the radius of convergence of
 $\sum \frac{ 1 \cdot 3 \cdot 5 \ldots \cdot (2n-1) }{2 \cdot 4 \cdot 6 \ldots \cdot (2n)} \frac{ x^n}{n}|_{n=\infty}$.

$\sum \prod_{k=1}^{n} \frac{ 2k-1}{2k} \cdot \frac{x^{n}}{n}|_{n=\infty}$
$=\sum (\prod \frac{ 2n-1}{2n}|_{n=\infty}) \cdot \frac{x^{n}}{n}|_{n=\infty}$
$=\sum (\prod \frac{ 2n}{2n}|_{n=\infty}) \cdot \frac{x^{n}}{n}|_{n=\infty}$
$=\sum \frac{x^{n}}{n}|_{n=\infty}$
$=\sum x^{n}|_{n=\infty}=0$
 when 
$|x| \lt 1$ then $r=1$

 With the ratio test:
 Let $a_{n} = \prod_{k=1}^{n} \frac{ 2k-1}{2k} \cdot \frac{x^{n}}{n}$,  
 $|\frac{a_{n+1}}{a_{n}}| \lt 1|_{n=\infty}$,  
 $|\prod_{k=1}^{n+1}\frac{2k-1}{2k} \cdot \frac{x^{n+1}}{n+1} \prod_{k=1}^{n}\frac{2k}{2k-1} \cdot \frac{n}{x^{n}} | \lt 1|_{n=\infty}$,  
 $| \frac{2n+1}{2n+2} x| \lt 1|_{n=\infty}$, 
 $|x| \lt 1$, $r=1$
\end{mex}

 We consider continuity
 at infinity.
 By considering the convergence sums,
 if they differ near a point and at a point,
 then the sum is discontinuous at a point.
\bigskip
\begin{mex}\label{MEX018}  
\cite[Example 2, p.815]{kreyszig}.
 Show $\sum \frac{x^{2}}{(1+x^{2})^{n}}|_{n=\infty}$ is a discontinuous sum.

 Case $x=0$,
 $\sum \frac{0^{2}}{(1+0^{2})^{n}}|_{n=\infty}$ 
 $= \sum \frac{0}{1^{n}}|_{n=\infty}$ 
 $= \sum 0|_{n=\infty}$.

 Case $x \neq 0$,
 $\sum \frac{x^{2}}{(1+x^{2})^{n}}|_{n=\infty}$ 
 $=\sum (\frac{x^{\frac{2}{n}}}{(1+x^{2})} )^{n} |_{n=\infty}$ 
 $=\sum (\frac{1}{(1+x^{2})} )^{n} |_{n=\infty}$ 
 $= \sum \alpha^{n}|_{n=\infty}$, $\alpha = \frac{1}{1+x^{2}} \neq 0$.

 Comparing the convergence sums,
 $\sum 0 \; z \; \sum \alpha^{n}|_{n=\infty}$, 
 $0 \; z \; \alpha^{n}|_{n=\infty}$, 
 $0 \neq \alpha^{n}|_{n=\infty}$ as $0$ is not an
 infinitesimal and $\alpha^{n} \in \Phi$ is.
 Alternatively,
 $\sum 0 \; z \; \sum \alpha^{n}|_{n=\infty}$, 
 $0 \; z \; \int \alpha^{n}\,dn|_{n=\infty}$, 
 $0 \neq \alpha^{n} \mathrm{ln}\,\alpha|_{n=\infty}$. 
 
 Both sums converge, as when realized
 their convergence sum is zero.
 Since the convergence sum is not continuous
 about $x=0$, the convergence sum is not uniform continuous
 about $x=0$.
 Hence, while the sum is convergent, the sum is not uniformly convergent.
\end{mex}
\section{Convergence sums and the derivative of a sequence at infinity} \label{S15}
For convergence sums, by threading a continuous curve through
 a monotonic sequence, a series difference can be made a derivative.
 Series problems with differences can be transformed and solved in the continuous domain.
 At infinity, a bridge between the discrete and continuous domains is
 made. Stolz theorem at infinity is proved.
 Alternating convergence theorem for convergence sums is proved.

\subsection{Introduction}\label{S1501}
 There have always been relationships between series with discrete
 change and integrals with continuous change.
 In solving both problems and proofs we observe similarities and differences. 

 Series have no chain rule. However, for monotonic sequences satisfying
 the convergence sums criteria we can construct a continuous function
 at infinity where the chain rule can be applied.
 This can be combined with convergence sums integral test.

 In topology, a coffee cup can be transformed by stretching into a donut. 
 Similarly, we can consider a monotonic sequence which by stretching
 deforms into a strictly monotonic sequence.

 Consider a positive monotonic continuous function and its integral at
 infinity.
 Provided that the function's plateaus do not sum to infinity,
 the integral has the same convergence or divergence as
 the strictly deformed function's integral.

 Since convergence sums are monotonic,
 and can be deformed to be strictly monotonic,
 the correlation between the series and integrals 
 can be coupled in a way that results
 in a non-zero derivative.
 The derivative of a sequence follows.

 We believe the derivative of a sequence
 significantly
 changes convergence testing
 by allowing an interchange
 between sums and integrals
 with the integral theorem
 via sequences and functions
 in a fluid way.

 At infinity with infinireals
 we provide a classical explanation of a geometric
 construction of a curve threaded through a 
 sequence of points (see Figure \ref{FIG07}).
 
This simplicity explains what 
 can be 
 highly technical arguments
 on integer sums and theorems, which are not transferable between
 sums and integrals. 
 The mirrored discrete formula may use integer arguments in the
 proof specific to number theory whereas the continuous
 formula may be proved again by altogether different means.
 Never shall they meet.
  
 We again find that the
 acceptance of infinity, be it initially
 disturbing compared with classical arguments,
 ends up augmenting, upgrading 
 or replacing them. 

 The derivative of a sequence
 is a bridge between
 the continuous and discrete convergence
 sums at infinity.
\subsection{Derivative at infinity}\label{S1503}

When solving problems with sequences,
 there
 is no chain rule for sequences,
 as there is for the continuous variable.
 However, forward and backward
 differences are used in numerical analysis
 to calculate derivatives
 in the continuous domain.

In the discrete domain of integers,
 sequences, by contrast may use an
 equivalent theorem such as
 Stolz theorem or Cauchy's condensation
 test, as an effective chain rule.

If we consider a calculus of sequences,
 the change is an integer change,
 hence the goal is to construct a
 derivative that has meaning there.

Consider the following example 
 which motivates  the possibility
 of having a derivative at infinity,
 by constructing a derivative with powers
 at infinity. 

Since a function
 can be represented by a power series,
 we now can convert between
 a difference and a derivative at infinity.
 This uses non-reversible arithmetic. 
\bigskip
\begin{mex}\label{MEX201}
Let $f(x) = x^{2}$.
 $f(x+1)-f(x)|_{x=\infty}$
 $=(x+1)^{2}-x^{2}|_{x=\infty}$
 $= x^{2} + 2x + 1 - x^{2}|_{x=\infty}$
 $=2x+1|_{x=\infty}$
 $=2x|_{x=\infty} = f'(x)$, as $2x \succ 1|_{x=\infty}$.
\end{mex}
\bigskip 
\begin{lem}\label{P005}
Generalizing the derivative 
 of a power at infinity. If $f(x) = x^{p}|_{x=\infty}$
 then   
 $\frac{df}{dx} = f(x+1)-f(x)|_{x=\infty}$
\end{lem}
\begin{proof} 
$f(x+1)-f(x)|_{x=\infty}$
 $=(x+1)^{p} - x^{p}|_{x=\infty}$ 
 $=(x^{p} + { p \choose 1 } x^{p-1} + { p \choose 2 } x^{p-2} + \ldots ) - x^{p}|_{x=\infty}$
 $= p x^{p-1}|_{x=\infty}$, as $x^{k+1} \succ x^{k}|_{x=\infty}$.
\end{proof}
\bigskip
\begin{mex}\label{MEX202}
 Find the derivative of $\mathrm{sin} \, x$.
 Since $\mathrm{sin}\,x$ behaves the same
 as it does for finite values as it does at
 infinity, take the difference at infinity.
 Let $f(x) = \mathrm{sin}\,x$.
 $f(x+1)-f(x)|_{x=\infty}$
 $= \mathrm{sin}(x+1) - \mathrm{sin} \,x|_{x=\infty}$
 $= ((x+1) - \frac{1}{3!}(x+1)^{3} + \frac{1}{5!}(x+1)^{5} - \ldots)$
 $- (x - \frac{1}{3!}x^{3} + \frac{1}{5!}x^{5} - \ldots)|_{x=\infty}$
 $= 1 + (- \frac{1}{3!}(x+1)^{3} + \frac{1}{5!}(x+1)^{5} - \ldots)$
 $+( \frac{1}{3!}x^{3} - \frac{1}{5!}x^{5} + \ldots)|_{x=\infty}$
 $=1 + \sum_{k=1}^{\infty} (-1)^{k}( \frac{1}{(2k+1)!}(x+1)^{2k+1} - \frac{1}{(2k+1)!}x^{2k+1}) |_{x=\infty}$

Consider
 $\frac{1}{(2k+1)!}(x+1)^{2k+1}|_{x=\infty}$, taking the two
 most significant terms,
 $\frac{1}{(2k+1)!}(x+1)^{2k+1}|_{x=\infty}$
 $=\frac{1}{(2k+1)!}x^{2k+1} + \frac{1}{(2k+1)!}{2k+1 \choose 1} x^{2k}|_{x=\infty}$
 $=\frac{1}{(2k+1)!}x^{2k+1} + \frac{1}{(2k)!} x^{2k}|_{x=\infty}$

Substituting the expression into the previous sum, 
$f(x+1)-f(x)|_{x=\infty}$
$= 1 + \sum_{k=1}^{\infty} (-1)^{k}( \frac{1}{(2k+1)!}x^{2k+1} + \frac{1}{(2k)!} x^{2k} - \frac{1}{(2k+1)!}x^{2k+1} )|_{x=\infty}$
$= 1 + \sum_{k=1}^{\infty} (-1)^{k}\frac{1}{(2k)!} x^{2k}|_{x=\infty}$
$= \mathrm{cos}\,x|_{x=\infty}$,
 since a power series $f'(x)=\mathrm{cos}\,x$.
\end{mex}

Given a function $f(x)$, we can determine
 its derivative at infinity
 by converting $f(x)$ to a power series,
 taking the difference, and converting
 from the power series back into
 a function.
\bigskip
\begin{theo}\label{P201}
When $f(x) = \sum_{k=0}^{\infty} c_{i}x^{i}$,
\[ \frac{d f(x)}{dx} = f(x+1)-f(x)|_{x=\infty} \]
\end{theo}
\begin{proof}
Given $f(x) = \sum_{k=0}^{\infty} a_{k}x^{k}$,
 $\frac{df(x)}{dx}$
 $= \sum_{k=0}^{\infty} \frac{d}{dx} a_{k}x^{k}$
 $= \sum_{k=0}^{\infty} k a_{k} x^{k-1}$.
 Consider the difference, 
 $f(x+1)-f(x) = \sum_{k=0}^{\infty} (a_{k}(x+1)^{k} - a_{k}x^{k})$
 $= \sum_{k=0}^{\infty} k a_{k}x^{k-1}$ 
 $= \frac{d f(x)}{dx}$,
 by 
 Lemma \ref{P005}.
\end{proof}

An application of the derivative
 at infinity is, with the comparison
 logic, where 
 rather than
 either
 assume 
 that infinitesimally close
 expressions are equal
 or 
 using orders of higher magnitude
 to simplify under addition
 by forming a difference
 we can 
 obtain the derivative.
 Since the derivative is a function, we have
 an asymptotic result.
\bigskip
\begin{mex}\label{MEX203}
 While solving for relation $z$: 
 $f + \mathrm{ln}(n+1) \;\; z \;\; g + \mathrm{ln}\,n|_{n=\infty}$,
 $f + \mathrm{ln}(n+1) -  \mathrm{ln}\,n \;\; z \;\; g|_{n=\infty}$,
 $f + \frac{d}{dn}\mathrm{ln}\,n \;\; z \;\; g|_{n=\infty}$,
 $f + \frac{1}{n} \;\; z \;\; g|_{n=\infty}$
\end{mex}

 Without the derivative at infinity,
 with an assumed $f \simeq g|_{n=\infty}$
 logical errors in
 the calculation are more easily made.
 This can be addressed by 
 solving using magnitude arguments
 and non-reversible arithmetic;
 however, this does not yield
 an asymptotic error estimate.

 With the use of the sequence derivative,
 an asymptotic expression of the
 difference is formed. 
 
 The sequence derivative can be an alternative
 to the use of the binomial theorem. 
 (However, if there is any doubt other well known
 methods such as the binomial theorem are available.)
\bigskip
\begin{mex}\label{MEX012}
Using the binomial theorem,
 $(2n+1)^{\frac{1}{2}}-(2n)^{\frac{1}{2}}|_{n=\infty}$
 $=(2n)^{\frac{1}{2}}(1+\frac{1}{2n})^{\frac{1}{2}} - (2n)^{\frac{1}{2}}|_{n=\infty}$
 $= (2n)^{\frac{1}{2}}(1+\frac{1}{2}\frac{1}{2n}+\ldots-1)|_{n=\infty}$
 $=(2n)^{\frac{1}{2}}\frac{1}{4n}$
 $= \frac{1}{2n^{\frac{1}{2}}}|_{n=\infty}=0$

The same calculation with the derivative at infinity and
 non-reversible arithmetic.

 $(2n+1)^{\frac{1}{2}}-(2n)^{\frac{1}{2}}|_{n=\infty}$
 $=(2n+2)^{\frac{1}{2}}-(2n)^{\frac{1}{2}}|_{n=\infty}$
 $=(2(n+1))^{\frac{1}{2}}-(2n)^{\frac{1}{2}}|_{n=\infty}$
 $= \frac{d}{dn} (2n)^{\frac{1}{2}}|_{n=\infty}$
 $= \frac{1}{2}(2n)^{-\frac{1}{2}}2|_{n=\infty}$
 $= \frac{1}{2n^{\frac{1}{2}}}|_{n=\infty}=0$. $2n+1 = 2n+2|_{n=\infty}$
\end{mex}

 The following definitions and results are given, as logarithms
 are extensively used with sequences and convergence tests.
\bigskip
\begin{defy}\label{DEF201}
Let $\mathrm{ln}_{k}$ be $k$ nested log functions, by default
 having variable $n$. $\mathrm{ln}_{k} = \mathrm{ln}(ln_{k-1})$, $\mathrm{ln}_{0} = n$.
\end{defy}
\bigskip
\begin{defy}\label{DEF200}
Let $L_{w} = \prod_{k=0}^{w} \mathrm{ln}_{k}$. 
\end{defy}
\bigskip
\begin{lem}\label{P202} 
 $\frac{d}{dn} \mathrm{ln}_{w} = \frac{1}{L_{w-1}}|_{n=\infty}$
\end{lem}
\bigskip
\begin{mex}\label{MEX204}
 In the following
 comparison, $(\mathrm{ln}_{3}(x+1) - \mathrm{ln}_{3}\,x)|_{x=\infty}$
 $= \frac{d}{dx}\mathrm{ln}_{3}\,x|_{x=\infty}$
 $= \frac{1}{L_{2}(x)}|_{x=\infty}$
 $=0$. In a sense this is the error term. \cite[Example 3.19]{cebp21}. 
\begin{align*}
x^{\frac{p}{p+1}} \; z \; x^{\mathrm{ln}_{2}(x) / \mathrm{ln}_{2}(x+1)}|_{x=\infty} \tag{Solve for relation $z$} \\
\mathrm{ln}(x^{\frac{p}{p+1}}) \; (\mathrm{ln}\,z) \; \mathrm{ln}( x^{\mathrm{ln}_{2}(x) / \mathrm{ln}_{2}(x+1)})|_{x=\infty} \\
\frac{p}{p+1} \, \mathrm{ln}\,x \; (\mathrm{ln}\,z) \; \frac{\mathrm{ln}_{2}(x)}{ \mathrm{ln}_{2}(x+1)}\,\mathrm{ln}\,x|_{x=\infty} \\
p \, \mathrm{ln}_{2}(x+1) \; (\mathrm{ln}\,z) \; (p+1) \mathrm{ln}_{2}(x)|_{x=\infty} \\
\mathrm{ln}\,p + \mathrm{ln}_{3}(x+1) \; (\mathrm{ln}_{2}\,z) \; \mathrm{ln}(p+1) + \mathrm{ln}_{3}(x)|_{x=\infty} \\
\mathrm{ln}\,p + (\mathrm{ln}_{3}(x+1) - \mathrm{ln}_{3}\,x) \; (\mathrm{ln}_{2}\,z) \; \mathrm{ln}(p+1)|_{x=\infty} \tag{Apply derivative} \\
\mathrm{ln}\,p \; (\mathrm{ln}_{2}\,z) \; \mathrm{ln}(p+1)|_{x=\infty} \\
\mathrm{ln}\,p \lt \mathrm{ln}(p+1)|_{x=\infty} \\
\mathrm{ln}_{2}\,z = \;\lt, \; z = e^{e^{\lt}} = \;\lt
\end{align*}
\end{mex}

 In working with integers, it is sometimes convenient to solve the
 problem for real numbers, then translate back into the
 integer domain.

 The development of a way to convert between the integer domain or
 the domain of sequences, and
 the continuous domain, is similarly beneficial.
 For example, converting between sums and integrals.

By threading a continuous function through a monotonic sequence,
 we can construct a continuous function with the monotonic properties.
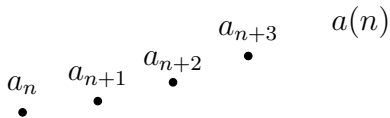
\begin{figure}[H]
\centering
\begin{tikzpicture}[domain=0.0:5.0]
  \draw[color=orange] plot[id=exp] function{0.05*x*x}; 
  \draw[fill=black] (1.0,0.05) circle [radius=0.05cm] ;
  \draw[fill=black] (2.0,0.2) circle [radius=0.05cm] ;
  \draw[fill=black] (3.0,0.45) circle [radius=0.05cm] ;
  \draw[fill=black] (4.0,0.8) circle [radius=0.05cm] ;
  \node [label={[shift={(1.0cm,0.0cm)}]$a_{n}$}] {};
  \node [label={[shift={(2.0cm,0.1cm)}]$a_{n+1}$}] {};
  \node [label={[shift={(3.0cm,0.30cm)}]$a_{n+2}$}] {};
  \node [label={[shift={(4.0cm,0.70cm)}]$a_{n+3}$}] {};
  \node [label={[shift={(5.5cm,0.75cm)}]$a(n)$}] {};
\end{tikzpicture}
\caption{Monotonic function and sequence through points} \label{FIG07}
\end{figure}
Further, since a monotonic series or integral can
 be deformed to a strictly monotonic series or integral
 (Theorem \ref{P212}),
 we need only consider the strictly monotonic case.
 By definition of the convergence
 Criterion E3 (Section \ref{S01}),
 cases where this cannot be done are said
 to be undefined. 

 Consider a positive sequence $(a_{n})$ in $*G$. 
 Without loss of generality,
 let $(a_{n})_{n=\infty}$ be strictly monotonic,
 either increasing or decreasing.

 Fit a curve, with conditions: $a(n) = a_{n}$.
 Let $a(x) = \sum_{k=0}^{n} c_{k}x^{k}$, pass through
 $n+1$ points. Solve for $(c_{k})$.

 Fitting a power series curve through
 a strictly monotonic sequence; 
 the curve fitted is also strictly monotonic
 (within the index interval).

 Since we determine convergence at infinity,
 we fit the curve for the sequence at infinity.
 Then $a(x)$ is strictly monotonic,
 and an analytic function.
 By converting a sequence difference,
 for example $a_{n+1}-a_{n}$
 to the continuous power series representation,
 Theorem \ref{P201} can be applied
 and a derivative formed.
 
In solving for one domain and transferring to the other,
 we can bridge between sequences and continuous functions.
\bigskip
\begin{defy}\label{DEF202}
 Let $(a_{n})|_{n=\infty}$ be a sequence at infinity
 and $a(n)$ a continuous function through the sequence.
\[ a_{n} = a(n)|_{n=\infty}; \; n \in \mathbb{J}_{\infty} \]
\end{defy}
\begin{defy}\label{DEF203}
Let the derivative of a sequence at infinity 
 be the difference of consecutive terms.
\[ a_{n+1} - a_{n} = \frac{d a_{n}}{dn}|_{n=\infty} \text{ or } a_{2n+1}-a_{2n} = \frac{ d a_{n}}{dn}  \text{ where } a_{n}|_{n=\infty} \neq \alpha \text{ a constant.} \]
\end{defy} 
 How the derivative of a sequence is defined is problem dependent.
 It is up to the user. 
 In a similar way we may start counting from $0$ or $1$.
 By the contiguous rearrangement theorem Theorem \ref{P210}, 
 we need only determine one contiguous rearrangement to determine
 convergence or divergence.

 Consider the technique of adequality \cite[p.5]{victors} more generally 
 to that of a principle of variation.
\[ d(f(A)) = f(A+E)-f(A) \]
 As a change in consecutive integers is $1$,
 \[ dn = (n+1)-n \]
 we can see a correspondence between a sequence derivative,
 and the continuous derivative.
\[ d(a_{n}) = a_{n+1}-a_{n} = \frac{ a_{n+1}-a_{n}}{1} =\frac{ d a_{n}}{dn} \]

To aid calculation, a convention of left to right equals symbol ordering is
  used to indicate which direction a conversion is taking place.
 Further, by redefining a variable from an integer to the continuous variable,
 will enable the transformation to be more natural and effortless. 
\bigskip
\begin{theo}\label{P203} 
 By threading a continuous function $a(n)$ through
 sequence $a_{n}$ and preserving
  monoticity.
 \[ \frac{ d a_{n} }{dn} = \frac{ d a(n)}{dn}|_{n=\infty} \]
\end{theo}
\begin{proof}
 Let $a(n)$ be represented by a power series.
 $\frac{ d a_{n}}{dn} = a_{n+1}-a_{n}$
 $= a(n+1) - a(n)$
 $= \frac{ d a(n)}{dn}|_{n=\infty}$ by Theorem \ref{P201}
\end{proof}
\bigskip
\begin{remk}
 The usefulness of
 the change of integers can be seen
 when considering 
 the equality of the Riemann sum to the integral
 \cite[Remark 2.1]{cebp10}, hence discrete change 
 has generality.
\end{remk}

 On the assumption that $\frac{ d^{k} a_{n}}{d n^{k}}$ 
 can be similarly defined.
\bigskip
\begin{defy}\label{DEF204}
 Converting between the discrete sequence and continuous curve through the sequence, with left to right direction. 
\[ f_{n}(a_{n}, \frac{ d a_{n} }{dn}, \ldots ) = f( a(n), \frac{d a(n)}{dn}, \ldots)|_{n=\infty} \text{   sequence to function} \] 
\[ f( a(n), \frac{d a(n)}{dn}, \ldots) = f_{n}(a_{n}, \frac{ d a_{n} }{dn}, \ldots )|_{n=\infty} \text{   function to sequence} \] 
\end{defy}
\bigskip
\begin{theo}\label{P204}
For a strictly monotonic sequence, we can construct an associated strictly monotonic function
 that is continually differentiable.
\end{theo}
\begin{proof}
For a strictly monotonic sequence, the sequence derivative is never $0$,
 a power series at infinity, say for $N$ infinite number of points,
 solving $N$ equations,
 the resulting 
 curve is continually differentiable.
\end{proof}

 With the interchangeability of the derivative
 between sequences and continuous functions,
 equations involving sequences can be solved
 as differential equations,
 and the result transformed back into
 the domain with sequences.  
 Bridging the continuous and discrete
 domains
 at infinity.
\bigskip
\begin{prop}
 If 
 $a_{n+1}-a_{n}|_{n=\infty}=\alpha$,
 then 
 $\frac{a_{n}}{n}|_{n=\infty}=\alpha$. \cite[2.3.14]{kaczor}
\end{prop}
\begin{proof}
 As an alternative to the use of
 Stolz theorem,
 $a_{n+1}-a_{n}|_{n=\infty}$
 $= \frac{a_{n+1}-a_{n}}{dn}|_{n=\infty}$
 $=\alpha$, 
 $\frac{d a(n)}{dn} = \alpha$, 
 separate the variables,
 $d (a(n)) = \int \alpha dn$,
 $a(n) = \alpha n|_{n=\infty}$,
 $a_{n} = \alpha n|_{n=\infty}$,
 $\frac{ a_{n}}{n}|_{n=\infty}=\alpha$.
\end{proof}
\bigskip
\begin{theo}\label{P205}
Stolz theorem. Given sequence $(y_{n})|_{n=\infty}$
 is monotonically
  increasing and diverges,
 $y_{n}|_{n=\infty}=\infty$, 
 and $\frac{ x_{n}-x_{n-1} }{ y_{n} - y_{n-1}}|_{n=\infty} = g$,  
then 
$\frac{ x_{n}}{y_{n}}|_{n=\infty} = g$
\end{theo}
\begin{proof}
$\frac{ x_{n}-x_{n-1} }{ y_{n} - y_{n-1}}|_{n=\infty}$
$=\frac{ x_{n}-x_{n-1} }{dn}\frac{dn}{ y_{n} - y_{n-1}}|_{n=\infty}$
$=\frac{ d x_{n}}{dn}\frac{dn}{d y_{n}}|_{n=\infty}$
$=\frac{ d x(n)}{dn}\frac{dn}{d y(n)}|_{n=\infty}$
$=\frac{ d x(n)}{d y(n)}|_{n=\infty}=g$,
 recognizing a separation of variables problem, 
 separate and integrate the variables. 
 $\int dx = g \int dy|_{n=\infty}$,
 $x(n) = g y(n)|_{n=\infty}$,
 $x_{n} = g y_{n}|_{n=\infty}$,
 $\frac{x_{n}}{y_{n}}|_{n=\infty}=g$.
\end{proof}

In applications with series expansions that include
 differences, when it is possible to arbitrarily truncate
 the series, apply the transforms for
 the new system.
\bigskip
\begin{mex}\label{MEX205}
Using the sequence derivative 
 with a $\mathrm{sin}$ expansion.
 $0 \lt a_{1} \lt 1$,  
 $a_{n+1} = \mathrm{sin} \, a_{n}$,  
 Show $n^{\frac{1}{2}} a_{n} = 3^{\frac{1}{2}}|_{n=\infty}$.

Within the interval, $0 \leq \mathrm{sin}\,x \lt x$,
 then $a_{n+1} \leq a_{n}$. Applying this to
 infinity, 
 $a_{n}|_{n=\infty}=0$
 Using a Taylor series expansion,
 a one term expansion fails, giving a derivative of $0$.
 However a two term expansion succeeds.

$\mathrm{sin} \, x = x - \frac{x^{3}}{3!} + \ldots$,  
$a_{n+1}$
$= \mathrm{sin}\,a_{n}|_{n=\infty}$
 $= a_{n} - \frac{a_{n}^{3}}{3!}|_{n=\infty}$,  
$a_{n+1}-a_{n} = - \frac{a_{n}^{3}}{6}|_{n=\infty}$,  
$\frac{d a(n)}{dn} = -\frac{a^{3}}{6}|_{n=\infty}$,  
$\frac{da}{a^{3}} = -\frac{dn}{6}|_{n=\infty}$,  
$-\frac{1}{ 2a^{2}} = -\frac{n}{6}|_{n=\infty}$, 
$\frac{1}{ a^{2}} = \frac{n}{3}|_{n=\infty}$, 
$3 = n a^{2}$,  
$3^{\frac{1}{2}} = n^{\frac{1}{2}}a_{n}|_{n=\infty}$.
\end{mex}

While it is standard practice of including
 the integral symbol when integrating,
 the integral itself may be 
 subject to algebraic simplification,
 on occasions,
 it can be better to leave off the integral symbol.
\bigskip
\begin{defy}
For a continuous variable, 
 integration can be expressed without the
 integral symbol. $(a \, dn)$ means $\int a\,dn$.
\end{defy}
\bigskip
When considering a change of variable, as in the chain rule,
 a variable is used to express the change.
 However this is not necessarily required,
 By the $d( )$ operator, integration and differentiation 
 are possible. This can be more direct.
\bigskip
\begin{mex}\label{MEX206}
$\int \frac{2u}{u^{2}+1}du$. Let $v = u^{2}+1$, $\frac{dv}{du} = 2u$.
 $\int \frac{2u}{u^{2}+1}du$
 $= \int \frac{dv}{du} \frac{1}{v}du$
 $= \int dv \frac{1}{v}$
 $= \mathrm{ln}\,v$ 

Alternatively without the variable,
$\int \frac{2u}{u^{2}+1}du$
 $= \int \frac{d(u^{2}+1)}{du} \frac{1}{u^{2}+1}du$
 $= \int d(u^{2}+1) \frac{1}{u^{2}+1}$
 $= \mathrm{ln}(u^{2}+1)$
\end{mex}

Formally the integral symbol $\int$ and the change
 in variable $dx$ integrate the expression between them
 $\int y(x) dx$. However, when working
 with the algebra and  cancelling, integration
 and differentiation become factors. The integral symbol is
 not always necessary, and the order of cancellation 
 does not necessarily put the variable at the right end.

 From the point of view of solving, 
 the integral symbol $\int$ may be omitted,
 where trying different
 combinations of change may be beneficial. 

 Providing the context is clear,
 you can remove the integral symbols, but
 include the symbols at the end when communicating.  

 The generalised p-series test (See Section \ref{S1802}).
\[
\sum \frac{1}{\prod_{k=0}^{w-1} \mathrm{ln}_{k} \cdot \mathrm{ln}_{w}^{p} }|_{n=\infty}
 = \left\{
  \begin{array}{rl}
    0  & \; \text{converges when } p \gt 1 \\
    \infty  & \; \text{diverges when } p \leq 1 
  \end{array} \right.
\]
\bigskip
\begin{mex}\label{MEX207}
\cite[p.89 3.3.6]{kaczor}.
 Given 
 $s_{n} = \sum_{k=1}^{n} a_{k}$, 
 $s_{n}|_{n=\infty}=\infty$. \\
 3.3.6.a Show
 $\sum \frac{a_{n+1}}{s_{n} \, \mathrm{ln} \, s_{n}}|_{n=\infty}=\infty$
 diverges.

 Transform the problem into the continuous domain.
 $\sum \frac{a_{n+1}}{s_{n} \, \mathrm{ln} \, s_{n}}|_{n=\infty}$
 $= \sum \frac{a_{n+1}}{s_{n} \, \mathrm{ln} \, s_{n}} dn|_{n=\infty}$
 $= \int \frac{a}{s \, \mathrm{ln} \, s} dn|_{n=\infty}$ where 
 $n$ has been redefined. 
 Let $a(n)$ and $s(n)$ be continuous functions to replace $a_{n}$
 and $s_{n}$ respectively.  
 $s = a \, dn$, $s|_{n=\infty}=\infty$.

Observing $\frac{ds}{dn} = \frac{ d(a\,dn)}{dn} = a$ then
 $\int \frac{a}{s \, \mathrm{ln} \, s} dn|_{n=\infty}$ 
 $= \int \frac{1}{s \, \mathrm{ln} \, s} \frac{ds}{dn} dn|_{n=\infty}$ 
 $= \int \frac{1}{s \, \mathrm{ln} \, s} \,ds|_{n=\infty}$ 
 $= \infty$ diverges.

Alternatively applying the chain rule.  
 $\int \frac{a}{s \, \mathrm{ln} \, s} dn|_{n=\infty}$ 
 $= \int \frac{a}{(a \, dn) \, \mathrm{ln} \, (a \, dn)} dn|_{n=\infty}$ 
 $= \int \frac{a}{(a \, dn) \, \mathrm{ln} \, (a \, dn)} \frac{dn}{d (a \, dn)} d(a \, dn)|_{n=\infty}$ 
 $= \int \frac{a}{(a \, dn) \, \mathrm{ln} \, (a \, dn)} \frac{1}{a} d(a \, dn)|_{n=\infty}$ 
 $= \int \frac{1}{(a \, dn) \, \mathrm{ln} \, (a \, dn)} d(a \, dn)|_{n=\infty}$ 
 $= \int \frac{1}{s \, \mathrm{ln} \, s} ds|_{s=\infty}$ 
 $=\infty$
 as on the boundary.
\end{mex}
 The derivative of a sequence(Definition \ref{DEF203})
 leads to a chain rule with sequences.
\bigskip
\begin{mex}\label{MEX208}
 Example \ref{MEX207}, solved with the derivative,
 noticing that $a_{n+1} = s_{n+1}-s_{n}$ and constructing
 a derivative $\frac{d s_{n}}{dn}$. 

 $\sum \frac{a_{n+1}}{s_{n} \, \mathrm{ln} \, s_{n}} dn |_{n=\infty}$
 $= \sum \frac{s_{n+1}-s_{n}}{s_{n} \, \mathrm{ln} \, s_{n}} dn |_{n=\infty}$
 $= \sum \frac{d s_{n}}{dn} \frac{1}{s_{n} \, \mathrm{ln} \, s_{n}} dn |_{n=\infty}$
 $= \sum \frac{1}{s_{n} \, \mathrm{ln} \, s_{n}} d s_{n} |_{S_{n}=\infty}$
 $=\infty$ diverges. 
\end{mex}
\bigskip
\begin{mex}\label{MEX209}
\cite[p.89 3.3.6.b]{kaczor}.
 Continued from Example \ref{MEX207}.
 Show $\sum \frac{a_{n}}{s_{n} (\mathrm{ln} \, s_{n})^{2}}|_{n=\infty}=0$ converges. 

$\sum \frac{a_{n}}{s_{n} (\mathrm{ln} \, s_{n})^{2}}|_{n=\infty}$
$=\sum \frac{s_{n} - s_{n-1}}{s_{n} (\mathrm{ln} \, s_{n})^{2}}dn|_{n=\infty}$
$=\sum \frac{d s_{n}}{dn} \frac{1}{s_{n} (\mathrm{ln} \, s_{n})^{2}}dn|_{n=\infty}$
$=\sum \frac{1}{s_{n} (\mathrm{ln} \, s_{n})^{2}}d s_{n}|_{s_{n}=\infty}$
$=\int \frac{1}{s (\mathrm{ln} \, s)^{2}}d s|_{s=\infty}$
$=0$ converges (Generalised p-series, $p=2 \gt 1$).
\end{mex}

The exception to the derivative forming
 a difference is when $a_{n}|_{n=\infty}=\alpha$ is
 a constant, see Definition \ref{DEF203}.
 The sum
 of the power series, instead of being an 
 infinite sum, reduces to a single term,
 or an infinity of terms with a non-monotonic
 function.
 At infinity, the power series could not be monotonic, or
 have a strict relation.
\bigskip
\begin{mex}\label{MEX210}
To demonstrate the case, applying the derivative
 to the following problem.
 
Let $(a_{n})$ be a sequence with
 $a_{n}|_{n=\infty}=\alpha \neq 0$,
 $a_{n} \gt 0$. Prove that the
 series $\sum_{k=1}^{\infty} (a_{n+1}-a_{n})$
 and $\sum_{k=1}^{\infty} (\frac{1}{a_{n+1}}-\frac{1}{a_{n}})$
 both absolutely converge or both absolutely
 diverge.  \cite[3.4.17]{kaczor}

 Reorganising the problem,
 show 
 $\sum(a_{n+1}-a_{n})|_{n=\infty}$
 and $\sum (\frac{1}{a_{n+1}}-\frac{1}{a_{n}})|_{n=\infty}$
 both absolutely converge or both absolutely diverge.

 Following the approach given in this paper.
$\sum (a_{n+1}-a_{n})|_{n=\infty}$
$=\sum \frac{d a_{n}}{dn} dn|_{n=\infty}$
$=\int \frac{d a}{dn} dn|_{n=\infty}$
$=\int da|_{n=\infty}$
$=a|_{n=\infty} = \alpha$ \\
$\sum \frac{a_{n}-a_{n+1}}{a_{n}a_{n+1}}dn|_{n=\infty}$
$=\int -\frac{da}{dn}\frac{1}{a^{2}}dn|_{n=\infty}$
$=\int -\frac{1}{a^{2}}da|_{n=\infty}$
$=\frac{1}{a}|_{n=\infty}$
 $=\frac{1}{\alpha}$ 

 Both the sums fail
 the convergence criterion E3 where
 we expect the sums
 at infinity to be either $0$ or $\infty$.

This is suggesting that for a constant we need to
 treat the theory separately. Here the problem
 is reconsidered
 with the reasoning that $a_{n}$ is a constant,
 and $a_{n+1}-a_{n}$ is an infinitesimal,
\begin{proof}
$\sum(\frac{1}{a_{n+1}}-\frac{1}{a_{n}})|_{n=\infty}$
 $= \sum \frac{a_{n}-a_{n+1}}{a_{n+1}a_{n}}|_{n=\infty}$
 $= \sum -\frac{1}{a_{n+1} a_{n}}(a_{n+1}-a_{n})|_{n=\infty}$
 $= \sum -\frac{1}{\alpha^{2}}(a_{n+1}-a_{n})|_{n=\infty}$
 $= \sum (a_{n+1}-a_{n})|_{n=\infty}$. Since at infinity
 the sums are equal, so is their absolute value sum.
\end{proof}
\end{mex}

When approximating numerically, solving
 for a variable by variation,
 it is common
 to incrementally approach the solution with
 numerical schemes. 
\[ \text{If } \delta_{n} \to 0 \text{ then } x_{n+1}-x_{n} = \delta_{n}, \;\;  x_{n+1}-x_{n} = \frac{d x_{n}}{dn} = \frac{d x(n)}{dn}|_{n=\infty} = 0 \]
 The iterative scheme has a solution when its derivative is zero,
 corresponding to the solution of the problem.
\bigskip
\begin{mex}\label{MEX211}
\cite[Example 5.4]{cebp21}
 We can show the derivative of $x_{n}$, successive
 approximations, as decreasing in the
 following algorithm.
 $x \in *G$; $\delta \in \Phi$;
 $(x+\delta)^{2}=2$. Develop an iterative scheme,
 $x^{2} + 2x \delta+\delta^{2}=2$;
 $x^{2} + 2x \delta=2$ as $2x \delta \succ \delta^{2}$,
 $x_{n}^{2} + 2x_{n} \delta_{n}|_{n=\infty} = 2$,
 $\delta_{n} = \frac{1}{x_{n}} - \frac{x_{n}}{2}|_{n=\infty}$.
 Couple by solving for $x_{n+1} = x_{n}+\delta_{n}$.

In the ideal case, $(x_{n}+\delta_{n})^{2}|_{n=\infty} \simeq 2$
 Provided $\delta_{n} \to 0$, $(x_{n})|_{n=\infty}$ is a
 series of progressions towards the solution.
 This can be expressed as a derivative.
 $x_{n+1} = x_{n} + \delta_{n}$,
 $x_{n+1}-x_{n} = \delta_{n}$,
 $\frac{d x_{n}}{d n} = \delta_{n}$.

 Transferring the algorithm $*G \to \mathbb{R}$,
 provided we observe the same decrease in $\delta_{n}$,
 the algorithm finds the solution.

 Let 
 $x_{1}=1.5$,
 $\delta_{n}: (
 -8.3 \times 10^{-2},
 -2.45 \times 10^{-3},
 -2.12 \times 10^{-6},
 -1.59 \times 10^{-12}, \ldots)$
 As the gradient is negative
 and decreasing,
 $n$ vs $x_{n}$ is monotonically
 decreasing and asymptotic to the solution
 $x_{n}|_{n=\infty}=\sqrt{2}$.
\end{mex}
\subsection{Convergence tests}\label{S1502}
\begin{theo}\label{P206}
 The Alternating convergence theorem (ACT).
If $(a_{n})|_{n=\infty}$ is a monotonic
 decreasing sequence and 
 $a_{n}|_{n=\infty}=0$ then $\sum (-1)^{n} a_{n}|_{n=\infty}=0$ is convergent.
\end{theo}
\begin{proof}
Compare against
 the boundary (Section \ref{S18}) between convergence 
 and divergence.
\begin{align*}
\sum (-1)^{n} a_{n} \; z \; \sum \frac{1}{\prod_{k=0}^{w} \mathrm{ln}_{k}}|_{n=\infty} \tag{Rearrangent, see (Section \ref{S16})} \\
\sum a_{2n}-a_{2n-1} \; z \; \sum \frac{1}{\prod_{k=0}^{w} \mathrm{ln}_{k}}|_{n=\infty} \tag{A sequence derivative} \\
\sum \frac{d a_{n}}{dn} \; z \; \sum \frac{1}{\prod_{k=0}^{w} \mathrm{ln}_{k}}|_{n=\infty} \tag{Discrete to continuous $n$} \\
\frac{d a(n)}{dn} \; z \; \frac{1}{\prod_{k=0}^{w} \mathrm{ln}_{k}}|_{n=\infty} \tag{Separation of variables} \\
d a(n) \; z \; \int \frac{1}{\prod_{k=0}^{w} \mathrm{ln}_{k}} dn |_{n=\infty} \\
a(n) \; z \; \mathrm{ln}_{w+1} |_{n=\infty} \tag{ substituting conditions, $a(n)|_{n=\infty}=0$ } \\
 0 \; z \; \infty, \;\; z =\; \lt 
\end{align*}
\end{proof}
\section{Rearrangements of convergence sums at infinity} \label{S16}
Convergence sums theory is concerned with
 monotonic series testing. On face value, this
 may seem a limitation but, by 
 applying
 rearrangement theorems at infinity,
 non-monotonic sequences can be rearranged
 into monotonic sequences. 
 The resultant monotonic series are convergence sums.
 The classes of convergence sums
 are greatly increased
 by the additional versatility applied to the theory.
\subsection{Introduction}\label{S1601}
 The premise of the paper is that
 convergence sums (Section \ref{S01}) order of terms affects convergence.
 Surprisingly, the most simple rearrangement of bracketing
 terms of the sum differently (addition being associative)
 is profoundly useful for sums at infinity,
 as these sums have an infinity of terms, and an order.

 We believe there is still much that is unknown regarding convergence.
 In fact, historically the discussions and difference of opinions were and perhaps are far apart. 
 Pringsheim says:
\begin{quote}
Since in a series of positive terms the order in which the terms come has nothing
 to do with convergence or divergence of the series$\ldots$ \cite[p.9]{cajori}
\end{quote}
 F. Cajori addresses this; however the above is a real problem. 
 That such a basic fact was not accepted may explain
 why the sums order has not previously been incorporated into convergence theory. 
 From our perspective,
 the existence of infinite integers has opened the possibilities, yet in
 general, the infinitesimal or infinity 
 still has not been accepted as a number.

 Knopp on ``Infinite sequences and series" \cite{knopp} does not refer to infinitesimals,
  or infinity as a partition.  Theorems are generally stated from
 a finite number to infinity;
 however he does state
 theorems from a certain point onwards: something is true,
 an infinity in disguise. 
 This is of course classical mathematics.
 The concept of infinity creeps in through subtle arguments,
 the use of null sequences, which effectively are
 at infinity. 
 For an example, see Theorems \ref{P217} and \ref{P218}
 which are the same theorem, said in a different way.

 Having said this, we find Knopp's exposition and communication exceptional. So we do 
 not necessarily agree with the content,
 but for the rearrangement theorems in this paper
 we look to Knopp for both mathematical depth
 and the presentation. If we do not succeed,
 this is our and not Knopp's fault. 

 However, we do find classical mathematics applied at infinity
 to be extremely useful, if anything, often
 extending the original concept.

 In an interesting way, similarly to topology
 that stretches or deforms shapes, converting a coffee cup
  into a doughnut, we can stretch a monotonic sequence to a
 strictly monotonic sequence for convergence testing,
 where criteria E3 excludes 
 plateaus. 

A rearrangement is a reordering.
 $(1, 2, 3)$ can be rearranged to $(2, 3, 1)$.
 For an infinite sequence, 
 we can partition the sequence into
 other infinite sequences.
\bigskip
\begin{mex}
Partition the natural numbers into odd and even sequences.
 $(1, 2, 3, 4, 5, \ldots)$
 $=(1) + (2) + (3) + (4) + (5) + \ldots$, select every second
 element to generate two sequences,
 $(1, 3, 5, \ldots)$ and $(2, 4, 6, \ldots)$.
\end{mex}

 By partitioning an infinite sequence into two or 
 more other infinite sequences, we can construct rearrangements
 by taking (or by copying the whole and deleting) from the partition sequences.
\bigskip
\begin{defy}\label{DEF205}
A `subsequence' is a sequence formed from 
 a given sequence by deleting
 elements without changing the relative
 position of the elements.
\end{defy}

 Just as we have uses for empty sets, we
 define an empty sequence.
\bigskip
\begin{defy}
Let $()$ define an empty sequence. 
\end{defy}

 We find it useful to consider partition sequences
 which are subsequences. $(a_{n})$ partitioned into
 subsequences $(b_{k})$ and $(c_{j})$.
 While these are only a subset of arrangements,
 they can be used in theory and calculation.
\bigskip
\begin{prop}
If $a = (a_{1}, a_{2}, a_{3}, \dots)$ is partitioned into
 $b = (b_{1}, b_{2}, b_{3}, \ldots)$   and $c = (c_{1}, c_{2}, c_{3}, \ldots)$,
 where $b$ and $c$ are subsequences of $a$.
 Let an element $a_{k}$ in $a$ be in
 either $b$ or $c$.
 We can form a rearrangement of $a$, by having positional counters
 in $b$ and $c$, and sampling to a new sequence. 
\end{prop}
\begin{proof}
 Let $d$ be an empty sequence.
 Start the counters at the first element, increment by one after each
 sample to sequence $d$ by appending to $d$. 
 Arbitrarily sample from $a$ and $b$ depending on the rearrangement
 choice.
\end{proof}

 What is interesting about infinity, is that
 you may iterate over the different
 partitions unevenly. For example, 
 in an unequal ratio. In this case, we say
 that the partitions are being sampled
 at different rates.
\bigskip
\begin{mex}
Given infinite sequences $(1, 3, 5, 7, \ldots)$
 and $(2, 4, 6, 8, \ldots)$,
 create a rearrangement that for every
 odd number, sample two even numbers.
 A ratio of $1:2$.

 $(1)+(2,4)+(3)+(6,8)+(5)+(10,12)+\ldots$
 $= (1, 2, 4, 3, 6, 8, 5, 10, 12, \ldots)$
 This is a rearrangement of $(1, 2, 3, 4, \ldots)$. 
\end{mex}

 Another example, although the partition of the original
 sequence is partitioned in two, an infinitely
 small number of terms are sampled from one partition
 compared with the other partition.
\bigskip
\begin{mex}
 $(2, 3, 4, 5, \ldots)$ rearranged in a $1: 2^{n}$ ratio between
 the odd and even numbers,
 $( (2), (3), (4, 6), (5), (8, 10, 12, 14), (7), (16, 18, 20, 22, 24, 26, 28, 30), (9),  \ldots)$ \\
 $= (2, 3, 4, 6, 5, 8, 10, 12, 14, 7, 16, 18, 20, 22, 24, 26, 28, 30, 9, \ldots)$.
\end{mex}

 For finite sums, a sum rearrangement does not change
 the sum. However, for an infinite sum the situation
 can become very different.
 The given order of terms affects convergence.

 Different orderings/rearrangements on the same partition 
 of infinite terms can radically change the sum's value
 and convergence or divergence result.

We develop rearrangement theorems at infinity.
 We also construct theorems at infinity then transfer
 these back to known theorems via the transfer principle \cite[Part 4]{cebp21}.
 Hence, we establish the usefulness of infinity at a point.

With the rearrangement of sums
 at infinity,
 it was found that 
 a conditionally convergent sum 
($\sum_{\mathbb{N}_{\infty}} a_{j}=0$ converges
 but $\sum_{\mathbb{N}_{\infty}} |a_{j}|=\infty$ diverges),
 can be rearranged into
 a divergent sum \cite[Theorem 5 p.80]{knopp}.
 We will encounter examples of this with
 convergence sums (Section \ref{S01}),
 when considering sum
 rearrangements 
 independently.

That is, consider a partition of a sum at infinity.
 A rearrangement of the sum at infinity could
 unevenly sample one partition compared with the other.
 Since the index is still iterating over infinity,
 all elements are still summed.

Partition $(a_{n})|_{n=\infty}$
 into $(b_{n})|_{n=\infty}$ and $(c_{n})|_{n=\infty}$.
$\sum a_{n}|_{n=\infty} = \sum b_{n}|_{n=\infty} + \sum c_{n}|_{n=\infty}$
\bigskip
\begin{mex}
 The elementary proof of
 the harmonic
 series divergence.
 Choose a rearrangement with a `variable period' 
 of powers of two. 
 $\sum_{k=1}^{\infty} \frac{1}{k}$,
 group in powers of two, 
 $\sum_{k=2}^{\infty} \frac{1}{k}$
 $= \frac{1}{2} + (\frac{1}{3}+\frac{1}{4}) + (\frac{1}{5}+\frac{1}{6}+\frac{1}{7}+\frac{1}{8}) + \ldots$
 $\geq \frac{1}{2} + \frac{1}{2} + \frac{1}{2} + \ldots$
 diverges.
 As a sequence rearrangement, 
  $(\frac{1}{2}, \frac{1}{4}+\frac{1}{6}, \frac{1}{8} + \frac{1}{10} + \frac{1}{12} + \frac{1}{14}, \ldots) \geq (\frac{1}{2}, \frac{1}{2}, \frac{1}{2}, \ldots)$

 $b_{n} = \sum_{k=2^{n}+1}^{2^{n+1}} \frac{1}{k}$,
 $b_{k} \geq \frac{1}{2}$,
 $0 \leq \sum \frac{1}{2}|_{n=\infty} \leq \sum b_{n}|_{n=\infty}$,
 $\sum b_{n}|_{n=\infty}=\infty$ diverges,
 $\sum \frac{1}{n}|_{n=\infty}=\infty$.
\end{mex}
\bigskip
\begin{mex}\label{MEX212}
$\frac{1}{2} - \frac{1}{3} + \frac{1}{4} - \frac{1}{5}+\ldots$ is conditionally
 convergent.
 $\sum (-1)^{n} \frac{1}{n}|_{n=\infty}=0$.
 At infinity, partition $((-1)^{n}\frac{1}{n})|_{n=\infty}$ into $(\frac{1}{2n})|_{n=\infty}$
 and $(-\frac{1}{2n+1})|_{n=\infty}$.
 As this is a conditionally convergent series,
 we can find a rearrangement of the series which diverges.

 By considering the even numbers of the sum, we can construct
 a divergent harmonic series.
 $\frac{1}{4} + \frac{1}{6} + \frac{1}{8} + \frac{1}{10} + \ldots$
 $= \frac{1}{2}( \frac{1}{2} + (\frac{1}{3}+\frac{1}{4}) + (\frac{1}{5}+\frac{1}{6}+\frac{1}{7}+\frac{1}{8}) + \ldots)$
 $= \frac{1}{2}( b_{0} + b_{1} + b_{2} + \ldots)$
 As in the previous example, $b_{k} \geq \frac{1}{2}$.

 Choose an arrangement: for every odd term summed,
 sum $2^{n}$ even $a_{n}$ terms; 
 for a ratio of $1:2^{n}$.

 $\sum a_{n}|_{n=\infty}$
 $= \frac{1}{2} + \sum_{k=0}^{n} (\frac{b_{k}}{2} - \frac{1}{2k+3})|_{n=\infty}$
 For convergence or divergence, consider the sum at infinity then
 $\sum (\frac{b_{n}}{2} - \frac{1}{2n+3})|_{n=\infty}$
 $=\sum \frac{b_{n}}{2} |_{n=\infty}$
 $\geq \sum \frac{1}{4}|_{n=\infty}=\infty$,
 $\sum a_{n}|_{n=\infty} = \infty$ diverges.
\end{mex}

 The problem of concern above,
 with a conditionally convergent sum,
 is that summing
 in different unequal rates of the
 partitions affects the sum's result. 
 
 Consideration of
 different arrangements leads to
 Riemann's rearrangement theorem (see Theorem \ref{P216}),
 where the same sum converges
 to a chosen value,
 or,
 for another rearrangement, 
 diverges (see Theorem \ref{P218}).

 To get around this,
 simply do not consider 
 sum rearrangements independently,
 but as a contiguous sum,
 hence when summing consider the order of the sum's terms.

 Since a sum's convergence or divergence is determined at infinity,
 we need only consider a contiguous sum at infinity.

 Since a sequence is a more generic structure than a sum, and a
 sum can be constructed from a sequence by applying a plus operation to adjacent sequence 
 terms, we describe the partitioning of a sequence, and an application
 to sums will follow.
\subsection{Periodic sums} \label{S1602}
A sequence is a more primitive structure than a set and retains
 its order. We first need to develop some notation to partition
 sequences, and sequences at infinity.
 This involves the generalization of the period
 on a contiguous sequence.

 Once this is done, we can in a sweeping move present
 the most general rearrangement theorem for convergence
 sums at infinity, which we call the first rearrangement theorem.
\bigskip
\begin{defy}\label{DEF206}
A `contiguous subsequence' is a subsequence with
 no
 deleted elements between its start
 and end elements.
\end{defy}
\bigskip
\begin{defy}\label{DEF207}
A partition of a sequence is contiguous if
 partitioned into continuous subsequences
 which when joined form the original sequence.
\end{defy}
\bigskip
\begin{mex}\label{MEX213}
 $(1, 2, 3, 4, 5, \ldots) \mapsto ((1,2), (3,4), (5,6) \ldots)$ 
 is a contiguous partition.
\end{mex}

A contiguous partition has the property
 of reversibility.
 If the subsequences are joined together,
 the formed sequence is the original 
 sequence.
\bigskip
\begin{defy}\label{DEF208}
  $(a_{n})= (b_{n})$
  when $(b_{n})$ is a contiguous partition of $(a_{n})$
\end{defy}
\bigskip
\begin{defy}\label{DEF209}
A periodic sequence has fixed length subsequences.
\end{defy}
\bigskip
\begin{defy}\label{DEF210}
A contiguous periodic sequence is a periodic sequence 
 of a contiguous sequence.
\end{defy}
\bigskip
\begin{mex}\label{MEX214}
 $(1, 2, 3, 4, 5, \ldots) = ((1,2), (3,4), (5,6) \ldots)$
 is a contiguous periodic sequence. 
\end{mex}

 We can consider the sequence itself as a contiguous periodic
 sequence with a period of $1$.
 This can then be partitioned into other contiguous periodic sequences.

 Given $(a_{n})|_{n=\infty}$
 then $(a_{2n},a_{2n+1})|_{n=\infty}$.
 Since at infinity, we start counting down by finite integers, both sequences can be put
 into one-one correspondence. 
\bigskip
\begin{prop}\label{P207}
$(a_{n})|_{n=\infty} = (\ldots, a_{n-2}, a_{n-1}, a_{n}, a_{n+1}, a_{n+2}, \ldots)|_{n=\infty}$
\end{prop}
\begin{proof}
At infinity can iterate both forwards and backwards
 as no greatest or least element.
 $(a_{n})|_{n=\infty}$
 $=(a_{n}, a_{n+1}, a_{n+2}, \ldots)|_{n=\infty}$
 $=(\ldots, a_{n-2}, a_{n-2}, a_{n}, a_{n+1}, a_{n+2}, \ldots)|_{n=\infty}$
\end{proof}
\bigskip
\begin{prop}\label{P208}
 $(a_{n})|_{n=\infty}$ can be 
 partitioned with a fixed period $\tau$,
 and a contiguous partition can be formed.
\[ (a_{n})|_{n=\infty} =(a_{\tau n}, a_{\tau n +1}, a_{\tau n + 2}, \ldots, a_{\tau n + \tau-1})|_{n=\infty} \] 
\end{prop}
\begin{proof}
Expand
 $(a_{\tau n}, a_{\tau n +1}, a_{\tau n + 2} + \ldots + a_{\tau n + \tau-1})|_{n=\infty}$ 
 $= \ldots + (a_{\tau n}, a_{\tau n +1}, a_{\tau n + 2} + \ldots + a_{\tau n + \tau-1}) + 
 (a_{\tau (n+1)}, a_{\tau (n+1) +1}, a_{\tau (n+1) + 2} + \ldots + a_{\tau (n+1) + \tau-1}) + \ldots 
|_{n=\infty}$ 
 $= \ldots + (a_{\tau n}, a_{\tau n +1}, a_{\tau n + 2} + \ldots + a_{\tau n + \tau-1}) + 
 (a_{\tau n+\tau}, a_{\tau n+\tau +1}, a_{\tau n+ \tau + 2} + \ldots + a_{\tau n+ 2\tau -1}) + \ldots 
|_{n=\infty}$  \\
  $= (\ldots, a_{\tau n -2}, a_{\tau n-1}, a_{\tau n}, a_{\tau n +1}, a_{\tau n + 2}, \ldots )|_{n=\infty}$.
 Apply 
 Proposition \ref{P207}.
\end{proof}

The concept of the period is extended to include arbitrary contiguous sequences.
 The period is described by a function $\tau(n)$ on the sequence index. 
\bigskip
\begin{defy}\label{DEF211}
 Let $\tau(n)$ describe a `variable periodic sequence'. \\
 $\tau(n) \geq 1$ \\
 $\tau(n+1) - \tau(n) \geq 1$ \\
 $\tau(n)$ contiguously partitions the sequence $(a_{n})$,
 $(a_{\tau(n)}, a_{\tau(n)+1}, \ldots, a_{\tau(n+1)-1})$
\end{defy}
\bigskip
\begin{prop}\label{P209}
 $(a_{n})|_{n=\infty} = (a_{\tau(n)}, a_{\tau(n)+1}, \ldots, a_{\tau(n+1)-1})|_{n=\infty}$
\end{prop}
\begin{proof}
 $(a_{n})|_{n=\infty} = ((a_{\tau(n)}, a_{\tau(n)+1}, \ldots, a_{\tau(n+1)-1}),
 (a_{\tau(n+1)}, a_{\tau(n+1)+1}, \ldots, a_{\tau(n+2)-1}), \ldots )|_{n=\infty}$ 
 $=(a_{\tau(n)}, a_{\tau(n)+1}, \ldots )|_{n=\infty}$ 
 $=(\ldots, a_{\tau(n)-2}, a_{\tau(n)-1}, a_{\tau(n)}, a_{\tau(n)+1}, \ldots )|_{n=\infty}$ 
\end{proof}
 
We now have two ways to classify the partitioning, periodic with fixed $\tau$ or
 a variable period with $\tau(n)$, both of which describe a contiguous partition. 

Construct sums with the same definitions as their associated sequences.
 A series is by definition sequential, applying
 a sum operator to a sequence.
\bigskip
\begin{defy}\label{DEF212}
A `contiguous series' from a series is defined by applying addition to a contiguous subsequence.
\end{defy}
\bigskip
\begin{defy}\label{DEF213}
A `periodic sum' is obtained by applying 
 addition to a periodic sequence.
\end{defy} 
\bigskip
\begin{defy}
A contiguous periodic sum is obtained by
 applying addition to a contiguous periodic sequence.
\end{defy}
\bigskip
\begin{defy}
A variable periodic sum
 is obtained by applying addition to
 a variable periodic sequence.
\end{defy}
\bigskip
\begin{defy}
A contiguous variable periodic sum
 is obtained by applying addition to
 a contiguous variable periodic sequence.
\end{defy}

A periodic sum and variable periodic sum
 are rearrangements of sums.
 The prepending of `contiguous'
 can be omitted,
 for convergence sums will require
 monotonicity.

Consider the series $1 -\frac{1}{2} + \frac{1}{3} - \frac{1}{4} + \frac{1}{5} - \ldots$.
 We notice that the series oscillates,
 continually rising and falling with the positive and
 negative terms added respectively.

 For convergence sums, the criteria requires a monotonic series,
 which clearly the above is not.
 However by considering the order of terms,
 taking two terms at a time, the above series
 is monotonic.  Put another way,
 by considering the order in the rearrangement 
 of the series, if we can transform the series
 to a monotonic series, this can be used to determine
 convergence/divergence.
\bigskip
\begin{mex}\label{MEX215}
Does 
 $1 -\frac{1}{2} + \frac{1}{3} - \frac{1}{4} + \frac{1}{5} - \ldots$ converge or diverge?
 Consider a sum rearrangement
 $(1-\frac{1}{2}) + (\frac{1}{3}-\frac{1}{4}) + \ldots$
 and the successive terms,
 then 
 $(\frac{1}{1}-\frac{1}{2}, \frac{1}{3}-\frac{1}{4}, \frac{1}{5}-\frac{1}{6}, \ldots)$,
 $(\frac{1}{2}, \frac{1}{3\cdot 4}, \frac{1}{5 \cdot 6}, \ldots)$,
 with a fixed period $\tau=2$ we see the sequence
 is strictly monotonically decreasing, and can be tested by
 our convergence criteria.

 Test the sum at infinity, 
 $\sum (-1)^{n+1} \frac{1}{n}|_{n=\infty}$
 $= \sum (\frac{1}{2n} - \frac{1}{2n+1})|_{n=\infty}$
 $= \sum \frac{1}{2n(2n+1)}|_{n=\infty}$
 $=\frac{1}{4}\sum \frac{1}{n^{2}}|_{n=\infty}$
 $=0$ converges, and by
 a contiguous rearrangement Theorem \ref{P210}, 
 $\sum (-1)^{n} \frac{1}{n}|_{n=\infty}=0$
 converges too.
\end{mex}
\bigskip
\begin{mex}
 A period of three convergence example.
 Does $\frac{1}{1}-\frac{1}{3}-\frac{1}{2}+\frac{1}{5}+\frac{1}{7}-\frac{1}{4}+\ldots$ converge or diverge?

 Consider the sum at infinity, the denominators have 
 a sequence $(4k+1, 4k+3, 2k+2)$.
 $\sum (\frac{1}{4k+1}+\frac{1}{4k+3}-\frac{1}{2k+2})|_{k=\infty}$
 $= \sum \frac{8k+5}{32k^{3} + 64k^{2}+38k+6}|_{k=\infty}$
 $= \sum \frac{1}{k^{2}}|_{k=\infty}=0$ converges. 
\end{mex}

If we consider any index sum, brackets
 can be placed about any contiguous series of terms.
 Let $s = a_{1} + a_{2} + a_{3} + \ldots$,
 $s = (a_{1} + a_{2}) + (a_{3} + a_{4} + a_{5}) +
 (a_{6}) + (a_{7}+ \ldots + a_{11}) + \ldots$.
\bigskip
\begin{defy}
 A contiguous sum rearrangement 
 constructs another sum by 
 bracketing sequential terms.
\end{defy}
\bigskip
\begin{defy}
 A contiguous series is a contiguous sum rearrangement.
\end{defy}
\bigskip
\begin{prop}
 If $s'$ is a contiguous sum rearrangement of $s$ then
 $s=s'$.
\end{prop}
\begin{proof}
 $s = a_{1} + a_{2} + \ldots$, $s' = b_{1} + b_{2} + \ldots$.
 Considering $s'$, if we replace $b_{k}$ with
 the original contiguous $a_{k}$ series for each $b_{k}$,
 the original sum $s$ is restored. 
\end{proof}

This idea extends to sums at infinity, where-by
 preserving the order and bracketing the
 terms, we can determine a sum's convergence
 from the rearranged sum. 
 This is necessary,
 as convergence sums require
 monotonic series as input.
\bigskip
\begin{theo}\label{P210}
A contiguous rearrangement theorem. \\ 
If $(b_{n})|_{n=\infty}$
 is a contiguous arrangement of $(a_{n})|_{n=\infty}$
 and $\sum a_{n}|_{n=\infty}$ is a monotonic convergence sum
 then $\sum a_{n} = \sum b_{n}|_{n=\infty}$.
\end{theo}
\begin{proof}
Replace $b_{n}$ by the contiguous $a_{n}$ terms
 restores $\sum a_{n}|_{n=\infty}$. 
\end{proof}

The advantage of the theorem
 is that we need not consider 
 all the different
 rearrangements of the series at infinity,
 one will suffice.
 Once a contiguous rearrangement is
 found which
 can be tested for convergence
 or divergence, all the other
 contiguous rearrangements
 have the same value. For a sum at infinity, by convergence criterion E3,
 this value is an infinireal, and the
 sum either converges or diverges.

A rearrangement of a non-monotonic series to a 
 monotonic series greatly 
 increases the class of 
 functions that can be considered.
\bigskip
\begin{mex}\label{MEX216}
In (Section \ref{S01}) 
 we showed, by comparison of sequential terms, 
 $(\frac{(-1)^{n}}{n^{\frac{1}{2}}-(-1)^{n}})|_{n=\infty}$ to be a non-monotonic sequence,
 and hence the Alternating Convergence Theorem (ACT)
 cannot be used to determine
 if
 $\sum \frac{(-1)^{n}}{n^{\frac{1}{2}}-(-1)^{n}}|_{n=\infty}$ converges or diverges.
 However,
 the fixed
 period contiguous rearrangement theorem
 can be used for just such an event.

$\sum \frac{(-1)^{n}}{n^{\frac{1}{2}}-(-1)^{n}}|_{n=\infty}$
 $= \sum \left( \frac{(-1)^{2n}}{(2n)^{\frac{1}{2}}-(-1)^{2n}} + \frac{(-1)^{2n+1}}{(2n+1)^{\frac{1}{2}}-(-1)^{2n+1}} \right)|_{n=\infty}$
 $= \sum \left( \frac{1}{(2n)^{\frac{1}{2}}-1} - \frac{1}{(2n+1)^{\frac{1}{2}}+1} \right)|_{n=\infty}$
$= \sum \left( \frac{ (2n+1)^{\frac{1}{2}}+1 - ((2n)^{\frac{1}{2}}-1) }{ ((2n)^{\frac{1}{2}}-1)((2n+1)^{\frac{1}{2}}+1) } \right)|_{n=\infty}$
$= \sum \left( \frac{ ( (2n+1)^{\frac{1}{2}} - (2n)^{\frac{1}{2}})+2 }{ ((2n)^{\frac{1}{2}}-1)((2n+1)^{\frac{1}{2}}+1) } \right)|_{n=\infty}$
$= \sum \left( \frac{ 2 }{ ((2n)^{\frac{1}{2}}-1)((2n+1)^{\frac{1}{2}}+1) } \right)|_{n=\infty}$
 $=\sum \frac{2}{2n}|_{n=\infty} = \infty$ diverges.

Using the binomial theorem
 to determine
 the asymptotic 
 result,
 $(2n+1)^{\frac{1}{2}}-(2n)^{\frac{1}{2}}|_{n=\infty}$
 $=(2n)^{\frac{1}{2}}(1+\frac{1}{2n})^{\frac{1}{2}} - (2n)^{\frac{1}{2}}|_{n=\infty}$
 $= (2n)^{\frac{1}{2}}(1+\frac{1}{2}\frac{1}{2n}+\ldots-1)|_{n=\infty}$
 $=(2n)^{\frac{1}{2}}\frac{1}{4n}|_{n=\infty}$
 $= \frac{1}{2^{\frac{3}{2}}n^{\frac{1}{2}}}|_{n=\infty}=0$.
\end{mex}

The following example is not essential,
 but an example of what the theory can do.
\bigskip
\begin{mex}\label{MEX217}
 Partition 
 $(2, 3, 4, \ldots)$
 $= ((2),(3, 4, \ldots, 7), (8, 9, \ldots 20), \ldots)$
 $= ( b_{n} )|_{n=1, 2, \ldots}$ where 
 $b_{n} = (\lfloor e^{n-1}\rfloor+1, \ldots, \lfloor e^{n} \rfloor)$.
 We observe if $y \in b_{n}$, $\lfloor \mathrm{ln}\,y\rfloor = n-1$.
 To prove this consider the following.
 
 $k \in \mathbb{J}$;
 since $e$ is transcendental, $e^{k} \not\in \mathbb{J}$.
 Consider intervals $e^{1}, e^{2}, \ldots, e^{n}$,
 then the least integer before $e^{n}$ is $\lfloor e^{n} \rfloor$.
 The next integer after $e^{n}$ is $\lfloor e^{n} \rfloor+1$.
 Then we can form the sequence
 $( ( \lfloor e \rfloor ), e, ( \lfloor e \rfloor+1, \ldots \lfloor e^{2}\rfloor), e^{2}, (\lfloor e^{2} \rfloor+1, \ldots \lfloor e^{3} \rfloor), e^{3}, ( \lfloor e^{3} \rfloor+1, \ldots \lfloor e^{4} \rfloor), e^{4}, \ldots )$
 By considering $e^{n-1} \lt \lfloor e^{n-1}\rfloor+1, \ldots, \lfloor e^{n} \rfloor \lt e^{n}$,
 apply $\lfloor \mathrm{ln}(x) \rfloor$ to the previous sequence,
 $n-1 \leq \lfloor \mathrm{ln}(\lfloor e^{n-1}\rfloor+1) \rfloor, \ldots, \lfloor \mathrm{ln}(\lfloor e^{n} \rfloor)  \rfloor \lt n$.

\cite[3.4.9]{kaczor}
 By considering a variable period between powers of $e$,
 the following inner sum is simplified.
$\sum \frac{ (-1)^{ \lfloor \mathrm{ln}\,n \rfloor } }{ n}|_{n=\infty}$
$= \sum  (-1)^{n} \sum_{k= \lfloor e^{n-1}+1 \rfloor}^{e^{n}} \frac{1}{ k}|_{n=\infty}$,
 focussing on the inner sum, 
 $\sum_{k= \lfloor e^{n-1}+1 \rfloor}^{e^{n}} \frac{1}{ k}|_{n=\infty}$
 $=\int_{ e^{n-1}+1 }^{e^{n}} \frac{1}{ k} \, dk|_{n=\infty}$
 $= \mathrm{ln} \, k |_{ e^{n-1}+1 }^{e^{n}} |_{n=\infty}$ 
 $ = n - (n-1)|_{n=\infty} =1$.
 Then 
 $\sum \frac{ (-1)^{ \lfloor \mathrm{ln}\,n \rfloor } }{ n}|_{n=\infty}$
 $= \sum  (-1)^{n}1 = \infty$ diverges.
\end{mex}
\subsection{Tests for convergence sums} \label{S1603}
 The following discussion concerns our
 second rearrangement theorem, converting
 monotonic functions to strictly monotonic functions.
\bigskip
\begin{defy}
A `subfunction' is a function formed from
 a given function by deleting intervals or points
 without changing the relative position of the
 intervals or points.
\end{defy}
\bigskip
\begin{defy}
Let 
 $\mathrm{unique}(a)$
 define a function where
 input $a$ is monotonic, returns a function
 with only unique values.
 If $b = \mathrm{unique}(a)$
 then the
 returning function $b$ is strictly monotonic.
 $a$ can be a sequence or function.
\end{defy}
\bigskip
\begin{mex}
$a = (1, 3, 3, 5, 7, 7, 9)$, $b = \mathrm{unique}(a)$
 then $b=(1, 3, 5, 7, 9)$.
 $a$ is monotonically increasing, $b$ is
 strictly monotonic increasing.
\end{mex}

If $b(n) = \mathrm{unique}(a(n))$
 then plateaus in $a(n)$ would be
 removed in $b(n)$.
For continuous $a(n)$ a function has the 
 interval with equality collapsed to a single point.
 For sequence $a_{n}$, multiple values
 of $a_{n}$ would be collapsed to a single unique $a_{n}$.

The following rearrangement theorem of a sum
 at infinity is given. It is important
 because it allows the reduction of
 a monotonic sequence to a strictly monotonic
 sequence for convergence testing.

 This also makes theory easier, for example we can
 construct a strictly monotonic power series
 through a strictly monotonic sequence, but not
 through a monotonic sequence. 
\bigskip
\begin{theo}\label{P212}
The second rearrangement theorem:  

 If $a$ is a monotonic function or sequence,
 $b = \mathrm{unique}(a)$, 
 then for series $\sum a_{n} = \sum b_{n}|_{n=\infty}$
 or for integrals $\int a(n)\,dn = \int b(n)\,dn|_{n=\infty}$
 when $\mathbb{R}_{\infty} \mapsto \{0,\infty\}$.
\end{theo}
\begin{proof}
By E3.5 and E3.6 (Section \ref{S0108}),
  if a series or integral plateaus,
 this is a localized event. 

\textbf{E3.5} For divergence,
 $\int a(n)\,dn|_{n=\infty}$ 
 or $\sum a_{n}|_{n=\infty}$
 can be made arbitrarily large. E3.5 (Section \ref{S0108})

\textbf{E3.6} For convergence, $\int a(n)\,dn|_{n=\infty}$
 or $\sum a_{n}|_{n=\infty}$
 can be
 made arbitrarily small. E3.6] (Section \ref{S0108})

Let $b = \mathrm{unique}(a)$.
 For sums, $\sum a_{n} = \sum b_{n} + \sum c_{n}|_{n=\infty}$ where $c_{n}$ are the deleted terms from $a_{n}$.
 $\sum c_{n}|_{n=\infty} \geq 0$.

$\sum b_{n} \geq \sum c_{n}|_{n=\infty}$
 as placing the series on one-one correspondence,
 $b_{n} \geq c_{n}|_{n=\infty}$.
 Consequently if $\sum a_{n}|_{n=\infty}=\infty$ diverges
 then $\sum b_{n}|_{n=\infty}=\infty$ diverges.
 If $\sum a_{n}|_{n=\infty}=0$ converges
 then $\sum b_{n}|_{n=\infty}$ converges.

The same argument is made for continuous $a(n)$
 and integrals. If $b(n) = \mathrm{unique}(a(n))$
 then $b(n)$ is a subfunction of $a(n)$. 
 $\int a(n)\,dn|_{n=\infty} = \int b(n)\,dn + \int c(n)\,dn|_{n=\infty}$ where $c(n)$ is a subfunction of $a(n)$
 and $\int b(n)\,dn \geq \int c(n)\,dn|_{n=\infty}$,
 again a one-one correspondence argument. $\int c(n)\,dn \geq 0$.
 If $\int a(n)\,dn|_{n=\infty}=\infty$ diverges
 then $\int b(n)\,dn|_{n=\infty}=\infty$ diverges.  
 If $\int a(n)\,dn|_{n=\infty}=0$ converges 
 then $\int b(n)\,dn|_{n=\infty}=0$ converges.  
\end{proof}
With convergence sums, we have, in the same way
 as (Section \ref{S01}), looked at existing tests and
 theorems, and asked how can we do this with
 partitioning at infinity, and using
 a our non-standard analysis.

For while \cite[Knopp]{knopp} is one of the
 best expositors, he does not use the infinite, 
 but $\epsilon$.

We raise this more as a point of difference than fact.
 A personal or subjective taste, than of necessity.
 However, partitioning at infinity,
 and a determination to reason there, have
 allowed us personally to find an alternative and
 easier way to reason. It may take another
 equally good expositor like Knopp to communicate
 this.

 We argue that a better way of reasoning at infinity is possible,
 but that it may take time and some other technical
 developments.  
 We have taken some of the known sum theorems
 described by Knopp, and applied our
 ideas of the infinite for alternative
 theorems at infinity. See Theorems \ref{P214} and \ref{P217}.
\bigskip
\begin{theo}\label{P213}
The Absolute convergence test (Theorem \ref{P015}).
 If $\sum |a_{n}||_{n=\infty}=0$
 then $\sum a_{n}|_{n=\infty}=0$
 was found to be a sum rearrangement theorem.
\end{theo}
\bigskip
\begin{theo}\label{P214}
``If $\sum a_{\nu}$ is an absolutely convergent series,
 and if $\sum a_{\nu}'$ is an arbitrary rearrangement of it,
 then this series is also convergent, 
 and both series have the same value." \cite[p.79 Theorem 4]{knopp}
\end{theo}
\begin{proof}
 $n_{0} = \mathrm{min}(+\Phi^{-1})$;
 $n_{1} = \mathrm{max}(+\Phi^{-1})$;
 $\sum a_{\nu} = \sum a_{\nu}|_{\nu \lt \infty} + \sum_{n_{0}}^{n_{1}} a_{\nu}$.
 $\sum a_{\nu}' = \sum a_{\nu}'|_{\nu \lt \infty} + \sum_{n_{0}}^{n_{1}} a_{\nu}'$.
 By Theorem \ref{P213}, $\sum_{n_{0}}^{n_{1}} a_{\nu}=0$
 then
 $\sum_{n_{0}}^{n_{1}} a_{\nu}'=0$.
 Since a rearrangement of a finite series has the same value,
 $\sum a_{\nu}|_{\nu \lt \infty} = \sum a_{\nu}'|_{\nu \lt \infty}$.
 Let $\sum a_{\nu}|_{\nu \lt \infty} = s$.
 Then $\sum a_{\nu} = s + 0=s$.
 $\sum a_{\nu}' = s + 0=s$.
\end{proof}

We emphasize rearrangements of infinite series
 relate to sampling a series at different rates,
 and their interest is also when the sampling rate
 does not matter. Hence the importance of the
 absolute convergence theorem.
\bigskip
\begin{lem}\label{P215}
 If we partition a conditionally
 convergent sum, $\sum a_{n}|_{n=\infty}=0$,
 $\sum |a_{n}||_{n=\infty}=\infty$, 
 into positive and
 negative sums, then one of these sums will diverge.
\end{lem}
\begin{proof}
 Let the two sums, of positive
 and negative terms, have sequences $(c_{n})|_{n=\infty}$
 and $(d_{n})|_{n=\infty}$:
 $|\sum c_{n}| \leq |\sum d_{n}||_{n=\infty}$.

  $\sum |c_{n}| \leq \sum |d_{n}||_{n=\infty}$,
  $\sum |c_{n}| + |\sum d_{n}| \leq \sum |d_{n}| + | \sum d_{n}| |_{n=\infty}$,
  $\sum |c_{n}| + \sum |d_{n}| \leq 2 \sum |d_{n}||_{n=\infty}$,
  $\sum |a_{n}| \leq 2 \sum |d_{n}||_{n=\infty}$,
  $\infty \leq 2 \sum |d_{n}||_{n=\infty}$,
 $| \sum d_{n}||_{n=\infty} = \infty$ diverges.
\end{proof}
\bigskip
\begin{prop}\label{P012}
 If we partition a conditionally
 convergent sum, $\sum a_{n}|_{n=\infty}=0$,
 $\sum |a_{n}||_{n=\infty}=\infty$, 
 into positive and
 negative sums, then both of these sums will diverge.
\end{prop}
\begin{proof}
By Lemma \ref{P215},
 let $\sum d_{n}|_{n=\infty}=\pm \infty$ be the divergent sum,
 when
 $\sum a_{n}|_{n=\infty} = \sum c_{n}|_{n=\infty} + \sum d_{n}|_{n=\infty}=0$ converges.
 Case $\sum d_{n}|_{n=\infty} = +\infty$, $+\infty + \sum c_{n}|_{n=\infty}=0$.
 $\sum c_{n}|_{n=\infty} = -\infty$ then $\sum c_{n}|_{n=\infty}$ diverges. 
 Similarly if $\sum d_{n}|_{n=\infty} = -\infty$, $\sum c_{n}|_{n=\infty}=+\infty$ diverges. 
\end{proof}
\bigskip
\begin{theo}\label{P216}
Riemann’s Rearrangement Theorem 
 If $\sum_{k=1}^{\infty} a_{k}$ is conditionally convergent, and $\alpha$ a given real number.
 Then there exists a rearrangement of the terms in $\sum_{k=1}^{n} a_{n}$ whose terms sum to $\alpha$.
\end{theo}
\begin{proof}
 The proof is almost the same as \cite{galanor},
 but with the $*G$ number system.
 Construct an algorithm that leads to
 a sum infinitesimally close to $\alpha$, which,  
 when transferred back to $\mathbb{R}$ is equal
 to $\alpha$. 
 Partition $a_{n}$ respectively into positive and negative sequences, monotonically
 decreasing in magnitude, $(c_{n})$ and $(d_{n})$.  
 Have integer variables $i$ for the current index into $(c_{n})$ and $j$ for 
 the current index into $(d_{n})$.
 $s_{0}=0$;
 If $s_{n} \lt \alpha$ then $s_{n+1} = c_{i} + s_{n}$
 and $i=i+1$ increments $i$;
 else $s_{n+1} = d_{j} + s_{n}$ and $j=j+1$ increments $j$.
 $s_{n}|_{n=\infty} \simeq \alpha$ generates a sum infinitesimally close to $\alpha$.
\end{proof}
\bigskip
\begin{theo}\label{P217}
``If $\sum a_{\nu}$ is a convergent, but not
 an absolutely convergent, series,
 then there are arrangements,
 $\sum a_{\nu}'$, of it that
 diverge." \cite[p.80 Theorem 5]{knopp}
\end{theo}
\bigskip
\begin{theo}\label{P218}
  If $\sum a_{n}|_{n=\infty}$
 is conditionally convergent,
 $\sum a_{n}|_{n=\infty}=0$,
 $\sum |a_{n}||_{n=\infty}=\infty$,
 then there exists rearrangements
 of it such that 
 $\sum a_{n}'|_{n=\infty}=\infty$ diverges.
\end{theo}
\begin{proof}
 By construction. 
 Using Lemma \ref{P215} we can partition $(a_{n})$
 into $(c_{n})$ and $(d_{n})$ where $\sum d_{n}|_{n=\infty}=\pm\infty$
 diverges and the sign is dependent on whether the
 negative or positive group diverges.

 Choose $b_{n} = c_{n}+ \sum_{k=k_{0}}^{k_{1}} d_{k}$,
 where $\sum_{k=k_{0}}^{k_{1}} d_{k}$
 is a contiguous sum of $d_{n}$ terms and 
 $|b_{n}|$ has a lower real bound $\alpha \neq 0$.
 Since $\sum d_{n}=\pm\infty$ this is always possible.
 As we can arbitrarily increase $k_{1}$, many of these
 $b_{n}$ rearrangements are possible.
 Have $b_{n}$ either all positive or all negative.

 Case $b_{n} \gt 0$.
 $b_{n} \geq \alpha$, $\sum b_{n} \geq \sum \alpha|_{n=\infty}$,
 $\sum b_{n}|_{n=\infty} \geq \infty$,
 $\sum b_{n}|_{n=\infty}=\infty$ diverges.
 Case $b_{n} \lt 0$.
 $b_{n} \leq \alpha$, $\sum b_{n} \leq \sum \alpha|_{n=\infty}$,
 $\sum b_{n}|_{n=\infty} \leq -\infty$,
 $\sum b_{n}|_{n=\infty}=-\infty$ diverges.
 Many rearrangements which diverge were found.
\end{proof}
 
We now explore the chain rule,
 which varies the rate of counting,
 dependent on the integration variable.
\bigskip
\begin{mex}\label{MEX218}
$\int \frac{1}{1+2n}dn|_{n=\infty}$,
 let $u = 1+2n$, $\frac{du}{dn}=2$
 then $\int \frac{1}{u}dn$
 $= \int \frac{1}{u}\frac{dn}{du}du$
 $=\frac{1}{2}\int\frac{1}{u}du$
 $=\frac{1}{2}\mathrm{ln}\,u|_{n=\infty}$
 
Consider the same integral, but integrate
 with another variable.
 We find the result
 of a change in unequal variables
 changed the integral result.

 $\int \frac{1}{1+2n}dn$
 $= 2 \int \frac{1}{2+4n}dn$,
 let $v = 2+4n$, $\frac{dv}{dn} = 4$,
 $2 \int \frac{1}{v}dn$
 $=2 \int \frac{1}{v}\frac{dn}{dv}dv$
 $= \frac{1}{2}\mathrm{ln}\,v|_{n=\infty}$

However
 $\frac{1}{2}\mathrm{ln}\,u \neq \frac{1}{2}\mathrm{ln}\,v|_{n=\infty}$
 as $1+2n \neq 2 + 4n|_{n=\infty}$.
\end{mex}
\bigskip
\begin{conjecture}
 A periodic sum $\tau$ can have
 its sums interchanged, if the change of
 variable stays the same.
\[ \sum a_{n} = \sum \sum_{k=1}^{\tau} b_{k,n}|_{n=\infty} = \sum_{k=1}^{\tau}\sum b_{k,n} = \sum_{k=1}^{\tau} \int b_{k}(n)\,dn|_{n=\infty} \]
\end{conjecture}

The sums instead of being summed contiguously,
 are rearranged into period columns.
 If these columns can be treated independently,
 this is advantageous as we can use 
 integrals to evaluate them.

However, because of our usage of the chain rule,
 we find here an example where we cannot
 apply the chain rule to one part of the
 partition differently to the others.
 Where infinity is concerned, we
 need to be more cautious. 
\bigskip
\begin{mex}\label{MEX219}
Does $1-\frac{1}{2} - \frac{1}{4} + \frac{1}{3} - \frac{1}{6} - \frac{1}{8} + \frac{1}{5} - \ldots$ converge or diverge?
 \cite[p.105 3.7.2]{kaczor} solves the sum;
 however we use this example to demonstrate theory.

$\sum ( \frac{1}{1+2n} - \frac{1}{2+4n} - \frac{1}{4+4n})|_{n=\infty}$
 $=\int \frac{1}{1+2n}dn - \int \frac{1}{2+4n}dn - \int \frac{1}{4+4n} dn|_{n=\infty}$
 $= \frac{1}{2} \mathrm{ln}(1+2n) - \frac{1}{4}\mathrm{ln}(2+4n) - \frac{1}{4} \mathrm{ln}(4+4n)|_{n=\infty}$
$= \frac{1}{4} ( 2 \, \mathrm{ln}(1+2n) - \mathrm{ln}(2+4n) - \mathrm{ln}(4+4n))|_{n=\infty}$
 $= \frac{1}{4} \mathrm{ln} \frac{ (1+2n)^{2}}{8(1+2n)(n+1)}|_{n=\infty}$
 $= \frac{1}{4} \mathrm{ln} \frac{ 1+2n}{8(n+1)}|_{n=\infty}$
 $=\frac{1}{4} \mathrm{ln}\frac{1}{4}$ 
 $=-0.34657359027997\ldots$
 $\neq 0$ or $\infty$ which 
 contradicts Criterion E3.

 We noted in Example \ref{MEX218} a different
 change in variable produced a
 different result.
 Keeping all the sums with variable $4n$,
 perform the same integration as before.

Integrating with the same rate for all sums.
$\sum ( \frac{1}{1+2n} - \frac{1}{2+4n} - \frac{1}{4+4n})|_{n=\infty}$
$=\sum ( \frac{2}{2+4n} - \frac{1}{2+4n} - \frac{1}{4+4n})|_{n=\infty}$
$=\sum ( \frac{1}{2+4n} - \frac{1}{4+4n})|_{n=\infty}$
$=\int \frac{1}{2+4n}dn - \int \frac{1}{4+4n}dn|_{n=\infty}$
 $= \frac{1}{4}(\mathrm{ln}(2+4n) - \mathrm{ln}(4+4n))|_{n=\infty}$
 $= \frac{1}{4} \mathrm{ln} \frac{2+4n}{4+4n}|_{n=\infty}$
 $=0$ converges.
 The sum agrees with the theory, and the
 correct result is found.
\end{mex}
\section{Ratio test and a generalization with convergence sums} \label{S17}
For positive series
 convergence sums we generalise the ratio test
 in $*G$ the gossamer numbers.
 Via a transfer principle, within the tests we construct variations.
 However, most significantly we connect and show
 the generalization to be equivalent to the boundary test. 
 Hence, the boundary test includes the generalized tests: the ratio test, Raabe's test,
 Bertrand's test and others.
\subsection{Introduction}\label{S1701}
Of the convergence sums (Section \ref{S01}) $\sum a_{n}|_{n=\infty}$, there exists
  two generalizations of convergence and divergence
 tests
 involving 
 the ratio of successive terms.
 Both are equivalent, and different expansions of
 $\frac{a_{n}}{a_{n+1}}$ and $\frac{a_{n+1}}{a_{n}}$.

 For example consider Raabe's test(Theorem \ref{P221}) for convergence.
 If $n (\frac{ a_{n}}{a_{n+1}}-1)|_{n=\infty} \gt 1$ is associated
 with the $\frac{a_{n}}{a_{n+1}}$ generalization,
 and 
 $n (\frac{ a_{n+1}}{a_{n}}-1)|_{n=\infty} \lt -1$ is
 associated with the $\frac{a_{n+1}}{a_{n}}$ generalization.
 Both tests can be rearranged to show the other.

 We find by considering the sum 
 in the more detailed number system with
 infinitesimals and infinities in $*G$ \cite[Part 1]{cebp21},
 that we can multiply and divide terms in the ratio and generalized tests,
 thus modifying the tests. For example, express
 the ratio test not as a ratio, but as a comparison without fractions,
 with no denominators.

 In the final comparison of the test,
 the numbers are projected to the extended real numbers,
 $*G \mapsto \overline{\mathbb{R}}$. The extension is used to include cases, where for example
 in the ratio test, the ratio is $0$ where the denominator is an infinity. Multiplying
 the denominator out results in a comparison of two numbers which differ by an infinity,
 hence cannot be compared in the reals.  

 While the ratio test is phrased as being in $\mathbb{R}$,
 in actuality the ratio test is a ratio between infinitesimals
 and infinities,
 none of which exist in $\mathbb{R}$ and the ratio is projected back to $\mathbb{R}$.

 The standard ratio test does this via the limit, where the infinitesimals and infinities
 are realised in the ratio via a transfer \cite[Part 4]{cebp21}. 
 For example, $\frac{1}{n+1}$ and $\frac{1}{n}$ are infinitesimals,
 their limit is $1$ and corresponds to 
 the indeterminate case for the ratio test. 
 Their ratio, $\frac{n}{n+1} = \frac{n+1}{n+1}-\frac{1}{n+1}$
 $= 1 - \frac{1}{n+1}_{n=\infty} = 1$.
 The infinitesimal is realized in this process.
 
 Working in $*G$ gives the flexibility to consider
 the ratio test as not set in stone,
 but where terms can be multiplied and divided.
 This is something which the real number system alone, is not suited
 to, because it does not have infinitesimals or a transfer principle.
 Though, as we have seen, the transfer principle is applied
 via limits.

 Finally we make a direct connection with this generalized ratio
 test and the boundary test (See Section \ref{S18}),
 and show their
 equivalence.
 
 The boundary test can be proved from the generalized ratio test,
 or vice versa.
\subsection{The ratio test and variations} \label{S1702}
 We now consider the tests
 from the number system's
 perspective.
 Consider the ratio test and generalizations
 in $*G$. The test can be algebraically 
 reformed, and a transfer principle used to
 apply back to $\mathbb{R}$.

A sum $\sum a_{n}|_{n=\infty}$,
 by having a negative gradient $\frac{d a_{n}}{dn} \lt 0$
 and being positive,
 does not have to converge.
 However, the ratio test is in part a gradient test;
 we can transform
 the test to the continuous variable as a first derivative
 test. The ratio test is, therefore both a gradient and a magnitude test,
 the gradient being a necessary but not a sufficient condition.

 We say $(\overline{\mathbb{R}}, \lt )$ to 
 mean $(*G, \lt) \mapsto (\overline{\mathbb{R}}, \lt)$,
 as the algebra is in $*G$ and transferred to
 $\overline{\mathbb{R}}$,
 with $*G \mapsto \overline{\mathbb{R}}$ as the last step.
 By symmetry of the relation $\lt$,
 $(\overline{\mathbb{R}},\lt)$ implies
 $(\overline{\mathbb{R}},\gt)$.

 These examples demonstrate the equivalence of the ratio test
 and modified ratio test.
\bigskip
\begin{mex}
Consider $\sum \frac{1}{n}$ by the ratio test, Theorem \ref{P211}.
 Let $a_{n} = \frac{1}{n}$,
 $\frac{a_{n+1}}{a_{n}}|_{n=\infty}$
 $= \frac{n}{n+1}|_{n=\infty}$
 $=1$ the indeterminate case.

Using the modified ratio test, Theorem \ref{P219}, 
 $a_{n+1} \; z \; a_{n}$,
 $\frac{1}{n+1} \lt \frac{1}{n}$, however, transferring
 this to $\mathbb{R}$,
 $0 \lt 0$ is a contradiction, hence this is also
 an indeterminate case.
 $(*G, \lt) \not \mapsto (\mathbb{R},\lt)$.
\end{mex}
\bigskip
\begin{mex}
 By the ratio test,
 $a_{n} = \frac{1}{e^{n}}$,
 $\frac{a_{n+1}}{a_{n}}|_{n=\infty}$
 $=\frac{e^{n}}{e^{n+1}}|_{n=\infty}$
 $=\frac{1}{e} \lt 1$ converges.

By the modified ratio test, 
 $a_{n+1} \; z \; a_{n}$,
 $\frac{1}{e^{n+1}} \lt \frac{1}{e^{n}}|_{n=\infty}$,
 $\frac{1}{e} \lt 1$, converges.
 $(*G, \lt) \mapsto (\mathbb{R},\lt)$.
\end{mex}

By considering the ratio test in a higher dimension,
 $*G$ with infinitesimals and infinities,
 the test does not have
 to be as a ratio. We can multiply
 and divide the terms,
 then by a transfer principle realize the
 test in $\overline{\mathbb{R}}$.
 Since the variations 
 may be used as convergence tests, all have been
 stated as theorems.
\bigskip
\begin{theo}\label{P211}
 Let $a_{n} \in *G$, $(\overline{\mathbb{R}},\lt)$.
\[ \text{If } \frac{a_{n+1}}{a_{n}} \lt 1 \text{ then } \sum a_{n}|_{n=\infty}=0 \text{ converges.} \]
\[ \text{If } \frac{a_{n+1}}{a_{n}} \gt 1 \text{ then } \sum a_{n}|_{n=\infty}=\infty \text{ diverges.} \]
\end{theo}
\begin{proof}
Given $z \in \{ \lt, \gt \}$,
 in $*G$,
 $\frac{a_{n+1}}{a_{n}}|_{n=\infty} \; z \; 1$,
 $a_{n+1} \; z \; a_{n}|_{n=\infty}$,
 apply Theorem \ref{P219}.
\end{proof}
\bigskip
\begin{theo}\label{P219}
 $a_{n} \in *G$;
 $\,(\overline{\mathbb{R}},\lt)$.
\[ \text{If } a_{n+1} \lt a_{n}  \text{ then } \sum a_{n}|_{n=\infty}=0 \text{ converges.} \]
\[ \text{If } a_{n+1} \gt a_{n}  \text{ then } \sum a_{n}|_{n=\infty}=\infty \text{ diverges.} \]
\end{theo}
\begin{proof}
Given $z \in \{ \lt, \gt \}$,
 in $*G$,
 $a_{n+1} \; z \; a_{n}|_{n=\infty}$,
 $a_{n+1}-a_{n} \; z \; 0|_{n=\infty}$, 
 $\frac{ d a_{n} }{dn}|_{n=\infty} \; z \; 0$,
 convert to the continuous domain,
 $\frac{ d a(n) }{dn}|_{n=\infty} \; z \; 0$,
 apply Theorem \ref{P220}.
\end{proof}
\bigskip
\begin{theo}\label{P220}
 $a(n) \in *G$; $\,(\overline{\mathbb{R}},\lt)$. 
\[ \text{If } \frac{d a(n)}{dn}|_{n=\infty} \lt 0 \text{ then } \int a(n)\,dn|_{n=\infty}=0 \text{ converges.} \]
\[ \text{If } \frac{d a(n)}{dn}|_{n=\infty} \gt 0 \text{ then } \int a(n)\,dn|_{n=\infty}=\infty \text{ diverges.} \]
\end{theo}
\begin{proof}
 Substitute $m=1$ into Theorem \ref{P224} which is equivalent
 to Theorem \ref{P225} with the inequalities inverted.
 The equality case is discarded.
\end{proof}
\begin{proof}
 Although more complex, we find another proof
 combining integrating over relations and
 the transfer condition.

 Consider a particle
 undergoing constant deceleration, where the particle cannot
 move backwards it will stop (In the infinitesimal domain,
 the particle can still be moving).
 The area swept by the particle has similarly stopped.

 Expressing the conditions.
 Let $s(n) = \int a(n)\,dn$.
 Deceleration 
  in $\mathbb{R}:\, \frac{d^{2} s(n)}{d n^{2}} \lt 0$.
 The particle can only move forward.
 In $*G: \frac{d s(n)}{dn} \geq 0$.
\begin{align*}
 \frac{d^{2} s(n)}{d n^{2}}|_{n=\infty} \lt 0 \tag{Integrating} \\
 0 \leq \frac{d s(n)}{dn}|_{n=\infty} \lt c \tag{$c$ is positive, integrating} \\
 0 \leq s(n) \lt cn + c_{2}|_{n=\infty} \tag{$cn|_{n=\infty} \succ c_{2}$} \\
 0 \leq s(n) \lt cn|_{n=\infty} \tag{transfer preserving inequality as
 $(\overline{\mathbb{R}},\lt)$} \\
 0 \leq s(n)|_{n=\infty} \lt \infty \tag{$s(n)$ converges} \\
\end{align*}
Consider when the particle is under constant acceleration.
\begin{align*}
 \frac{d^{2} s(n)}{d n^{2}}|_{n=\infty} \gt 0 \tag{Integrating} \\
 \frac{d s(n)}{dn}|_{n=\infty} \geq c \tag{$c$ is positive, integrate} \\
 s(n) \geq cn + c_{2}|_{n=\infty} \tag{$cn+c_{2} = cn|_{n=\infty}$} \\
 s(n) \geq \infty \tag{$s(n)$ diverges}
\end{align*}
\end{proof}
 By threading a continuous function through
 the monotonic sequence $a_{n}$ we can show the above
 to be the ratio test.
 The gradient of $a(n)$ is the curvature of $s(n)$.

Consider a sum of positive terms.
 Then the sum, by always having terms added, is increasing.
 Threading a continuous function through the series,
 the function $a(n)$ is always positive.
 $s(n) = \int a(n)\,dn$, $\frac{ d s(n)}{dn} = \frac{d}{dn} \int a(n)\,dn$
 $= a(n) \gt 0$ is true in $*G$.

 This is the continuous version of 
 a sum with a negative sequence derivative,
 $\frac{a_{n+1}}{a_{n}} \lt 0$, 
 $a_{n+1} \lt a_{n}$,
 $a_{n+1} -a_{n} \lt 0$,
 $\frac{ d a_{n}}{dn} \lt 0$,
 $\frac{ d a(n)}{dn} \lt 0$.
 The area or distance traveled by the particle is finite,
 and in the same way the sum is finite and converges.
\bigskip
\begin{mex}\label{MEX221}
$\sum \frac{1}{n}|_{n=\infty}=\infty$ is known to
 diverge.
 The ratio test fails to determine
 convergence.
 Let $a_{n} = \frac{1}{n}$,
 $\frac{a_{n+1}}{a_{n}}|_{n=\infty}$
 $=\frac{n}{n+1}|_{n=\infty}=1$
 is indeterminate.

In working with the higher dimension $*G$ which includes
 the infinireals, 
 when we realize and apply the tests,
 a less than relationship with infinitesimals
 is not a less than relationship in $\mathbb{R}$. 

\begin{align*}
 a_{n+1} \; & \; z \; a_{n}|_{n=\infty} \tag{Theorem \ref{P219}} \\
 \frac{1}{n+1} & \; z \; \frac{1}{n}|_{n=\infty} \\
 \frac{1}{n+1} & \lt \frac{1}{n}|_{n=\infty} \tag{Realizing the infinitesimals} \\
 0 & \lt 0 \text{ contradicts} \tag{Indeterminate result} \\
 \tag{Alternatively multiply the denominators out.} \\
 n & \lt n+1|_{n=\infty} \tag{Realizing the infinities} \\
 \infty & \not\lt \infty \tag{Indeterminate result}
\end{align*}
\end{mex} 
 The tests are the same and in their variation
 almost trivially similar
 to the classic ratio test. However it is nice to
 do things in different ways.

 The limit ratio test, in its application can be varied as
 a ratio expression, multiplying and dividing
 the numerator and denominator. Rather than
 seeing the test set in stone, you can manipulate it.
 At times this is trivial, in other instances this becomes
 a way to transform tests.
\bigskip
\begin{mex}\label{MEX220}
Determine the convergence or divergence of 
 $\sum \frac{1\cdot 3 \cdot \ldots (2n-1)}{3 \cdot 6 \cdot \ldots (3n)}|_{n=\infty}$.

Let $a_{n} =  \frac{1\cdot 3 \cdot \ldots (2n-1)}{3 \cdot 6 \cdot \ldots (3n)}$, 
 $a_{n+1} \; z \; a_{n}|_{n=\infty}$,
 $\frac{1 \cdot 3 \cdot 5 \cdot \ldots ( 2(n+1)-1) }{3 \cdot 6 \cdot \ldots (3(n+1))} \; z \; \frac{1\cdot 3 \cdot \ldots (2n-1)}{3 \cdot 6 \cdot \ldots (3n)}|_{n=\infty}$,
 $\frac{2n+1}{3n+3} \; z \; 1|_{n=\infty}$,
 $2n+1 \; z \; 3n+3|_{n=\infty}$
 $1 \lt n+3|_{n=\infty}$
 and
 by Theorem \ref{P219} the series converges.
\end{mex}

When $\frac{ a_{n+1}}{a_{n}}|_{n=\infty}=1$, expressed as $a_{n+1} \simeq a_{n}$, use Raabe's test.
\bigskip
\begin{theo}\label{P221} 
 Raabe's test 1. \\
\[
n (\frac{ a_{n}}{a_{n+1}}-1)|_{n=\infty}
 = \left\{ 
  \begin{array}{rl}
    \gt 1 & \; \text{then } \sum a_{n}|_{n=\infty}=0 \text{ is convergent,} \\
    \lt 1 & \; \text{then } \sum a_{n}|_{n=\infty}=\infty \text{ is divergent.}
  \end{array} \right. 
\]
\end{theo}
\begin{proof}
 Rearrange expression to one line. Let $z \in \{ \lt, \gt \}$.
 $n (\frac{ a_{n}}{a_{n+1}}-1)|_{n=\infty} \; z \; 1$, 
 $n a_{n} - n a_{n+1} \; z \; a_{n+1}$,
 $n a_{n} - (n+1) a_{n+1} \; z \; 0$,
 prove by Theorem \ref{P222}.
\end{proof}
\bigskip
\begin{theo}
 Raabe's test 2. \\
\[
n (\frac{ a_{n+1}}{a_{n}}-1)|_{n=\infty}
 = \left\{ 
  \begin{array}{rl}
    \lt -1 & \; \text{then } \sum a_{n}|_{n=\infty}=0 \text{ is convergent,} \\
    \gt -1 & \; \text{then } \sum a_{n}|_{n=\infty}=\infty \text{ is divergent.}
  \end{array} \right. 
\]
\end{theo}
\begin{proof}
 Rearrange expression to one line. Let $z \in \{ \lt, \gt \}$.
 $n(\frac{a_{n+1}}{a_{n}}-1) \; z \; -1$,
 $n(a_{n+1}-a_{n}) \; z \; - a_{n}$,
 $n a_{n+1} - (n-1) a_{n} \; z \; 0$,
 relable index $a_{n+1}$ to $a_{n}$,
 $n a_{n} - (n-1) a_{n-1} \; z \; 0$,
 $(n+1) a_{n+1} - n a_{n} \; z \; 0$,
 $n a_{n} - (n+1) a_{n+1} \; (-z) \; 0$,
 and prove
 by Theorem \ref{P222}.
\end{proof}
\bigskip
\begin{theo}\label{P222}
 Raabe's test 3. 
 In $*G$ and $(\overline{\mathbb{R}}, \lt)$.
\[
n a_{n} - (n+1) a_{n+1}|_{n=\infty}
 = \left\{ 
  \begin{array}{rl}
    \gt 0 & \; \text{then } \sum a_{n}|_{n=\infty}=0 \text{ is convergent,} \\
    \lt 0 & \; \text{then } \sum a_{n}|_{n=\infty}=\infty \text{ is divergent.}
  \end{array} \right. 
\]
If $n a_{n} - (n+1) a_{n+1} \gt 0|_{n=\infty}$ then $\sum a_{n}|_{n=\infty}=0$ is convergent. 
If $n a_{n} - (n+1) a_{n+1} \lt 0|_{n=\infty}$ then $\sum a_{n}|_{n=\infty} = \infty$ is divergent.
\end{theo}
\begin{proof}
 $m=0$ in Theorem \ref{P225},
 $\frac{ a_{n}}{a_{n+1}} \; z \; 1 + \frac{1}{n}$,
 $\frac{n a_{n}}{a_{n+1}} \; z \; n+1$,
 $n a_{n} - (n+1) a_{n+1} \; z \; 0$.
 Case $z = \; \gt$ converges.
 $z = \; \lt$ diverges.
\end{proof}
\bigskip 
\begin{theo}\label{P223}
 See \cite[3.2.16]{kaczor},
 reformed with at-a-point notation.
\[
n \, \mathrm{ln} \frac{a_{n}}{a_{n+1}}|_{n=\infty}
 = \left\{ 
  \begin{array}{rl}
    \gt 1 & \; \text{then } \sum a_{n}|_{n=\infty}=0 \text{ is convergent,} \\
    \lt 1 & \; \text{then } \sum a_{n}|_{n=\infty}=\infty \text{ is divergent.}
  \end{array} \right. 
\]
\end{theo}
\begin{proof} Rearrange into
 Raabe's theorem. Let $z \in \{ \lt, \gt \}$.

 $n \, \mathrm{ln} \frac{a_{n}}{a_{n+1}} \; z \; 1|_{n=\infty}$, 
 $\mathrm{ln} \frac{a_{n}}{a_{n+1}} \; z \; \frac{1}{n}|_{n=\infty}$,   
 $\frac{a_{n}}{a_{n+1}} \gt e^{\frac{1}{n}} |_{n=\infty}$,   
 $a_{n} \; z \; a_{n+1} e^{\frac{1}{n}} |_{n=\infty}$.
 Substitute 
 $e = (\frac{n+1}{n})^{n}|_{n=\infty}$
 into the inequality,  
 $a_{n} \; z \; a_{n+1} ((\frac{n+1}{n})^{n})^{\frac{1}{n}}|_{n=\infty}$,  
 $a_{n} \; z \; a_{n+1} \frac{n+1}{n}|_{n=\infty}$, 
 $n a_{n} - (n+1) a_{n+1} \; z \; 0 |_{n=\infty}$. 
 This is Raabe's test, Theorem \ref{P222}.
\end{proof}
\subsection{A Generalized test} \label{S1703}
The ratio test can be generalized to
 produce other tests
 with the sum of the boundary functions.
 Each test involves higher order terms.

In the preceding discussion we proved Raabe's test (Theorem \ref{P222})
 by transforming the theorem into the ratio test.

Knopp \cite[p.129]{knopp}  
 referred to a generalization of the
 ratio test Theorem \ref{P224}, saying
 ``only the test for $k=0$ 
 and at most $k=1$ have any practical importance."
 Presumably this is because the ratio and Raabe tests are
 most often used.
\bigskip 
\begin{theo}\label{P224}
\cite[p.129]{knopp} with $m$ terms.
\[
[\frac{a_{n+1}}{a_{n}} - 1 + (\frac{1}{n} + \frac{1}{n\,\mathrm{ln}\,n} + \ldots + \frac{1}{n \, \mathrm{ln}\,n \, \ldots \mathrm{ln}_{m}\,n })] n \cdot \mathrm{ln}\,n \ldots \mathrm{ln}_{m}\,n = \left\{ 
  \begin{array}{rl}
    \lt 0 & \; \text{then } \sum a_{v} \text{ is convergent,} \\
    \geq 0 & \; \text{ then } \sum a_{v} \text{ is divergent.}
  \end{array} \right. 
\]
\end{theo}
\begin{proof}
 By Proposition \ref{P226}, rearrange to
 Theorem \ref{P225} which is subsequently proved.
\end{proof}

 Constructing a ratio of $\frac{a_{n}}{a_{n+1}}$ instead of
 $\frac{a_{n+1}}{a_{n}}$
 leads to a different,
 but equivalent formation, Theorem \ref{P225}.
 See Proposition \ref{P226}.
\bigskip
\begin{theo}\label{P225}
\[
\frac{a_{n}}{a_{n+1}} - (1 + \frac{1}{n} + \frac{1}{n\,\mathrm{ln}\,n} + \ldots + \frac{1}{n \, \mathrm{ln}\,n \, \ldots \mathrm{ln}_{m}\,n })|_{n=\infty} = \left\{ 
  \begin{array}{rl}
    \gt 0 & \; \text{then } \sum a_{n}|_{n=\infty}=0 \text{ converges,} \\
    \leq 0 & \; \text{ then } \sum a_{n}|_{n=\infty}=\infty \text{ diverges.}
  \end{array} \right. 
\]
\end{theo}

\begin{defy} An undefined sum has a value of $0$. E.g. $\sum_{k=2}^{1} x=0$
\end{defy} 

Then when $m=-1$,
 $\sum_{k=0}^{m} \frac{1}{\mathrm{ln}_{k}} = 0$
 Restating Theorem \ref{P225} with sum notation,
 we can define the sum to produce the ratio and higher order tests.
\[
\frac{a_{n}}{a_{n+1}} - (1 + \sum_{k=0}^{m}\frac{1}{\mathrm{ln}_{k}})|_{n=\infty} = \left\{ 
  \begin{array}{rl}
    \gt 0 & \; \text{then } \sum a_{n}|_{n=\infty}=0 \text{ converges,} \\
    \leq 0 & \; \text{ then } \sum a_{n}|_{n=\infty}=\infty \text{ diverges.}
  \end{array} \right. 
\]
Successive values of $m$ from $-1$ produce the tests.
 For example, with $m=-1$ and the remove of equality for
 the divergence case, gives the ratio test.
 
\begin{table}[H]
  \centering
  \begin{tabular}{|c|c|c|} \hline
 $m$ & \text{Comparison of terms} & \text{Test} \\ \hline
 $-1$ & $\frac{a_{n}}{a_{n+1}} \; z \; 1$ & \text{Ratio test} \\
 $0$ & $\frac{a_{n}}{a_{n+1}} \; z \; 1 + \frac{1}{n}$ & \text{Raabe's test} \\
 $1$ & $\frac{a_{n}}{a_{n+1}} \; z \; 1 + \frac{1}{n} + \frac{\rho_{n}}{n \, \mathrm{ln}\,n}$ & \text{Bertrand's test} \cite{bertrand} \\
  \hline
  \end{tabular}
  \caption{Tests} \label{FIG06}
\end{table} 
The table entry for Bertrand's test
 excluded the p-series as this is another
 test. $\rho \gt 1$ and $\rho \lt 1$ for the largest
 values
 of the sums become
 $1 + \frac{1}{n} + \frac{1}{n\,\mathrm{ln}\,n} \gt 1$ and
 $1 + \frac{1}{n} + \frac{1}{n\,\mathrm{ln}\,n} \lt 1$
 respectively. These are the only cases that
 need to be considered, as $\rho$ is just
 a real number. The assumption being $\rho \prec n \,\mathrm{ln}\,n|_{n=\infty}$, 
 hence it could be factored to a real number greater than $1$.

The generalized ratio test is proved by transforming
 the test to the boundary test, which we assume is true.
 By doing this,
 the boundary test is shown to be very general,
 and useful in proving other tests.

\begin{proof} \textbf{Theorem \ref{P225}}
 Assume the boundary test is true.
 Using algebra we transform the generalized ratio
 test into the boundary test.
 
\begin{align*}
\frac{a_{n}}{a_{n+1}} \; z \; 1+\sum_{i=0}^{m} \frac{1}{\prod_{k=0}^{i} \mathrm{ln}_{k}}|_{n=\infty} \tag{Generalized ratio}  \\
a_{n} \; z \; a_{n+1}( 1+\sum_{i=0}^{m} \frac{1}{\prod_{k=0}^{i} \mathrm{ln}_{k}}|_{n=\infty})  \\
a_{n}-a_{n+1} \; z \; a_{n+1}(\sum_{i=0}^{m} \frac{1}{\prod_{k=0}^{i} \mathrm{ln}_{k}}|_{n=\infty}) \tag{Interpet the difference as a derivative (Section \ref{S15})} \\
-\frac{ d a_{n+1}}{dn} \; z \; a_{n+1}(\sum_{i=0}^{m} \frac{1}{\prod_{k=0}^{i} \mathrm{ln}_{k}}|_{n=\infty})  \\
-\frac{ d a}{dn} \; z \; a(\sum_{i=0}^{m} \frac{1}{\prod_{k=0}^{i} \mathrm{ln}_{k}}|_{n=\infty}) \tag{Convert to the continuous domain} \\
-\int\frac{1}{a} d a \; z \; \int \sum_{i=0}^{m} \frac{1}{\prod_{k=0}^{i} \mathrm{ln}_{k}} dn |_{n=\infty} \tag{Separation of variables integral} \\
- \mathrm{ln}\,a \; z \; \sum_{i=0}^{m} \int \frac{1}{\prod_{k=0}^{i} \mathrm{ln}_{k}} dn |_{n=\infty}  \\
- \mathrm{ln}\,a \; z \; \sum_{i=0}^{m} \mathrm{ln}_{i+1} |_{n=\infty}  \\
 \mathrm{ln}\,a \; (-z) \; -\mathrm{ln}(\prod_{i=0}^{m} \mathrm{ln}_{i}) |_{n=\infty} \tag{Raising to a base of $e$ does not change the relation} \\
a \; (-z) \; \frac{1}{ \prod_{i=0}^{m} \mathrm{ln}_{i}} |_{n=\infty} \tag{The boundary test (See Section \ref{S18}) } \\
a_{n} \; (-z) \; \frac{1}{ \prod_{i=0}^{m} \mathrm{ln}_{i}} |_{n=\infty} \tag{Convert to a series}  
\end{align*}
 The $-z$ is correct, as the generalized ratio test defined $z$ in
 the opposite direction.
\end{proof}
\bigskip
\begin{cor}\label{P243}
The boundary test and the generalized ratio test are equivalent.
\end{cor}
\begin{proof}
 Since the algebra transformation from the ratio test to the
 boundary test is reversible, by starting from the boundary
 test and, in reverse order
 to the previous proof of Theorem \ref{P225},
 proceed to the generalized ratio
 test,
 hence both tests are equivalent.
\end{proof}
\bigskip
\begin{prop}\label{P226}
Theorem \ref{P224} and Theorem \ref{P225}
 are equivalent.
\end{prop}
\begin{proof}
\begin{align*}
(\frac{a_{n+1}}{a_{n}}-1 + \frac{1}{n} + \frac{1}{n\,\mathrm{ln}\,n} + \ldots + \frac{1}{n\,\mathrm{ln}\,n\ldots \mathrm{ln}_{m}\,n}) n\,\mathrm{ln}\,n\ldots \mathrm{ln}_{m}\,n \; z \; 0 \\
(\frac{a_{n+1}}{a_{n}}-1) \prod_{j=0}^{m}\mathrm{ln}_{j}  + ( \mathrm{ln}_{1}\cdot\mathrm{ln}_{2}\ldots\mathrm{ln}_{m} + \mathrm{ln}_{2}\cdot\mathrm{ln}_{3}\cdot \ldots \cdot \mathrm{ln}_{m} + \ldots + 1) \; z \; 0 \\
\frac{ a_{n+1}-a_{n}}{a_{n}} + \sum_{j=0}^{m} \frac{1}{\prod_{k=0}^{j} \mathrm{ln}_{k}} \; z \; 0  \\
\sum_{j=0}^{m} \frac{1}{\prod_{k=0}^{j} \mathrm{ln}_{k}} \; z \; -\frac{1}{a_{n}}\frac{d a_{n}}{dn} \\
-\frac{ d a}{dn} \; z \; a(\sum_{i=0}^{m} \frac{1}{\prod_{k=0}^{i} \mathrm{ln}_{k}}|_{n=\infty}) \tag{reversing to form the other ratio test} \\
 -\frac{ d a_{n+1}}{dn} \; z \; a_{n+1}(\sum_{i=0}^{m} \frac{1}{\prod_{k=0}^{i} \mathrm{ln}_{k}}|_{n=\infty}) \\
 a_{n}-a_{n+1} \; z \; a_{n+1}(\sum_{i=0}^{m} \frac{1}{\prod_{k=0}^{i} \mathrm{ln}_{k}}|_{n=\infty}) \\
 a_{n} \; z \; a_{n+1}( 1+\sum_{i=0}^{m} \frac{1}{\prod_{k=0}^{i} \mathrm{ln}_{k}}|_{n=\infty}) \\
 \frac{a_{n}}{a_{n+1}} \; z \; 1+\sum_{i=0}^{m} \frac{1}{\prod_{k=0}^{i} \mathrm{ln}_{k}}|_{n=\infty}
\end{align*}
 Reversing the above implies the other test. Hence
 both tests are equivalent.
\end{proof}
\section{The Boundary test for positive series} \label{S18}
 With convergence sums, a universal comparison test for positive series
 is developed, which compares a positive monotonic series
 with an infinity of generalized p-series.  The boundary 
 between convergence and divergence is an infinity of
 generalized p-series.
 This is a rediscovery and reformation of a 175 year old convergence/divergence test.
\subsection{Introduction}\label{S1801}
  We see the boundary test as the most general and powerful of
 all convergence tests.
 A tool for other tests and theory, large or small.

The boundary test we believe to be the most general
 positive series test. Both du Bois-Reymond's theory,
 and described by Hardy, 
 the comparison of functions were forgotten 
 largely because the theory had
 no perceived applications.
 The rediscovery of this test should place
 it as something which has been missing from convergence and divergence
 theory. 
\bigskip
\begin{defy}
$L_{w} = \prod_{k=0}^{w} \mathrm{ln}_{j}$, $\;L_{w}=1$ when $w \lt 0$
\end{defy}
\begin{quote}
Du Bois-Reymond seems to have been led to 
 consider this ordering of functions
 by way of an attempt to construct an ideal series or integral which would serve as
 a boundary between convergent and divergent series or integrals, based on
 BERTRAND'S series. These are sometimes called ABEL'S series
 $...$ series of the form $\sum \frac{1}{L_{w-1} \mathrm{ln}_{w}^{p}}$ \cite[p.103]{fisher}.
\end{quote}

 Generally the convergence sums (Section \ref{S01}) and comparison is in $*G$.
 When we say a sum at infinity is $0$, the transfer $\Phi \mapsto 0$
 has been applied.
 When we say a sum at infinity is infinity, the transfer $\Phi^{-1} \mapsto \infty$
 was applied if the sum was defined.
\subsection{Generalized p-series}\label{S1802}
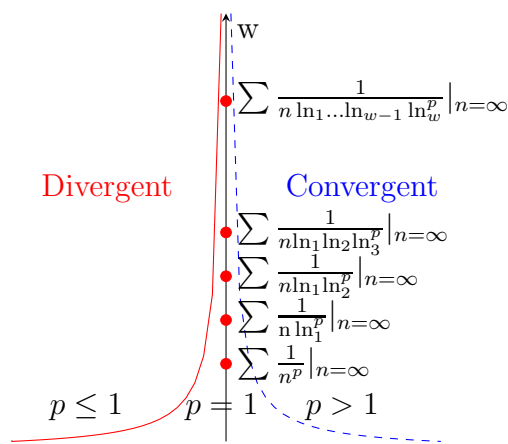
\begin{figure}[H]
\centering
\begin{tikzpicture}
\begin{axis}
  [
    ticks=none,
    axis y line=center,
    ylabel={w},
    axis x line=none,
    xlabel={p}
  ]
  \addplot[domain=0.1:4.4,blue,dashed]{ 1/x};
  \addplot[domain=-4.4:-0.1,red]{ -1/x};
  \addplot[red, mark = *, nodes near coords=$\sum \frac{1}{n \, \mathrm{ln}_{1} \ldots \mathrm{ln}_{w-1} \, \mathrm{ln}_{w}^{p}}|_{n=\infty}$, every node near coord/.style={color=black, anchor=180}] coordinates {(0,8)};
  \addplot[red, mark = *, nodes near coords=$\sum \frac{1}{n \mathrm{ln}_{1} \mathrm{ln}_{2}\mathrm{ln}_{3}^{p}}|_{n=\infty}$, every node near coord/.style={color=black,anchor=180}] coordinates {(0,5)};
  \addplot[red, mark = *, nodes near coords=$\sum \frac{1}{n \mathrm{ln}_{1} \mathrm{ln}_{2}^{p}}|_{n=\infty}$, every node near coord/.style={color=black, anchor=180}] coordinates {(0,4)};
  \addplot[red, mark = *, nodes near coords=$\sum \frac{1}{\mathrm{n \, ln}_{1}^{p}}|_{n=\infty}$, every node near coord/.style={color=black, anchor=180}] coordinates {(0,3)};
  \addplot[red, mark = *, nodes near coords=$\sum \frac{1}{n^{p}}|_{n=\infty}$, every node near coord/.style={color=black, anchor=180}] coordinates {(0,2)};
  \addplot[blue, nodes near coords=Convergent, every node near coord/.style={anchor=180}] coordinates {(1,6)};
  \addplot[red, nodes near coords=Divergent, every node near coord/.style={anchor=180}] coordinates {(-4,6)};
  \node [label={[shift={(2.8cm,0.0cm)}]$p=1$}] {};
  \node [label={[shift={(1.0cm,0.0cm)}]$p \leq 1$}] {};
  \node [label={[shift={(4.4cm,0.0cm)}]$p > 1$}] {};
  \draw (1.5,1) [dashed] to (1.5,2);
\end{axis}
\end{tikzpicture}
\caption{Generalized p-series between convergence/divergence} \label{fig:M1}
\end{figure}

We discussed the possibility of two straight lines
 approaching each other, with two possibilities at
 infinity. The lines never meet, or the lines meet
 at infinity. We find a more complex
 case where two classes of sums about $p=1$, having
 their functions close, never meet,
 one sums to infinity, the other zero.

The following discussion is on
 the generalization of the p-series
 through an observation and investigation. While the results are known,
 they are expressed 
 with infinitary calculus at
 infinity. 

It so happens that the p-series
 has for all values $p \gt 1$ 
 the series converges, and all values
 $p \leq 1$ the series diverges.
\bigskip
\begin{theo}\label{P227}
 If $p \leq 1 \Rightarrow \sum \frac{1}{n^{p}}|_{n=\infty}=\infty$
  diverges. 
 If $p \gt 1 \Rightarrow \sum \frac{1}{n^{p}}|_{n=\infty}=0$ converges. 
\end{theo}
\bigskip 
\begin{defy}\label{DEF012}
Let $k$-nested natural logarithms
 be represented by $\mathrm{ln}_{k}(x)$:  
$\mathrm{ln}_{0}\,x = x$,  
$\mathrm{ln}_{k}\,x = \mathrm{ln}( \mathrm{ln}_{k-1}\,x)$.
 For convenience, let $\mathrm{ln}_{k}$ without an argument mean $\mathrm{ln}_{k}(n)$.
\end{defy}
\bigskip
\begin{theo}\label{P228}
Let $f(w,n) = \int \frac{1}{\prod_{k=0}^{w-1} \mathrm{ln}_{k}(n)} \cdot \frac{1}{(\mathrm{ln}_{w}(n))^{p} } dn$.
 When $w \geq 1$ then $f(w,n) = f(w-1,\mathrm{ln}\,n)$ and 
 $(f(w,n) = f(0,n) = \int \frac{1}{t^{p}} dt)|_{n=\infty}$
 where $t =\mathrm{ln}_{w}\,n|_{n=\infty}$.
\end{theo}
\begin{proof}
 Let $n = e^{v}$, $\frac{dn}{dv} = e^{v}$. 
$f(w,n) = \int \frac{1}{n} \frac{1}{\prod_{k=1}^{w-1} \mathrm{ln}_{k}(n)} \cdot \frac{1}{(\mathrm{ln}_{w}(n))^{p} } \frac{dn}{dv} dv$ 
$= \int \frac{1}{n} \frac{1}{\prod_{k=1}^{w-1} \mathrm{ln}_{k}(e^{v})} \cdot \frac{1}{(\mathrm{ln}_{w}(e^{v}))^{p} } e^{v} dv$ 
$= \int \frac{1}{n} \frac{1}{\prod_{k=1}^{w-1} \mathrm{ln}_{k-1}(v)} \cdot \frac{1}{(\mathrm{ln}_{w-1}(v))^{p} } n dv$ 
$= \int \frac{1}{\prod_{k=1}^{w-1} \mathrm{ln}_{k-1}(v)} \cdot \frac{1}{(\mathrm{ln}_{w-1}(v))^{p} } dv$ 
$= \int \frac{1}{\prod_{k=0}^{w-2} \mathrm{ln}_{k}(v)} \cdot \frac{1}{(\mathrm{ln}_{w-1}(v))^{p} } dv$ 
$=f(w-1,v)$

 Apply $f(w,n) = f(w-1,\mathrm{ln}\,n)$, $w$ times to remove
 $\frac{1}{\prod_{k=0}^{w-1} \mathrm{ln}_{k}}$ inside the integral.
 $(f(w,n) = f(0,n) = \int \frac{1}{t^{p}} dt)|_{n=\infty}$,
 $t =\mathrm{ln}_{w}\,n|_{n=\infty}$.
\end{proof}
\bigskip
\begin{theo}\label{P229}
\[ \int \frac{1}{L_{w}} dn = \mathrm{ln}_{w+1}|_{n=\infty} \]
\end{theo}
\begin{proof}
Substitute $p=0$ and $w=w+1$ into Theorem \ref{P228}
 then $t = \mathrm{ln}_{w+1}\,n|_{n=\infty}$ and
 $f(0,n)|_{n=\infty} = \int \frac{1}{t^{p}} dt|_{n=\infty}$
 $=t|_{n=\infty}$
 $= \mathrm{ln}_{w+1}|_{n=\infty}$ 
\end{proof}

Since the divergence
 of the sum is the
 same as the divergence of the integral,
 by Theorem \ref{P229} the following generalization
 of the harmonic series always 
 diverges.
\[ ( \sum \frac{1}{n}, \sum \frac{1}{n\, \mathrm{ln}\,n}, \sum \frac{1}{n \, \mathrm{ln}\,n \, \mathrm{ln}_{2}}, \sum \frac{1}{n \, \mathrm{ln} \, \mathrm{ln}_{2} \, \mathrm{ln}_{3}}, \ldots )|_{n=\infty} \text{ sums diverge} \]

 We could have reasoned the above by considering
 scales of infinities \cite[Part 2]{cebp21}.
 $L_{0} \prec L_{1} \prec L_{2} \prec L_{3} \prec \ldots|_{n=\infty}$,  
 $\frac{1}{L_{0}} \succ \frac{1}{L_{1}} \succ \frac{1}{L_{2}} \succ \frac{1}{L_{3}} \succ \ldots|_{n=\infty}$,   
 assuming 
 $f \succ g$ then $\sum f \succ \sum g$,
 $\sum \frac{1}{L_{0}} \succ \sum \frac{1}{L_{1}} \succ \sum \frac{1}{L_{2}} \succ \sum \frac{1}{L_{3}} \succ \ldots|_{n=\infty}$ 
 which shows the sums diverging more slowly.
\bigskip
\begin{defy}
Let 
 $\sum \frac{1}{\prod_{k=0}^{w-1} \mathrm{ln}_{k} \cdot \mathrm{ln}_{w}^{p} }$
 or $\int \frac{1}{\prod_{k=0}^{w-1} \mathrm{ln}_{k} \cdot \mathrm{ln}_{w}^{p} }\,dn|_{n=\infty}$
 be called the generalized p-series. 
\end{defy}
\bigskip
\begin{theo}\label{P230}
\[
\sum \frac{1}{\prod_{k=0}^{w-1} \mathrm{ln}_{k} \cdot \mathrm{ln}_{w}^{p} }|_{n=\infty}
 = \left\{
  \begin{array}{rl}
    0  & \; \text{converges when } p \gt 1 \\
    \infty  & \; \text{diverges when } p \leq 1 
  \end{array} \right.
\]
\end{theo}
\begin{proof}
By the integral theorem the convergence and divergence of
 the series is the same as the respective integral.
 $\sum \frac{1}{\prod_{k=0}^{w-1} \mathrm{ln}_{k} \cdot \mathrm{ln}_{w}^{p} }$
 $=\int \frac{1}{\prod_{k=0}^{w-1} \mathrm{ln}_{k}(n) } \cdot \frac{1}{(\mathrm{ln}_{w}(n))^{p} } \,dn|_{n=\infty}$
 $= f(w,n)$

Apply Theorem \ref{P228},
$f(w,n) = f(w-1,\mathrm{ln}\,n)$, $w$ times to remove $\frac{1}{\prod_{k=0}^{w-1} \mathrm{ln}_{k}(n)}$ inside the integral.

 $(f(w,n) = f(0,n) = \int \frac{1}{t^{p}} dt)|_{n=\infty}$, $t =\mathrm{ln}_{w}\,n$, 
 which is known to converge when $p \gt 1$ and diverge when $p \leq 1$.
 Hence the generalized p-series proved.
\end{proof}
\subsection{The existence of the boundary and tests}\label{S1803}
Arguments for and against the boundary
 have existed from its inception,
 when it was realized that
 no one series could partition
 all positive series into
 either converging or diverging series.

\begin{quote}
KNOPP says, for example, that it is clear
 "that it is quite useless to attempt to introduce anything of the nature of a boundary
 between convergent and divergent series, as was suggested by P. DU BOIS-REYMOND
 ... in whatever manner we may choose to render it precise, it will never correspond
 to the actual circumstances" [KNOPP 1954, 313 or 1951, 304; his emphasis].
 \cite[p.135]{fisher}
\end{quote}

 While teaching (see \cite{hart}) comments.
\begin{quote}
Recently one of our students
 remarked that the harmonic series 
 $\sum \frac{1}{n}$ acts as a boundary
 between convergence and divergence.
\end{quote}

This was not that different from
 the historical account where new tests emerged,
 and was contradicted by Abel,  
 $\sum \frac{1}{n\,\mathrm{ln}\,n}$ diverges.

The series was generalized
 to the generalized p-series (Section \ref{S1802}),
 and considered as a boundary between convergence
 and divergence.

 However, ironically it was du Bois-Reymond who disproved
 the boundary \cite[p.103]{fisher}.

\begin{quote}
 One can ask whether or not the convergence or divergence
 of all series with positive terms can be settled by comparison with a real multiple
 of one of these series. The answer is no, and this was first shown
 by DU BOIS-REYMOND [DU BOIS-REYMOND 1873, 88-91].
\end{quote}

 See Theorem \ref{P241} and Theorem \ref{P239}. Hardy summarizes
 (see \cite[pp.67--68]{ordersofinfinity}).

\begin{quote}
 Given any divergent
 series we can always find one more slowly divergent.
 $\ldots$ given any convergent series,
 we can find one more slowly convergent.
\end{quote}  

A. Pringsheim was of the same opinion.  
\begin{quote}
The analogy with the irrational numbers is a logical blunder; one can insert between the elements of the two
classes defined by $x^{2} \lt 2$ and $x^{2} \gt 2$
 a new thing corresponding to the relation
 $x^{2} = 2$,
 but between the convergence and divergence of positive
 series, there is no such "third". \cite[p.151]{fisher} 
\end{quote}

 However, the concept of an ideal function did not go away.
\begin{quote}
BOREL says "we know that there is no function of $n$,
 in the ordinary sense of the word, which has this property; that is, we call
 $O(n)$ an ideal function; it is not a true function" [BOREL 1946, 148]. But there it is,
nevertheless. PRINGSHEIM'S attempts to argue and ridicule it out of existence failed
to sway BOREL. \cite[p.147]{fisher}
\end{quote}

\bigskip
\begin{conjecture}
 While one series cannot be the boundary,
 what is to say that an infinite collection
 of series cannot form the boundary.
\end{conjecture}

While one series alone cannot 
 separate all convergent and divergent series,
 Theorem \ref{P239} and Theorem \ref{P241} do not say anything
 about an infinity of such series.
 Indeed the generalized p-series
 keeps decreasing its terms
 in size. In other words an appeal
 that there is no smallest 
 series fails on a collection
 of infinitely smaller term series, is invalid. 

 Surprisingly, 
 after deriving a general boundary
 test which is the subject of this paper,
 it was found in fact to be a rediscovery.

 The test was referenced
 in the appendix (not an important place for a major test) in Hardy's Orders of Infinity,
 with no example and limited description.
 Hardy does cite other
 references, referring to the `logarithmic criteria' by De Morgan,
 attributing the criteria in 1839 \cite[p.67]{ordersofinfinity}.
 Further, the test does not appear as a general test in
 the current known convergent tests.
   
 However, we believe this
 is the long sort after universal test
 for a positive series either converging or diverging.

The test given by Hardy was incorrect for
 the divergent series case, possibly a typing error.
\bigskip 
\begin{theo}\label{P231}
\textbf{The logarithmic test} \cite[Appendix II p.66]{ordersofinfinity} with correction. \\
 \\
The series $\sum a_{n} \;\; (a_{n} \geq 0)$
\[ \text{is convergent if } a_{n} \preceq \frac{1}{ \prod_{k=0}^{w-1} \mathrm{ln}_{k} \cdot (\mathrm{ln}_{w})^{1+\alpha}} \\
 \text{where }\alpha \gt 0, \]
\[ \text{and divergent if } \\
a_{n} \succeq \frac{1}{ \prod_{k=0}^{w}  \mathrm{ln}_{k}} \]
The integral $\int^{\infty} f(x) \,dx \;\; (f \geq 0)$ \\ 
\[ \text{is convergent if } f(x) \preceq \frac{1}{ \prod_{k=0}^{w-1} \mathrm{ln}_{k}\,x \cdot (\mathrm{ln}_{w}\,x)^{1+\alpha}} \\
 \text{where }\alpha \gt 0,\]
\[ \text{and divergent if } \\
f(x) \succeq \frac{1}{ \prod_{k=0}^{w}  \mathrm{ln}_{k}\,x} \]
\end{theo}

 We developed a detailed convergence criteria E3
 (Section \ref{S01}) where
 comparing a sum we are able to remove the sum/integral and
 compare the corresponding monotonic functions. In the reformed test, a 
 direct comparison with the boundary is made.
\begin{align*}
\sum a_{n} \; z \; \sum \frac{1}{\prod_{k=0}^{w} \mathrm{ln}_{k}}|_{n=\infty} \\
 a_{n} \; z \; \frac{1}{\prod_{k=0}^{w} \mathrm{ln}_{k}}|_{n=\infty} 
\end{align*} 
By expressing the boundary and comparison differently,
 in the original test the two functions 
 (1): $\frac{1}{ \prod_{k=0}^{w-1} \mathrm{ln}_{k}\,x \cdot (\mathrm{ln}_{w}\,x)^{1+\alpha}}$
 convergent and asymptotic to the boundary, and
 (2): $\frac{1}{ \prod_{k=0}^{w}  \mathrm{ln}_{k}}$
 the p-series on the boundary in Theorem \ref{P231} are rephrased
 to compare against the p-series function (2) only, see Theorem \ref{P232}.
 Also, the relation
 being solved for is a simpler relation. Only one relation needs
 to be solved for, not two.

 We also needed an algebra for comparing functions to make the
 boundary test usable. 

Some further differences in the theorems, $\succeq$ in the
 logarithmic test is replaced by $\geq$ in the boundary test.
 This can be explained where the original test is
 considered from a finite perspective, and $\succeq$
 removes the transient sum terms before reaching
 infinity.  By considering the series
 at infinity and requiring a monotonic sequence of
 terms this was avoided.
\bigskip
\begin{remk}
 By considering different rearrangements and intervals,
 we can deform $a_{n}$ into a monotonic sequence (Section \ref{S16}) 
 comparable with the generalized
 p-series. If this is not possible then by definition the series
 or integral diverges.
 Similarly with functions, $a(n)$ can be deformed.
\end{remk}
\bigskip 
\begin{theo}\label{P232}
\textbf{The boundary test} 
\[ 
\sum a_{n} \; z \; \sum \frac{1}{\prod_{k=0}^{w} \mathrm{ln}_{k}\,n}|_{n=\infty},\;\; z = \left\{
  \begin{array}{rl}
    \lt & \; \text{then } \sum a_{n}|_{n=\infty}=0 \text{ is convergent,} \\ 
    \geq & \; \text{then } \sum a_{n}|_{n=\infty}=\infty \text{ is divergent.} 
  \end{array} \right.
\]
\[ 
\int a(n)\,dn \; z \; \int \frac{1}{\prod_{k=0}^{w} \mathrm{ln}_{k}\,n}\,dn|_{n=\infty},\;\; z = \left\{
  \begin{array}{rl}
    \lt & \; \text{then } \int a(n)\,dn|_{n=\infty}=0 \text{ is convergent,} \\ 
    \geq & \; \text{then } \int a(n)\,dn|_{n=\infty}=\infty \text{ is divergent.} 
  \end{array} \right.
\]

Since $w$ is a fixed integer, there is an infinity of tests.
 By removing the sum and solving the comparison
 of functions a unique $w$ is found. 
\end{theo}
\begin{proof}
$\sum a_{n} \; z \; \sum \frac{1}{\prod_{k=0}^{w} \mathrm{ln}_{k}}|_{n=\infty}$,
$\int a(n)\,dn \; z \; \int \frac{1}{\prod_{k=0}^{w} \mathrm{ln}_{k}}dn|_{n=\infty}$,
 differentiate,
 $a(n) \; z \; \frac{1}{\prod_{k=0}^{w} \mathrm{ln}_{k}}|_{n=\infty}$,
 by Theorem \ref{P233}
 solve for $z = \{ \lt, \geq \}$. 
\end{proof}
\bigskip
\begin{theo}\label{P233}
\textbf{The boundary test comparison} 
\[ 
 a_{n} \; z \; \frac{1}{\prod_{k=0}^{w} \mathrm{ln}_{k}\,n}|_{n=\infty},\;\; z = \left\{
  \begin{array}{rl}
    \lt & \; \text{then } \sum a_{n}|_{n=\infty}=0 \text{ is convergent,} \\ 
    \geq & \; \text{then } \sum a_{n}|_{n=\infty}=\infty \text{ is divergent.} 
  \end{array} \right.
\]
\[ 
 a(n) \; z \; \frac{1}{\prod_{k=0}^{w} \mathrm{ln}_{k}\,n}|_{n=\infty},\;\; z = \left\{
  \begin{array}{rl}
    \lt & \; \text{then } \int a(n)\,dn|_{n=\infty}=0 \text{ is convergent,} \\ 
    \geq & \; \text{then } \int a(n)\,dn|_{n=\infty}=\infty \text{ is divergent.} 
  \end{array} \right.
\]
 Solving the
 comparison of functions a unique $w$ is found.
\end{theo}
\begin{proof}
 In Collary \ref{P243} 
 we show the equivalence between the boundary test
 and the generalized ratio test.
 If we consider the generalized ratio test is proved,
 since we establish the equivalence
 the boundary test would consequently be proved.

 We believe the generalized ratio test is proved,
 though Knopp \cite[p.129]{knopp} did not cite the proof.
 Instead he referred to the generalized p-series as criteria.
 Since they are proved, this is not contradictory, as a criteria are
 the assumptions, that which is held to be true. 
\end{proof}
\begin{proof}
 The boundary is equivalent to the generalized p-series 
 test Theorem \ref{P227}.
 By solving for the relation, the appropriate p-series
 may be found and tested against.

 In general, non-reversible arguments of magnitude are
 used to reduce lower order terms. This is required
 to meet the convergence criteria E3 (Section \ref{S01}).
\end{proof}
\bigskip
\begin{remk}
 If a circular argument occurs, then additional information,
 perhaps an identity may be required to be found to 
 solve for the relation. Similarly further conditions may
 be input if the problem is ill posed or incomplete.
\end{remk}
\bigskip
\begin{remk} If a positive sum is less than the boundary the sum converges; else the sum diverges. 
\end{remk}

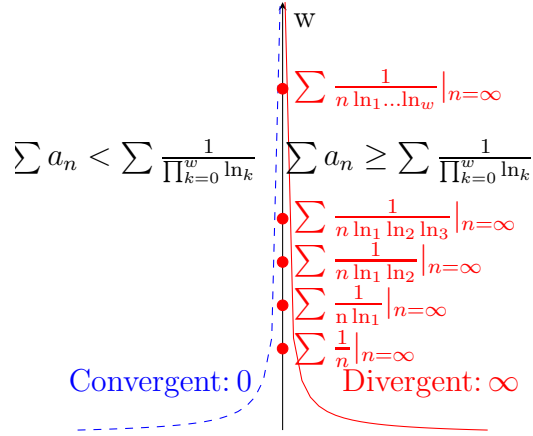
\begin{figure}[h]
\centering
\begin{tikzpicture}
\begin{axis}
  [
    ticks=none,
    axis y line=center,
    ylabel={w},
    axis x line=none,
    xlabel={p}
  ]
  \addplot[domain=0.1:8.4,red]{ 1/x};
  \addplot[domain=-8.4:-0.1,blue,dashed]{ -1/x};
  \addplot[red, mark = *, nodes near coords=$\sum \frac{1}{n \, \mathrm{ln}_{1} \ldots \mathrm{ln}_{w}}|_{n=\infty}$, every node near coord/.style={anchor=180}] coordinates {(0,8)};
  \addplot[red, mark = *, nodes near coords=$\sum \frac{1}{n \, \mathrm{ln}_{1} \, \mathrm{ln}_{2} \, \mathrm{ln}_{3}}|_{n=\infty}$, every node near coord/.style={anchor=180}] coordinates {(0,5)};
  \addplot[red, mark = *, nodes near coords=$\sum \frac{1}{n \, \mathrm{ln}_{1} \, \mathrm{ln}_{2}}|_{n=\infty}$, every node near coord/.style={anchor=180}] coordinates {(0,4)};
  \addplot[red, mark = *, nodes near coords=$\sum \frac{1}{\mathrm{n \, ln}_{1}}|_{n=\infty}$, every node near coord/.style={anchor=180}] coordinates {(0,3)};
  \addplot[red, mark = *, nodes near coords=$\sum \frac{1}{n}|_{n=\infty}$, every node near coord/.style={anchor=180}] coordinates {(0,2)};
  \addplot[red, nodes near coords=$\text{Divergent:}\, \infty$, every node near coord/.style={anchor=180}] coordinates {(2.0,1.2)};
  \addplot[blue, nodes near coords=$\text{Convergent:}\,0$, every node near coord/.style={anchor=180}] coordinates {(-9.2,1.2)};
  \node [label={[shift={(1.0cm,3.00cm)}]$\sum a_{n} < \sum \frac{1}{\prod_{k=0}^{w} \mathrm{ln}_{k}}$}] {};
  \node [label={[shift={(5.1cm,3.00cm)}]$\sum a_{n} \geq \sum \frac{1}{\prod_{k=0}^{w} \mathrm{ln}_{k}}|_{n=\infty}$}] {};
\end{axis}
\end{tikzpicture}
\caption{Series $\sum a_{n}|_{n=\infty}$ compared with the boundary} \label{fig:M2}
\end{figure}

 The boundary test has a convergence criterion specifically for
 evaluating sums as either $0$ or $\infty$. The logarithmic test
 compares the inner components. 
\bigskip
\begin{prop}\label{P017}
 The Logarithmic test Theorem \ref{P231}
 implies
 the Boundary test Theorem \ref{P232} 
\end{prop}
\begin{proof}
 By removing the transients, the conditions are simplified.

Divergent case. 
 $a_{n} \succeq \frac{1}{L_{w}}$,
 $M_{1} \in \mathbb{R}^{+}$,
 $\exists M_{1}$: $M_{1} a_{n} \geq \frac{1}{L_{w}}$,
 $M_{1} a_{n} L_{w} \geq 1|_{n=\infty}$,
 $\mathrm{ln}\,M_{1} + \mathrm{ln}\, a_{n} + \mathrm{ln}\,L_{w} \geq 0|_{n=\infty}$,
 $\mathrm{ln}\, a_{n} + \mathrm{ln}\,L_{w} \geq 0|_{n=\infty}$ as
 $\mathrm{ln}\,M_{1} \prec \mathrm{ln}\,a_{n}|_{n=\infty}$,
 $a_{n} L_{w} \geq 1|_{n=\infty}$,
 $a_{n} \geq \frac{1}{L_{w}}|_{n=\infty}$.

 Convergent case. Let $\alpha \gt 0$,
 $a_{n} \preceq \frac{1}{L_{w-1} \mathrm{ln}_{w}^{1+\alpha}}$,
 $M_{2} \in \mathbb{R}^{+}$,
 $\exists M_{2}$: 
 $a_{n} \leq M_{2} \frac{1}{L_{w-1} \mathrm{ln}_{w}^{1+\alpha}}$,
 $a_{n} M_{2}^{-1} L_{w-1} \mathrm{ln}_{w}^{1+\alpha} \leq 1$,
 $\mathrm{ln}\,a_{n} - \mathrm{ln}\,M_{2} + \mathrm{ln}\,L_{w-1} + \mathrm{ln}\,\mathrm{ln}_{w}^{1+\alpha} \leq 0$,
 $\mathrm{ln}\,a_{n} + \mathrm{ln}\,L_{w-1} + \mathrm{ln}\,\mathrm{ln}_{w}^{1+\alpha} \leq 0$ as $\mathrm{ln}\,a_{n} \succ \mathrm{ln}\,M_{2}$, 
 $a_{n} \leq \frac{1}{L_{w-1} \mathrm{ln}_{w}^{1+\alpha}}$,
 since $\frac{1}{L_{w-1} \mathrm{ln}_{w}^{1+\alpha}} \lt \frac{1}{L_{w-1} \mathrm{ln}_{w}^{1}}$ then
 $a_{n} \lt \frac{1}{L_{w}}$.

 Both cases are disjoint and cover the line. The conditions 
 correspond exactly with the boundary test.
\end{proof}

The way the logarithmic test was structured
  shows that
 there was no explicit separation between
 finite and infinite numbers.
 With the removal of the
 transients 
 there is no need to solve for $M_{1}$ and $M_{2}$.
 
\begin{quote}
 It is a reasonable question to ask how 
 the most important and general
 positive series convergence and divergence test
 was found and lost.
\end{quote}

 Hardy was of the belief that du Bois-Reymond's
 theory was highly original, but served no purpose.
 Du Bois-Reymond was looking for a theory of
 the continuum. It may well be a case of once again,
 a theoretical piece of mathematics that appears utterly
 useless, even after having been considered,
 becomes essential to our understanding.

 We have arrived at this view by several steps,
 a separation between finite and infinite numbers
 \cite[Part 6]{cebp21},
 a number system with infinities \cite[Part 1]{cebp21},
 an algebra for comparing functions \cite[Part 3]{cebp21}
 and convergence sums (Section \ref{S01}).

On the importance of the logarithmico-exponential
 scales, to argue for the generality of the test, Hardy comments:
\begin{quote}
 No function has yet presented itself in analysis the
 laws of whose increase, in so far as they can be stated at all, cannot be
 stated, so to say, in logarithmico-exponential terms. \cite[p.48]{ordersofinfinity}
\end{quote}

 Since the boundary test is a comparison against the whole
 logarithmico-exponential scale, then if what Hardy said is
 true; that all such testable functions can be expressed
 in logarithmico-exponential scale,
 then as the boundary test compares against this scale, all
 functions as said are able to be compared with the boundary
 by the boundary test.
 The boundary test, with the monotonic
 constraint is theoretically complete. 

As discussed in (Section \ref{S16}), the theory 
 of convergence sums is extended
 to include testing for non-monotonic sequences.
 By considering arrangements,
 convert non-monotonic series to monotonic series for
 convergence testing, thereby increasing the classes of series 
 which can be tested.
 It follows that this makes the boundary test more general.

 As a consequence of Theorem \ref{P232}:
 If a sum is less than the boundary at infinity then the sum converges.
 Symbolically,
 if $\sum a_{n} \lt \sum \frac{1}{\prod_{k=0}^{w} \mathrm{ln}_{k}}|_{n=\infty}$
 then $\sum a_{n}|_{n=\infty}=0$ converges.
 Therefore when solving
 $\sum a_{n} \; z \; \frac{1}{\prod_{k=0}^{w} \mathrm{ln}_{k}}$,
 if we solve for $z$ and find $z = \; \lt$, it immediately follows that $\sum a_{n}|_{n=\infty}=0$
 converges.
\subsection{The boundary test Examples}\label{S1804}
\begin{mex}\label{MEX223}
$\sum \frac{1}{n^{2}}|_{n=\infty}$ is the p-series with $p=2$
 and is known to converge.
 Testing this series against the boundary. 
\begin{align*}
\sum \frac{1}{n^{2}} \;\; z \;\; \sum \frac{1}{\prod_{k=0}^{w}\mathrm{ln}_{k}} |_{n=\infty} \\ 
\frac{1}{n^{2}} \;\; z \;\; \frac{1}{\prod_{k=0}^{w}\mathrm{ln}_{k}} |_{n=\infty} \\ 
\prod_{k=0}^{w}\mathrm{ln}_{k} \;\; z \;\; n^{2}|_{n=\infty} \\ 
\sum_{k=0}^{w}\mathrm{ln}_{k+1} \;\; (\mathrm{ln}\,z) \;\; 2 \, \mathrm{ln}\,n|_{n=\infty} \\ 
 0 \;\; (\mathrm{ln}\,z) \;\; 2 \, \mathrm{ln}\,n|_{n=\infty} \tag{as $\mathrm{ln}\,n \succ \mathrm{ln}_{k+1}$} \\ 
 0 \lt 2 \, \mathrm{ln}\,n|_{n=\infty} \tag{$\mathrm{ln}\,z = \; \lt$,
 $z = e^{\lt} = \; \lt$} \\ 
\sum \frac{1}{n^{2}} \lt \sum \frac{1}{\prod_{k=0}^{w}\mathrm{ln}_{k}} |_{n=\infty} \tag{$\sum a_{n}|_{n=\infty}=0$ converges} 
\end{align*}
\end{mex}
\bigskip
The boundary test can handle products such as $n!$ by converting them to sums via the $\mathrm{log}$ operation,
 and integrating the function in the continuous domain.
\bigskip
\begin{mex}\label{MEX224}
Determine convergence or divergence of $\sum \frac{1}{n!}|_{n=\infty}$.

$\sum \frac{1}{n!} \; z \; \sum \frac{1}{\prod_{k=0}^{w} \mathrm{ln}_{k}}|_{n=\infty}$,
 $\prod_{k=0}^{w} \mathrm{ln}_{k} \; z \; n!|_{n=\infty}$,
 $\sum_{k=1}^{w+1} \mathrm{ln}_{k} \; (\mathrm{ln}\,z) \; \sum_{k=1}^{n} \mathrm{ln}\,k|_{n=\infty}$,
 $\sum_{k=1}^{w+1} \mathrm{ln}_{k} \; (\mathrm{ln}\,z) \; \int^{n} \mathrm{ln}\,n \,dn|_{n=\infty}$,
 $\sum_{k=1}^{w+1} \mathrm{ln}_{k} \; (\mathrm{ln}\,z) \; n \,\mathrm{ln}\,n |_{n=\infty}$,
 $\mathrm{ln}\,n \; (\mathrm{ln}\,z) \; n \,\mathrm{ln}\,n |_{n=\infty}$,
 $\mathrm{ln}\,z = \; \lt$, $z = e^{\lt} = \; \lt$,
 $\sum \frac{1}{n!}|_{n=\infty} = 0$ converges.
\end{mex}
\bigskip
\begin{mex}\label{MEX222}
Determine the convergence/divergence of
 $\sum \frac{n!}{(2n)!} |_{n=\infty}$ 

This problem would normally be done 
 more simply with the ratio or comparison test.

\begin{align*}
\sum \frac{n!}{(2n)!} \; z \; \sum \frac{1}{\prod_{k=0}^{w} \mathrm{ln}_{k}} |_{n=\infty} \tag{Compare against the boundary} \\ 
\frac{n!}{(2n)!} \; z \; \frac{1}{\prod_{k=0}^{w} \mathrm{ln}_{k}} |_{n=\infty} \\ 
n! \prod_{k=0}^{w} \mathrm{ln}_{k} \; z \; (2n)!|_{n=\infty} \\ 
\mathrm{ln}(n! \prod_{k=0}^{w} \mathrm{ln}_{k}) \; (\mathrm{ln}\, z) \; \mathrm{ln}\,(2n)!|_{n=\infty} \\ 
\sum_{k=1}^{n} \mathrm{ln}\,k + \sum_{k=1}^{w+1} \mathrm{ln}_{k} \; (\mathrm{ln}\,z) \; \sum_{k=1}^{2n} \mathrm{ln}\,k |_{n=\infty} \tag{Apply log law, convert products to sums} \\
\int_{1}^{n} \mathrm{ln}\,x dx + \sum_{k=1}^{w+1} \mathrm{ln}_{k} \; (\mathrm{ln}\,z) \; \int_{1}^{2n} \mathrm{ln}\,x dx |_{n=\infty} \tag{discrete to a continuous domain} \\
\int^{n} \mathrm{ln}\,x dx + \sum_{k=1}^{w+1} \mathrm{ln}_{k} \; (\mathrm{ln}\,z) \; \int^{2n} \mathrm{ln}\,x dx |_{n=\infty} \\
n \,\mathrm{ln}\,n + \sum_{k=1}^{w+1} \mathrm{ln}_{k} \; (\mathrm{ln}\,z) \; (2n) \mathrm{ln}(2n)|_{n=\infty} \\
n \,\mathrm{ln}\,n + \mathrm{ln}\,n \; (\mathrm{ln}\,z) \; (2n) \mathrm{ln}(2n)|_{n=\infty} \tag{largest boundary} \\
n \,\mathrm{ln}\,n \; (\mathrm{ln}\,z) \; (2n) \mathrm{ln}(2n)|_{n=\infty} 
 \tag{$\mathrm{ln}\,n \prec n\,\mathrm{ln}\,n$, $2n \prec 2n \,\mathrm{ln}(2n)|_{n=\infty}$} \\
\mathrm{ln}\,z = \; \prec, \;\; z = e^{\prec} = \; \lt \text{ converges}
\end{align*}
\end{mex}

\bigskip
\begin{mex}
Determine convergence or divergence of
 $\frac{1}{\mathrm{ln}_{2}\,3} + \frac{1}{\mathrm{ln}_{2}\,4} + \frac{1}{\mathrm{ln}_{2}\,5} + \ldots$ 

 $\frac{1}{\mathrm{ln}_{2}\,n} \; z \; \frac{1}{\prod_{k=0}^{w} \mathrm{ln}_{k}}|_{n=\infty}$,
 $\prod_{k=0}^{w} \mathrm{ln}_{k} \; z \; \mathrm{ln}_{2}|_{n=\infty}$,
 $\sum_{k=1}^{w+1} \mathrm{ln}_{k} \; (\mathrm{ln}\,z) \; \mathrm{ln}_{3}|_{n=\infty}$,
 $\mathrm{ln}_{1} \geq \mathrm{ln}_{3}|_{n=\infty}$,
 $\mathrm{ln}\,z = \; \geq$, $z = e^{\geq} = \; \geq$
 and
 the series diverges.
\end{mex}
\bigskip
\begin{mex}
Ramanujan gave the following $\frac{1}{\pi} = \frac{2\sqrt{2}}{9801} \sum_{k=0}^{\infty} \frac{(4k)!(1103+26390k)}{(k!)^{4} 396^{4k}}$.
 Although the following does not calculate the sum,
 by comparing against the boundary we show the sum converges.

\begin{align*}
  \sum \frac{(4n)!(1103+26390n)}{(n!)^{4} 396^{4n}} \; z \; \sum \frac{1}{\prod_{k=0}^{w} \mathrm{ln}_{k}\,n}|_{n=\infty} & \tag{remove sum, cross multiply} \\
  (4n)!(26390n) \prod_{k=0}^{w} \mathrm{ln}_{k}\,n \; z \; (n!)^{4} 396^{4n}|_{n=\infty} & \tag{undo multiplication} \\
  \mathrm{ln}(4n)! + \mathrm{ln}\,n + \sum_{k=1}^{w+1}\mathrm{ln}_{k}\,n \; (\mathrm{ln}\,z) \; 4\,\mathrm{ln}(n!) + 4n \,\mathrm{ln}\,396|_{n=\infty} & \tag{$w=0$ largest left side} \\
  \sum_{k=1}^{4n} \mathrm{ln}\,k + 2\,\mathrm{ln}\,n \; (\mathrm{ln}\,z) \; 4 \sum_{k=1}^{n}\mathrm{ln}\,k + 4n \,\mathrm{ln}\,396|_{n=\infty} & \tag{$\int \mathrm{ln}\,n\,dn = n\,\mathrm{ln}\,n|_{n=\infty}$} \\ 
  4n\,\mathrm{ln}\,4n + 2\,\mathrm{ln}\,n \; (\mathrm{ln}\,z) \; 4 n\,\mathrm{ln}\,n + 4n \,\mathrm{ln}\,396|_{n=\infty} \\ 
  4n\,\mathrm{ln}\,n + \mathrm{ln}\,4\cdot 4 n+ 2\,\mathrm{ln}\,n \; (\mathrm{ln}\,z) \; 4 n\,\mathrm{ln}\,n + 4n \,\mathrm{ln}\,396|_{n=\infty} \\
  \mathrm{ln}\,4\cdot 4 n+ \; (\mathrm{ln}\,z) \; 4 \cdot \mathrm{ln}\,396 \cdot n|_{n=\infty} \\ 
 (\mathrm{ln}\,z) = \; \lt, \; z = \; \lt & \tag{sum converges}
\end{align*}
\end{mex}
\subsection{Convergence tests}\label{S1805}
The boundary test can be used to prove other convergence 
 tests. In particular the generalized ratio test (Section \ref{S17}):
 which uses the boundary test to
 prove the ratio test,
 Raabe's test, Bertrand's test,
 as a consequence
 of proving the more general test, the generalized ratio test. 

These other tests require a rearrangement of the sequence
 at infinity, and hence is a discussion for another paper,
 where the sequence will be rearranged before input into the general
 boundary test.
\bigskip
\begin{theo}\label{P234}
  nth root convergence test:
If $|a_{n}|^{\frac{1}{n}}|_{n=\infty} \lt 1$ then
 $\sum a_{n}|_{n=\infty}=0$ converges.
\end{theo}
\begin{proof}
Comparing against the boundary
 with the known convergent
 condition, 
$\sum |a_{n}| \lt \sum \frac{1}{\prod_{k=0}^{w} \mathrm{ln_{k}}}|_{n=\infty}$
 then $\sum |a_{n}||_{n=\infty}=0$ converges. 
Removing the sum. 
$|a_{n}| \lt \frac{1}{\prod_{k=0}^{w} \mathrm{ln_{k}}}|_{n=\infty}$,  
$|a_{n}|^{\frac{1}{n}} \lt \frac{1}{ (\prod_{k=0}^{w} \mathrm{ln_{k}})^{\frac{1}{n}}}|_{n=\infty}$, 
$|a_{n}|^{\frac{1}{n}} \lt \frac{1}{ \prod_{k=0}^{w} (\mathrm{ln_{k}}^{\frac{1}{n}})}|_{n=\infty}$.  
 Since $n^{\frac{1}{n}}=(\mathrm{ln}_{0})^{\frac{1}{n}}=1$ then $(\mathrm{ln_{k}})^{\frac{1}{n}}|_{n=\infty}=1$,  
$|a_{n}|^{\frac{1}{n}} \lt 1|_{n=\infty}$ 
\end{proof}

\bigskip
\begin{theo}\label{P236}
 The power series test.
 If $|x| \lt 1$ then $\sum x^{n}|_{n=\infty}=0$ converges.
\end{theo}
\begin{proof} 
By the absolute value convergence test, consider positive $x$ only.
 Compare against the boundary.
 $\sum x^{n} \; z \; \sum \frac{1}{\prod_{k=0}^{w} \mathrm{ln}_{k}}|_{n=\infty}$, 
 $x^{n} \; z \; \frac{1}{\prod_{k=0}^{w} \mathrm{ln}_{k}}|_{n=\infty}$, 
 $x^{n} \prod_{k=0}^{w} \mathrm{ln}_{k}|_{n=\infty} \; z \; 1$, 
 $n\, \mathrm{ln}\,x + \sum_{k=1}^{w+1} \mathrm{ln}_{k}|_{n=\infty} \; (\mathrm{ln}\,z) \; 0$, 
 $(\mathrm{ln}\,z) = \; \lt$ only when $\mathrm{ln}\,x$ is negative, $0 \lt x \lt 1$. Reversing
 the process,
$n\, \mathrm{ln}\,x + \sum_{k=1}^{w+1} \mathrm{ln}_{k}|_{n=\infty} \lt 0$, 
 $x^{n} \prod_{k=0}^{w} \mathrm{ln}_{k}|_{n=\infty} \lt 1$, 
 $x^{n} \lt \frac{1}{\prod_{k=0}^{w} \mathrm{ln}_{k}}|_{n=\infty}$, 
 $\sum x^{n} \lt \sum \frac{1}{\prod_{k=0}^{w} \mathrm{ln}_{k}}|_{n=\infty}$,
 as less than the boundary the sum converges. 
\end{proof}
\bigskip
\begin{theo}\label{P018}
 If $a \succ\!\succ b$ then  
 $\sum ab|_{n=\infty} = \sum a|_{n=\infty}$ 
\end{theo}
\begin{proof}
 By definition $a \succ\!\succ b$ then $\mathrm{ln}\,a + \mathrm{ln}\,b = \mathrm{ln}\,a$.
 Consider
 $\sum ab \; z \; \sum \frac{1}{\prod_{k=0}^{w}}|_{n=\infty}$,
 $ab \prod_{k=0}^{w}\mathrm{ln}_{k} \; z \; 1|_{n=\infty}$, 
 $\mathrm{ln}\,a + \mathrm{ln}\,b + \sum_{k=1}^{w+1}\mathrm{ln}\,k \; (\mathrm{ln}\,z) \; 0|_{n=\infty}$,
 $\mathrm{ln}\,a + \sum_{k=1}^{w+1}\mathrm{ln}\,k \; (\mathrm{ln}\,z) \; 0|_{n=\infty}$, reverse the process, 
 $\mathrm{ln}(a \prod_{k=0}^{w} \mathrm{ln}_{k}) \; (\mathrm{ln}\,z) \; \mathrm{ln}\,1|_{n=\infty}$,
 $a \prod_{k=0}^{w} \mathrm{ln}_{k} \; z \; 1|_{n=\infty}$,
 $\sum a \; z \; \sum \frac{1}{\prod_{k=0}^{w}\mathrm{ln}_{k}}|_{n=\infty}$.
\end{proof} 
\bigskip
\begin{mex}
 $\sum \frac{n^{\frac{1}{n}}}{n^{2}}|_{n=\infty}$
 We can immediately observe $n^{\frac{1}{n}} \prec\!\prec n^{2}|_{n=\infty}$
 if we recognise $n^{\frac{1}{n}}|_{n=\infty}$
 as a constant. Then apply the theorem
 $\sum \frac{n^{\frac{1}{n}}}{n^{2}}|_{n=\infty}$
 $=\sum \frac{1}{n^{2}}|_{n=\infty}=0$ converges.
 Alternatively,  
 consider $n^{\frac{1}{n}} \; z \; n^{2}|_{n=\infty}$,
 $\frac{1}{n} \mathrm{ln}\,n \; (\mathrm{ln}\,z) \; 2 \mathrm{ln}\,n|_{n=\infty}$,
 $\frac{1}{n} \; (\mathrm{ln}\,z) \; 2|_{n=\infty}$,
 $\frac{1}{n} \prec 2|_{n=\infty}$, then
 $n^{\frac{1}{n}} \prec\!\prec n^{2}|_{n=\infty}$.
\end{mex}

 The following are references to the boundary test being
 applied to solve convergence tests in other papers.

\begin{itemize}
  \item Alternating convergence test (see \ref{S0710}, Theorem \ref{P206})
  \item Generalized ratio test (Section \ref{S17}) 
\end{itemize}
\subsubsection{L'Hopital's convergence test}\label{S180501}
 The following is L'Hopital's rule in a weaker form.
 For example, the ratio of two polynomials
 equally reduces the numerator and denominator.
\bigskip
\begin{prop}\label{P235}
If $\frac{f}{g} = \frac{f'}{g'}|_{n=\infty}$, where $f$ and $g$
 are in indeterminate form $\infty/\infty$ or $0/0$ then
 $\sum \frac{f}{g} = \sum \frac{f'}{g'}|_{n=\infty}$
\end{prop}
\begin{proof}
 $\sum \frac{f}{g} \; z \; \sum \frac{1}{\prod_{k=0}^{w} \mathrm{ln}_{k}}|_{n=\infty}$, 
 $\frac{f}{g} \; z \; \frac{1}{\prod_{k=0}^{w} \mathrm{ln}_{k}}|_{n=\infty}$, 
 $\frac{f'}{g'} \; z \; \frac{1}{\prod_{k=0}^{w} \mathrm{ln}_{k}}|_{n=\infty}$, 
 $\sum \frac{f'}{g'} \; z \; \sum \frac{1}{\prod_{k=0}^{w} \mathrm{ln}_{k}}|_{n=\infty}$.
 Since both sums have the
 same relation then
$\sum \frac{f}{g} = \sum \frac{f'}{g'}|_{n=\infty}$ 
\end{proof}

 A counter example
 for the general L'Hopital was given \cite{hopital2},
 with convergence sums
 $\sum \frac{\mathrm{ln}_{2}\,n}{n\,\mathrm{ln}\,n}|_{n=\infty}$ diverges, $f = \mathrm{ln}_{2}$, $g_{n} = n\,\mathrm{ln}\,n$ but $\sum \frac{f'}{g'} = \sum \frac{1}{n \,(\mathrm{ln}\,n)^{2}}|_{n=\infty}=0$ converges.
 We can generalize the counter example.
\bigskip
\begin{prop}
$w \geq 1$,
$f = \mathrm{ln}_{w+1}$,
 $g = \prod_{k=0}^{w}\mathrm{ln_{k}}$;
 $\sum \frac{f}{g}|_{n=\infty}=\infty$ diverges,
 but $\sum \frac{f'}{g'}|_{n=\infty}=0$ converges. 
\end{prop}
\begin{proof}
 $\sum \frac{f}{g} = \sum \frac {1}{g}|_{n=\infty}$ 
 $= \infty$ diverges,
 as $g \succ\!\succ f$.

 $f' = \frac{1}{\prod_{k=0}^{w}\mathrm{ln}_{k}}$,
 $g' = \prod_{k=1}^{w}\mathrm{ln}_{k}$;
$\sum \frac{f'}{g'}|_{n=\infty}$
 $= \sum \frac{1}{\prod_{k=0}^{w}\mathrm{ln}_{k}} \frac{1}{\prod_{k=1}^{w}\mathrm{ln}_{k}}|_{n=\infty}$
 $\leq \sum \frac{1}{n\,(\mathrm{ln}\,n)^{2}}=0$ converges.
\end{proof}

 Taking the derivative in the counter example
 case has the numerator $f$ interacting with
 the denominator $g$ when there is no interaction.
 Since the boundary is the `suspect', lets
 exclude this case.
\bigskip
\begin{conjecture}\label{P245}
 When $\frac{f}{g} \neq \frac{\mathrm{ln}_{w+1}}{ \prod_{k=0}^{w}\mathrm{ln_{k}}}|_{n=\infty}$ then 
 $\sum \frac{f}{g} = \sum \frac{f'}{g'}|_{n=\infty}$ 
\end{conjecture}
\subsection{Representing convergent/divergent series}\label{S1806}
 As the boundary test
 is complete (Theorem \ref{P232}), 
 since the boundary separates between convergence
 and divergence, for a given series,
 a lower bound for divergent 
 series and an upper bound for convergent
 series exists.

 Although this may seem a trivial rearrangement of the boundary test,
 we show how we may isolate and describe classes of
 convergence and divergence from the original test.

 For example, 
 the idea of a class of series that is convergent to the boundary,
 but is also convergent may seem counter intuitive.
\bigskip
\begin{mex} 
 $\frac{1}{n^{2}}$ is bounded above, $\frac{1}{n^{2}} \lt \frac{1}{n^{p}}$ when $1 \lt p \lt 2$, which is known to converge. 
 Applying a comparison, 
 $\frac{1}{n^{2}} \lt \frac{1}{n^{p}}|_{n=\infty}$,
 $\sum \frac{1}{n^{2}} \lt \sum \frac{1}{n^{p}}|_{n=\infty}$,
 $\sum \frac{1}{n^{2}} \leq 0$,
 $\sum \frac{1}{n^{2}} =0$.
\end{mex}
\bigskip
\begin{mex}
$\frac{1}{n \, (\mathrm{ln}\,n)^{\frac{1}{2}}}$
 is bounded below, $\frac{1}{n\, \mathrm{ln}\,n} \lt \frac{1}{n \, (\mathrm{ln}\,n)^{\frac{1}{2}}} \lt \frac{1}{n}$ then $\frac{1}{n\, \mathrm{ln}\,n}$ is the first discrete lower bound
 of the boundary, $w=1$.
 Applying a comparison,
 $\frac{1}{n\, \mathrm{ln}\,n} \lt \frac{1}{n \, (\mathrm{ln}\,n)^{\frac{1}{2}}}|_{n=\infty}$,
 $\sum \frac{1}{n\, \mathrm{ln}\,n} \lt \sum \frac{1}{n \, (\mathrm{ln}\,n)^{\frac{1}{2}}}|_{n=\infty}$,
 $\infty \leq \sum \frac{1}{n \, (\mathrm{ln}\,n)^{\frac{1}{2}}}|_{n=\infty}$,
 $\sum \frac{1}{n \, (\mathrm{ln}\,n)^{\frac{1}{2}}}|_{n=\infty} = \infty$
 diverges.
\end{mex}
\bigskip
\begin{remk}
 All positive monotonic divergent series are bounded by the
 boundary below.
\end{remk}
\bigskip
\begin{lem}\label{P237}
 For monotonic and diverging series
 $\sum a_{n}|_{n=\infty} = \infty$,
 for some $w$,
 \[ \sum \frac{1}{\prod_{k=0}^{w} \mathrm{ln}_{k}}|_{n=\infty} \leq \sum a_{n}|_{n=\infty} \]
\end{lem}
\begin{proof}
 Boundary test Theorem \ref{P232}: swap sides and inequality direction for the divergent case.
\end{proof}

For convergent series $\sum a_{n}|_{n=\infty}=0$ 
 either the series is asymptotic
 to `below the boundary' or below
 the boundary$(a_{n} \lt \frac{1}{n})|_{n=\infty}$. 
\bigskip
\begin{remk}
  All positive monotonic convergent series
 are bounded above by a series less than the boundary. 
\end{remk}
\bigskip
\begin{lem}\label{P238}
When 
 $\sum a_{n}|_{n=\infty} = 0$ converges then for some $w$
 and $p \gt 1$,
\[
\sum a_{n}|_{n=\infty} \leq \sum \frac{1}{\prod_{k=0}^{w} \mathrm{ln}_{k} \cdot \mathrm{ln}_{w+1}^{p} }|_{n=\infty} \]
\end{lem}
\begin{proof}
 Boundary test Theorem \ref{P232}: convergent case.
\end{proof}

 These bounds are available as another tool.  
 Applying the lower and upper bound at the boundary to
 derive the following theorems.
\bigskip
\begin{remk}
 There always exists a series that converges more slowly.
\end{remk}
\bigskip
\begin{theo}\label{P239}
If $\sum_{n=1}^{\infty} a_{n}$ is a convergent
 series with positive terms 
 then there exists a monotonic sequence $( b_{n} )_{n=1}^{\infty}$
 such that $\lim\limits_{n \to \infty} b_{n} = \infty$
 and series $\sum_{n=1}^{\infty} a_{n} b_{n}$ converges.
\end{theo}
\begin{proof}
 Reform as the convergence sum, solve by Theorem \ref{P240},
 then 
 apply the transfer principle,
 transferring a convergence sum to a sum 
 (See Theorem \ref{P090}).
\end{proof}
\bigskip
\begin{theo}\label{P240}
 If $\sum a_{n}|_{n=\infty}=0$ is a convergent series
 with positive terms then there exists a monotonic sequence
 $( b_{n} )|_{n=\infty}$
 such that $b_{n}|_{n=\infty}=\infty$ 
 and
 $\sum a_{n} b_{n}|_{n=\infty}=0$ converges.
\end{theo}
\begin{proof}
 Since $\sum a_{n}|_{n=\infty}=0$
 then
 by Lemma \ref{P238},
 $\exists c, p \gt 1$: 
 $a_{n} \leq \frac{1}{L_{c-1} (\mathrm{ln}_{c})^{p} }$.  
 For positive $b_{n}$, 
 $a_{n} b_{n} \leq \frac{1}{L_{c-1} (\mathrm{ln}_{c})^{p} } b_{n}$,
 $(\sum a_{n} b_{n} \leq \sum \frac{1}{L_{c-1} (\mathrm{ln}_{c})^{p} } b_{n})|_{n=\infty}$.

 Compare  $\sum \frac{1}{L_{c-1} (\mathrm{ln}_{c})^{p} } b_{n}|_{n=\infty}$
 against the boundary.
 $(\sum \frac{1}{L_{c-1} (\mathrm{ln}_{c})^{p} } b_{n} \; z \; \sum \frac{1}{L_{w}})|_{n=\infty}$, 
 $(\frac{1}{L_{c-1} (\mathrm{ln}_{c})^{p} } b_{n} \; z \; \frac{1}{L_{w}})|_{n=\infty}$,
 $(L_{w} b_{n} \;\; z \;\; L_{c-1} (\mathrm{ln}_{c})^{p})|_{n=\infty}$,  
 $(\mathrm{ln}(L_{w} b_{n}) \; (\mathrm{ln}\, z) \; \mathrm{ln}(L_{c-1} (\mathrm{ln}_{c})^{p}))|_{n=\infty}$,  
 $(\sum_{k=1}^{w+1} \mathrm{ln}_{k} + \mathrm{ln}\,b_{n} \;\; (\mathrm{ln}\,z) \;\; \sum_{k=1}^{c} \mathrm{ln}_{k} + p \, \mathrm{ln}_{c+1})|_{n=\infty}$.

Choose $b_{n} = \mathrm{ln}_{w+1}$, then $\sum_{k=1}^{w+1} \mathrm{ln}_{k} + \mathrm{ln}\,b_{n}  = \sum_{k=1}^{w+1} \mathrm{ln}_{k}|_{n=\infty}$ 
 as $\mathrm{ln}_{w+1} \succ \mathrm{ln}_{w+2}|_{n=\infty}$.

 Reversing the process, 
 $(\sum_{k=1}^{w+1} \mathrm{ln}_{k} \; (\mathrm{ln}\,z) \; \sum_{k=1}^{c} \mathrm{ln}_{k} + p \, \mathrm{ln}_{c+1})|_{n=\infty}$, 
 $(\mathrm{ln}(L_{w}) \; (\mathrm{ln}\, z) \; \mathrm{ln}(L_{c-1} (\mathrm{ln}_{c})^{p})) |_{n=\infty}$, 
 $(L_{w} \; z \; L_{c-1} (\mathrm{ln}_{c})^{p}) |_{n=\infty}$,  
 $(\frac{1}{L_{c-1} (\mathrm{ln}_{c})^{p} } \; z \; \frac{1}{L_{w}})|_{n=\infty}$,  
 $(\sum \frac{1}{L_{c-1} (\mathrm{ln}_{c})^{p} } \; z \; \sum \frac{1}{L_{w}})|_{n=\infty}$.  

Since the left hand side sum 
 $\sum \frac{1}{L_{c-1} (\mathrm{ln}_{c})^{p} }$
is always convergent, and the right hand sum
 $\sum \frac{1}{L_{w}}$
 is always
 divergent, $z = \; \lt$.
 $\sum \frac{1}{L_{c-1} (\mathrm{ln}_{c})^{p} } b_{n}|_{n=\infty}=0$ then
 $\sum a_{n}b_{n}|_{n=\infty}=0$ converges. 
\end{proof}
\bigskip
\begin{remk}
 There always exists a series that diverges more slowly.
\end{remk}
\bigskip
\begin{theo}\label{P241}
 If $\sum_{n=1}^{\infty} a_{n}$ is a divergent
 series with positive terms then there exists a monotonic sequence
 $(b_{n})_{n=1}^{\infty}$ such that $\lim\limits_{n \to \infty} b_{n}=0$ and
 the series $\sum_{n=1}^{\infty} a_{n} b_{n}$ diverges. 
\end{theo}
\begin{proof}
 Reform as the convergence sum, solve by Theorem \ref{P242},
 then 
 apply the transfer principle,
 transferring a convergence sum to a sum
 (See Theorem \ref{P090}).
\end{proof}
\bigskip
\begin{theo}\label{P242}
 If $\sum a_{n}|_{n=\infty}=\infty$ is a divergent
 series with positive terms then there exists a monotonic sequence
 $(b_{n})_{n=\infty}$ such that $b_{n}|_{n=\infty}=0$ 
 and 
 $\sum a_{n} b_{n}|_{n=\infty} = \infty$ diverges.
\end{theo}
\begin{proof}
Since $\sum a_{n}$ is diverging there exists a boundary sequence
 that acts as a lower bound.
 $\exists w: \frac{1}{L_{w}} \leq a_{n}|_{n=\infty}$,
 $\frac{1}{L_{w}} b_{n} \leq a_{n} b_{n} |_{n=\infty}$,
 let $b_{n} = \frac{1}{\mathrm{ln}_{w+1}}$,  
 $\frac{1}{L_{w} \mathrm{ln}_{w+1}} \leq a_{n} b_{n} |_{n=\infty}$,
 $\frac{1}{L_{w+1} }  \leq a_{n} b_{n} |_{n=\infty}$,
 $\sum \frac{1}{L_{w+1} }  \leq  \sum a_{n} b_{n} |_{n=\infty}$,
 $\infty  \leq  \sum a_{n} b_{n} |_{n=\infty}$,
 then $\sum a_{n}b_{n}|_{n=\infty}=\infty$ diverges.
\end{proof}

{\em RMIT University, GPO Box 2467V, Melbourne, Victoria 3001, Australia}\\
{\em chelton.evans@rmit.edu.au}

\begin{thebibliography}{99}
\bibitem{gold} L. Bibiloni, P. Viader, and J. Paradis, \emph{On a Series of Goldbach and Euler}, The Mathematical Association of America Monthly 113, March 2006.
\bibitem{cebp21} Chelton D. Evans, William K. Pattinson, \emph{Extending du Bois-Reymond's Infinitesimal and Infinitary Calculus Theory}
\bibitem{cebp10} Chelton D. Evans and William K. Pattinson, \emph{The Fundamental Theorem of Calculus with Gossamer numbers}
\bibitem{apostol} T. M. Apostol, \emph{Calculus}, Vol. 1, Second Edition, Xerox College Publishing.
\bibitem{ordersofinfinity} G. H. Hardy, \emph{Orders of infinity}, Cambridge Univ. Press, 1910. 
\bibitem{eulerho}  L. Euler, \emph{De progressionibus harmonicis observationes}, Opera Omnia, ser. I, vol. 14, Teubner, Leipzig, 1925, pp.87--100. 
\bibitem{kaczor} W.J. Kaczor, M. T. Nowak, \emph{Problems in Mathematical Analysis I}, AMS, Providence , 2000. 
\bibitem{kreyszig} E. Kreyszig, \emph{Advanced Engineering Mathematics}, 7th edition, John Wiley \& Sons.
\bibitem{kaczor2} W.J.Kaczor and M.T. Norwak, \emph{Problems in Mathematical Analysis II Continuity and Differentiation}, AMS, Providence, 2000 
\bibitem{laugwitz} D. Laugwitz, \emph{Theory of Infinitesimals. An Introduction to Nonstandard Analysis}, Accademia Nazionale dei Lincei, Roma, 1980
\bibitem{bertrand} From MathWorld--A Wolfram Web Resource, \emph{Bertrand's Test}, available at \url{http://mathworld.wolfram.com/BertrandsTest.html}
\bibitem{knopp} Konrad Knopp, \emph{Infinite sequences and series} (F. Bagemihl Trans.), New York Dover Publications, 1956
\bibitem{stringtheory} \emph{Is the sum of all natural numbers -1/12}, \url{http://math.stackexchange.com/questions/633285/is-the-sum-of-all-natural-numbers-frac112}
\bibitem{ramanujanpi} \emph{Ramanujan's approximation for $\pi$}, \url{http://math.stackexchange.com/questions/908535/ramanujans-approximation-for-pi}
\bibitem{demidovich} G Varanenkov, et al, \emph{Problems in Mathematical Analysis}(B. Demidovich editorship)(G. Yankovsky Trans.), Mir 1976, Seventh printing 1989.
\bibitem{victors} J. Bair P. Blaszczyk, R. Ely, V. Henry, V. Kanoevei, K. Katz, M. Katz, S. Kutateladze, T. Mcgaffy, D. Schaps, D. Sherry and S. Shnider, \emph{IS MATHEMATICAL HISTORY WRITTEN BY THE VICTORS?},
 see http://arxiv.org/abs/1306.5973
\bibitem{galanor} G. Stewart, \emph{Riemann's Rearrangement Theorem}, Mathematics Teacher , v80 n8 pp.675--681 Nov 1987.
\bibitem{cajori} F. Cajori, \emph{Evolution of criteria of convergence}, Bull. Amer. Math. Soc. 2 1892.
\bibitem{hart} F. Hartmann and D. Sparows, \emph{Investigating Possible Boundaries Between Convergence and Divergence}, available at \url{http://www42.homepage.villanova.edu/frederick.hartmann/Boundaries/Boundaries.pdf} 
\bibitem{fisher} G. Fisher, \emph{The infinite and infinitesimal quantities of du Bois-Reymond and their reception} , 27. VIII. 1981, Volume 24, Issue 2, pp 101--163, Archive for History of Exact Sciences
\bibitem{hopital2} StackExchange, \emph{L'Hopital's rule and series convergence}, available at \url{http://math.stackexchange.com/questions/77024/lhopitals-rule-and-series-convergence} 
\end{thebibliography}
\end{document}